\documentclass[10pt]{amsart}

\usepackage{graphicx, color}
\usepackage{psfrag}

\newtheorem{thm}{Theorem}[section]
\newtheorem{ass}[thm]{Assumption}
\newtheorem{coro}[thm]{Corollary}
\newtheorem{lem}[thm]{Lemma}
\newtheorem{prop}[thm]{Proposition}

\theoremstyle{definition}

\newtheorem{defn}[thm]{Definition}
\newtheorem{eg}[thm]{Example}

\theoremstyle{remark}

\newtheorem{remk}[thm]{Remark}

\newcommand{\Hcal}{ {\mathcal H}}
\newcommand{\Mcal}{ {\mathcal M}}
\newcommand{\Dcal}{ {\mathcal D}}
\newcommand{\Ucal}{ {\mathcal U}}
\newcommand{\Vcal}{ {\mathcal V}}
\newcommand{\Ecal}{ {\mathcal E}}
\newcommand{\Acal}{ {\mathcal A}}

\newcommand{\tensor}{\otimes}

\newcommand{\eps}{\varepsilon}

\newcommand{\spec}{{\rm spec}}
\newcommand{\dom}{{\rm dom}}

\newcommand{\vphi}{\varphi}

\newcommand{\R}{{\mathbf R}}
\newcommand{\Z}{{\mathbf Z}}
\newcommand{\N}{{\mathbf N}}

\newcommand{\Cc}{{\mathcal C}_0}

\newcommand{\refeq}[1]{(\ref{#1})}
\newcommand{\und}{\frac{1}{2}}
\newcommand{\dt}{\frac{2}{3}}
\newcommand{\qt}{\frac{4}{3}}
\newcommand{\unt}{\frac{1}{3}}

\newcommand{\td}{\frac{3}{2}}

\begin{document}

\title[Spectral simplicity]{Spectral simplicity \\ and asymptotic
separation of variables}

\author{Luc Hillairet}

\author{Chris Judge}
\thanks{
L.H. would like to thank the Indiana University 
for its invitation and hospitality 
and the ANR programs `Teichm\"uller' and `R\'esonances et chaos quantiques' 
for their support.
C.J. thanks the Universit\'e de Nantes, MATPYL 
program, L'Institut Fourier, and the Max Planck Institut 
f\"ur Mathematik-Bonn  for hospitality and support.}

\begin{abstract}
We describe a method for comparing the real 
analytic eigenbranches of two families, $(a_t)$ and $(q_t)$, of 
quadratic forms that degenerate as $t$ tends to zero. 
We consider families $(a_t)$ amenable to `separation of variables'
and show that if $(q_t)$ is asymptotic to $(a_t)$  at first order as $t$ tends to $0$,
then the generic spectral simplicity of $(a_t)$ implies 
that the eigenbranches of $(q_t)$ are also generically one-dimensional. 
As an application, we prove that for the generic triangle (simplex) 
in  Euclidean space (constant curvature space form) 
each eigenspace of the Laplacian is one-dimensional.   
\end{abstract}

\maketitle

\section{Introduction}

In this paper we continue a study of generic spectral simplicity that began 
with \cite{HlrJdg09} and \cite{Slit}. In particular, we 
develop a method that allows us to prove the following. 

\begin{thm}\label{ThmIntroTriDDD}
For almost every Euclidean triangle $T \subset \R^2$, each eigenspace of the 
Dirichlet Laplacian associated to $T$ is one-dimensional.
\end{thm}

Although we establish the existence of triangles with simple Laplace spectrum,
we do not know the exact geometry of a single triangle that has simple spectrum. Up
to homothety and isometry, there are only two Euclidean triangles whose Laplace
spectrum has been explicitly computed, the equilateral triangle and the right is-
coceles triangle, and in both these cases the Laplace spectrum has multiplicities
\cite{Lame} \cite{PinskyTriangle} \cite{Berard} \cite{Harmer}. 
Numerical results indicate that other triangles might have spectra 
with multiplicities [BryWlk84]. Note that non-isometric triangles  
have different spectra \cite{Durso} \cite{LucDiffractive}.

More generally, we prove that almost every simplex in Euclidean space 
has simple Laplace spectrum. Our method applies to other settings as well. 
For example, we have the following.

\begin{thm}  \label{alpha}
 For all but countably many $\alpha$, each eigenspace of the 
 Dirichlet Laplacian associated to the geodesic triangle $T_{\alpha}$
 in the hyperbolic plane with angles $0$, $\alpha$, and $\alpha$, 
 is one-dimensional. 
\end{thm}

If $\alpha= \pi/3$, then $T_{\alpha}$ is isometric to 
a fundamental domain for the group $SL_2(\Z)$ acting on the 
upper half-plane as linear fractional transformations.  
P. Cartier \cite{Cartier} conjectured that $T_{\pi/3}$ 
has simple spectrum. This conjecture remains open (see \cite{Sarnak}).

Until now, the only extant methods for proving that domain had simple
Laplace spectrum consisted of either explicit computation of the spectrum,
a perturbation of a sufficiently well-understood domain, 
or a perturbation within an infinite dimensional space of domains. 
As an example of the first approach, using separation of variables 
one can compute the Laplace spectrum of each rectangle exactly and find
that this spectrum is simple iff the ratio of the squares of the sidelengths 
is not a rational number. In \cite{HlrJdg09} we used this fact and the analytic 
perturbation method to show that almost every polygon with at least four sides has simple
spectrum. The method for proving spectral simplicity by making perturbations
in an infinite dimensional space originates with J. Albert \cite{Albert} 
and K. Uhlenbeck  
\cite{Uhlenbeck}. In particular, it is shown in \cite{Uhlenbeck} that  
the generic compact domain with smooth boundary has simple spectrum.

In the case of Euclidean triangles, the last method does not apply since the 
space of triangles is finite dimensional. We also do not know how to compute 
the Laplace spectrum of a triangle other than the right-isoceles and equilateral
ones. One does know the eigenfunctions of these two triangles sufficiently well
to apply the perturbation method, but unfortunately the eigenvalues
do not split at first order and it is not clear to us what happens at second order.

As a first step towards describing our approach, we consider the following example. 
Let $T_t$ be the family of Euclidean right triangles with vertices $(0,0)$, $(1,0)$,
and $(1,t)$ and let $q_t$ denote the associated Dirichlet energy form 
\[ q_t(u)~ =~ \int_{T_t}  \|\nabla u\|^2~ dx~ dy. \]
For each $u,v \in C^{\infty}_0(T_t)$, we have $q_t(u,v)=\langle \Delta_t u, v\rangle$ 
where $\Delta_t$ is the Laplacian, and hence the spectrum of $\Delta_t$ equals 
the spectrum of $q_t$ with respect to the $L^2$-inner product on $T_t$.

As $t$ tends to zero, the triangle $T_t$ degenerates to the segment
that joins $(0,0)$ and $(1,0)$.  Note that the spectrum of an interval is simple
and hence one can hope to use this to show that $T_t$ has simple spectrum for 
some small $t>0$. 

\begin{figure}[h]   
\begin{center}


\includegraphics[totalheight=1.5in]{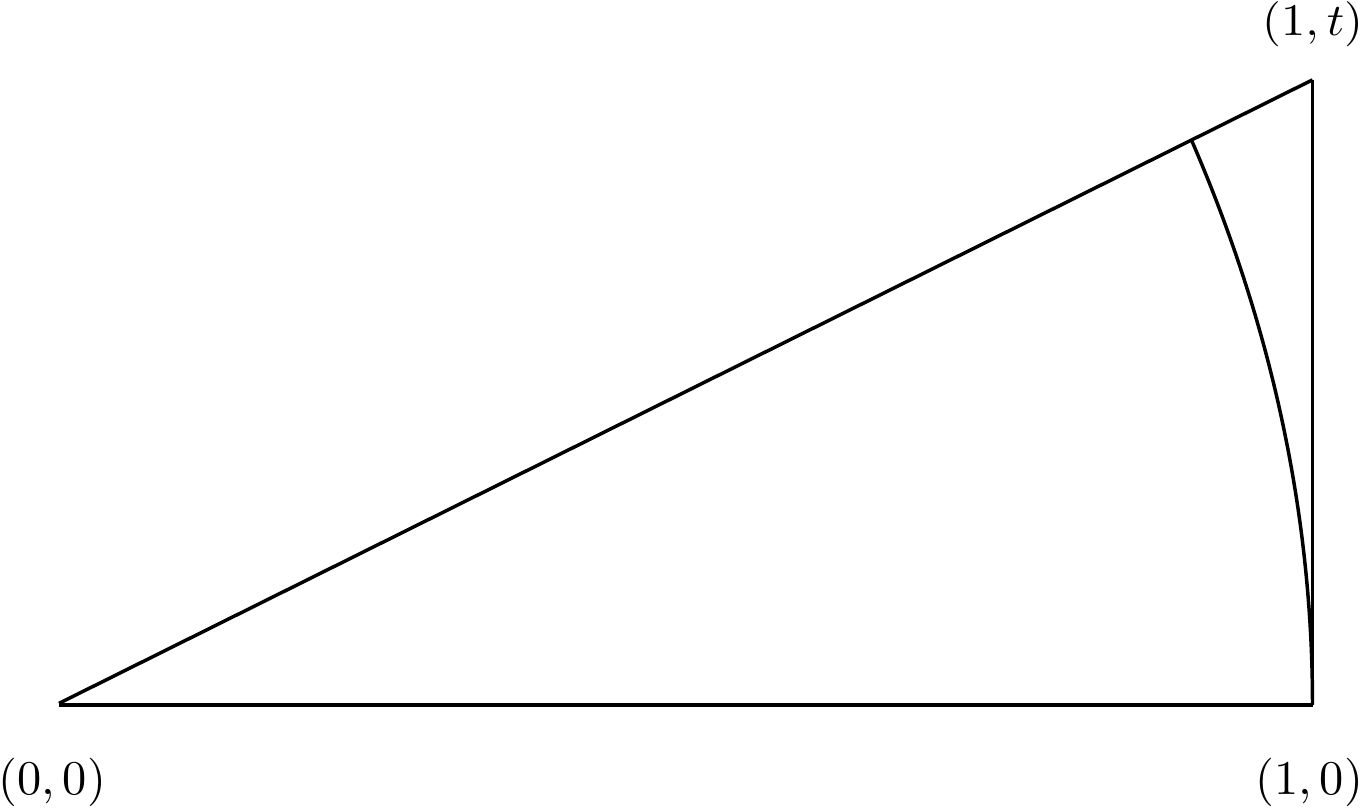}

\end{center}
\caption{\label{TriSect}The triangle $T_t$ and the sector $S_t$.}
\end{figure}

Indeed, the spectral study of domains that degenerate to a one-dimensional object 
is quite well developed. In particular, it is known that---up to renormalization---the 
spectrum of the ordered eigenvalues of the domain converge to eigenvalues of the 
limiting object (see, for example, \cite{Post} and  \cite{Friedlander}). 
Using these kinds of results it is quite easy to prove that 
for each $n \in \N$, there exists $t_n>0$ so that the the first $n$ 
eigenvalues of $T_{t_n}$ are simple.

Unfortunately, this does not imply the existence of a triangle 
{\em all} of whose eigenvalues are simple. This subtle point is 
perhaps best illustrated by a different example whose spectrum can be explicitly
calculated: Let $C_t$ be the cylinder $[0,1] \times \R/{t\Z}$. 
The spectrum of the Dirichlet Laplacian on $C_t$ is 
\[ \left\{\pi^2 \cdot \left( k^2+\ell^2/t^2\right)~ |~ 
        (k, \ell) \in \N \times (\N \cup \{0\}) \right\}.  \]
Moreover, for each $t>0$ and $(k, \ell) \in  \N \times \N$, 
each eigenspace is 2-dimensional. On the other hand,  the first 
$n$ eigenvalues of the cylinder $C_{t}$ are simple iff $t < (n^{2}-1)^{-1}$. 

The example indicates that the degeneration approach to proving
spectral simplicity does not work at the `zeroeth order' approximation. 
The method that we describe here is at the next order.  
For example, in the case of the degenerating triangles $T_t$, there is  
a second quadratic form $a_t$ to which $q_t$ is asymptotic 
in the sense that $a_t-q_t$ is of the same order of magnitude as $t \cdot a_t$. 
Geometrically, the quadratic form $a_t$ corresponds to the Dirichlet energy form 
on the sector, $S_t$, of the unit disc with angle $\arctan(t)$. See Figure \ref{TriSect}. 

The spectrum of $a_t$ may be analyzed using polar coordinates and separation of variables.
In particular, we obtain the Dirichlet quadratic form $b$ associated to the interval 
of angles $[0, \arctan(t)]$, and, asssociated to each eigenvalue 
$(\ell\cdot \pi/\arctan(t))^2$ of $b$, we have a quadratic form $a_t^{\ell}$
on the radial interval $[0,1]$.  Each eigenfunction of $a_t^{\ell}$ is of the form 
$r \mapsto J_{\nu}(\sqrt{\lambda} \cdot r)$ where $J_{\nu}$ is a 
Bessel function of order $\nu=2 \pi/\arctan(t)$ and where the eigenvalue, $\lambda$, is 
determined by the condition that this function vanish at $r=1$. 
The spectrum of $a_t$ is the union of the spectra of $a_t^{\ell}$ over $\ell \in \N$.

\begin{figure}[h]   
\begin{center}


\includegraphics[totalheight=2.5in]{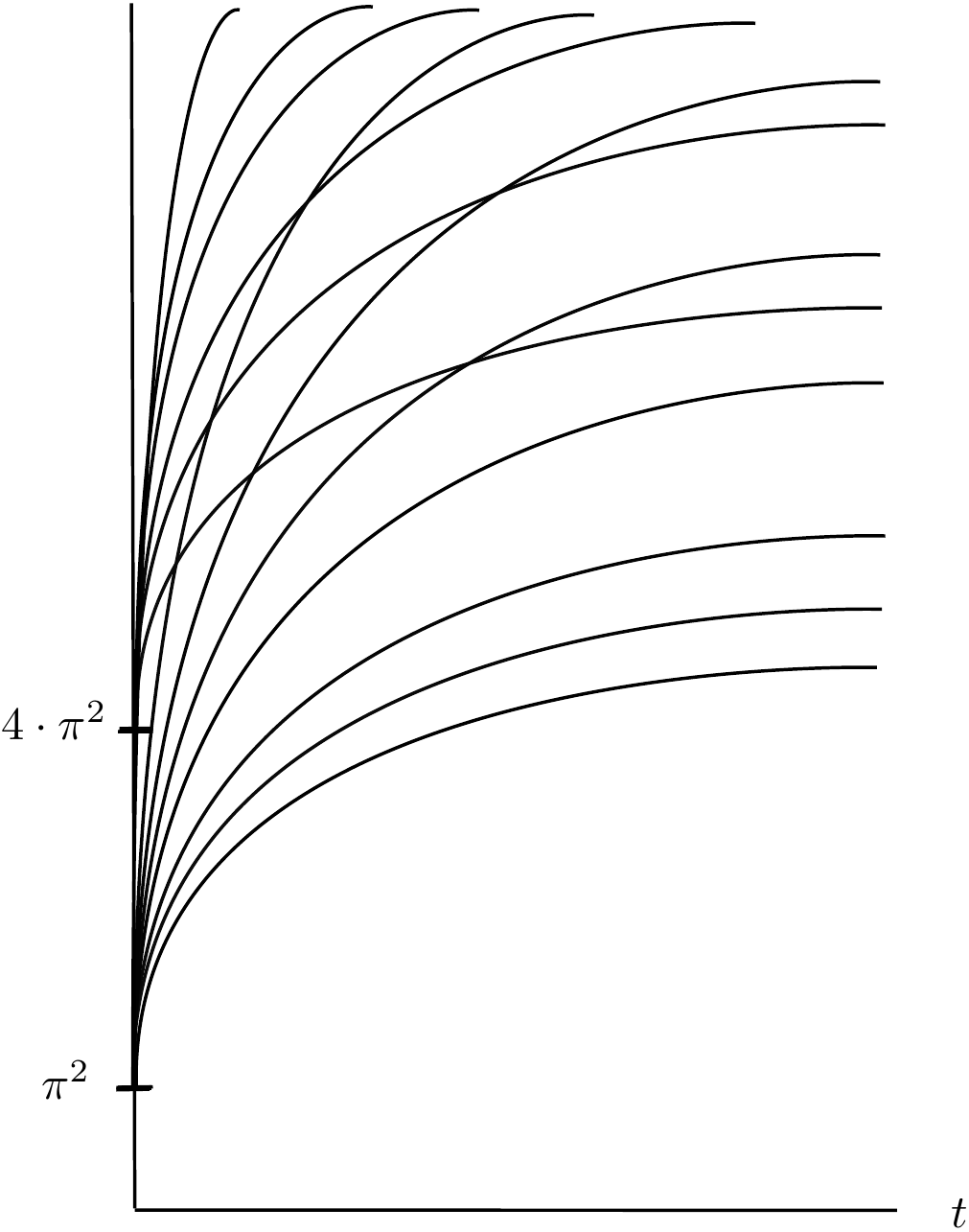}

\end{center}
\caption{\label{atspec}The spectrum of the family $a_t$}
\end{figure}

In order to better describe the qualitative 
features of the spectrum of $a_t$ as $t$ tends to zero, 
we renormalize by multiplying by $t^2$. 
Figure \ref{atspec} 
presents the main qualitative features of the renormalized spectrum of $a_t$.
For each $\ell \in \N$, the (renormalized) real-analytic eigenvalue branches of $a$ coming 
from $a_t^{\ell}$ converge to the threshold  $(\ell \cdot\pi)^2$.  The 
eigenvalues of $a_t^{\ell}$ are simple for all $t$, and for all but countably 
many $t$, the spectrum of $a_t$ is simple. 
 
From the asymptotics of the zeroes of the Bessel function, one can show that  
the distance between any two (renormalized) real-analytic 
eigenbranches of $a_t^{\ell}$ is of order at least $t^\frac{2}{3}$. 
Note that this `super-separation' of eigenvalues can have strong implications
for the simplicity problem and, in fact, is central to our method. 
Indeed, simplicity would follow if one were to prove 
that each real-analytic eigenvalue branch of $q_t$ lies in an $O(t)$ neighborhood
of a real-analytic eigenvalue branch of $a_t$ and that at most one 
eigenfunction branch of $q_t$ has its eigenvalue branch lying in this neighborhood.

In fact, as sets, the distance between the spectrum of $a_t$ and the spectrum 
of $q_t$ is $O(t)$, and, with some work, one can prove that each (renormalized)
real-analytic eigenvalue branch of $q_t$ converges to a threshold in 
$\{(\ell \cdot \pi)^2~ |~ \ell \in \N \}$ (Theorem \ref{qLimit}).
Nonetheless,  infinitely many real-analytic eigenbranches 
of $a_t$ converge to each threshold and the crossing pattern 
of these branches and the branches of $q_t$ can be quite complicated. 
Semiclassical analysis predicts\footnote{In semiclassical terminology, 
$t$ corresponds to $h$ and, away for $(\ell \cdot\pi)^2,$ the energy is non-critical.  
Hence one would expect eigenvalues to be separated at order $t$.} that the eigenvalues of 
$a_t^{\ell'}$ become separated at order $t$ away from the threshold $(\ell \cdot \pi)^2$.
On the other hand, two real-analytic eigenbranches that converge to the same threshold
stay separated at order $t^{\dt}$.  In order to use the super-separation of 
eigenvalues, we will need to show that each eigenvector 
branch of $q_t$ whose eigenvalue branch converges to a particular threshold 
does not interact with eigenvector branches of $a_t$ that converge 
to another threshold (see Lemmas 12.3 and 12.4). In this sense, 
we will asymptotically separate variables.
 
One somewhat novel feature of this work is the melding of  
techniques from semiclassical analysis
and techniques from analytic perturbation theory. 
We apply quasimode and concentration estimates to make comparative 
estimates of the eigenvalues and eigenfunctions of $a_t$ and $q_t$. 
We then feed these estimates into the variational formulae of analytic
perturbation theory in order to track the real-analytic branches. 

So far, our description of the method has been limited to the special case 
of degenerating right triangles. In \S \ref{AdiabaticSection} we make a 
change of variables that places the problem for right triangles 
into the following more general context.  We suppose that there exists 
a positive abstract quadratic from $b$ with simple discrete spectrum, and define
\begin{equation}  \label{atIntro}   
 a_t(u \otimes \vphi)~ =~ t^2 \cdot (\vphi,\vphi)  \int_0^{\infty} |u'(x)|^2~ dx~  
  +~ b(\vphi) \int_0^{\infty} |u(x)|^2~ dx. 
\end{equation}
We consider this family of quadratic forms relative to the weighted $L^2$-inner product
defined by
\[ \langle u \otimes \varphi , v \otimes \psi \rangle~ 
    =~ (\varphi, \psi) \int_0^{\infty}  u \cdot v~ \sigma~ dx. 
\]
where $\sigma$ is a smooth positive function with $\sigma'<0$ and 
$\lim_{x \rightarrow \infty} \sigma(x)=0$. See \S \ref{SectionReduction}. 
The spectrum of $a_t$ decomposes 
into the joint spectra of 
\begin{equation}  \label{atmu}
   a_t^{\mu}(u)~ 
 =~ t^2  \int_0^{\infty} |u'(x)|^2~ dx~  
  +~ \mu \int_0^{\infty} |u(x)|^2~ dx. 
\end{equation}
where $\mu$ is an eigenvalue of $b$ (and hence is positive). Because $\sigma$ is a decreasing function, 
an eigenfunction of $a_t^{\mu}$ with eigenvalue $E$ 
oscillates for $x<< x_E=  (E - \mu \cdot \sigma)^{-1}(0)$
and decays rapidly for  $x>> x_E$. Since $\sigma'<0$, one can approximate 
the eigenfunction (or a quasimode at energy $E$) with Airy functions in 
a neighborhood of $x_E$.  A good deal of the present work is based
on this approximation by Airy functions.  For example, the asymptotics of the
zeroes of Airy functions underlies the super-separation of eigenvalues. 

The following is the general result. 

\begin{thm}[Theorem \ref{SimplicityTheorem}]
If $q_t$ is a a real-analytic family of positive quadratic forms 
that is asymptotic to $a_t$ at first order (see Definition \ref{AsymptoticDefn}), 
then for all but countably many $t$, the spectrum of $q_t$ is simple. 
\end{thm}

Using an induction argument that begins with the triangle, we obtain
the following. 

\begin{coro} \label{Simplex}
For almost every simplex in Euclidean space, each eigenspace of the 
associated Dirichlet Laplacian is one-dimensional.
\end{coro}

Note that Dirichlet boundary conditions can be replaced by any
boundary condition that corresponds to a positive quadratic form $b$.  
In particular, we can choose any mixed Dirichlet-Neumann condition on the faces
of the simplex except for all Neumann.

Using the `pulling a vertex' technique of \cite{HlrJdg09}, we 
can extend generic simplicity to certain classes of polyhedra.  
For example, a $d$-dimensional polytope $P$ is called {\em $k$-stacked} if 
$P$ can be triangulated by introducing only faces of dimension $d-1$ \cite{Grunbaum}. 

\begin{coro}
Almost every $d-1$-stacked convex polytope $P \subset \R^d$ with $n$ vertices
has simple Dirichlet spectrum.
\end{coro}

Finally, we note that by perturbing the curvature of Euclidean space
as in \S 4 of \cite{HlrJdg09} we obtain the following. 

\begin{coro}
Almost every simplex in a constant curvature space form has 
simple Dirichlet Laplace spectrum. 
\end{coro}

\subsection*{Organization of the paper}
In \S \ref{QuasimodeSection} we use standard 
resolvent estimates to quantify the assertion that if two quadratic forms 
are close, then their spectra are close. In particular, we consider the 
projection, $P_{a}^I(u)$, of an eigenfunction $q$ with eigenvalue $E$
onto the eigenspaces of $a$ whose eigenvalues lie in an interval $I \ni E$. 
We show that this projection is essentially a quasimode
at energy $E$ for $a$. 
 
In \S \ref{AsymptoticSection} we specialize these estimates
to the case of two real analytic families of quadratic 
forms $a_t$ and $q_t$. We define what it means for $q_t$ to be asymptotic to $a_t$
at first order. We show that if the first order variation, $\dot{a}_t$, of $a_t$
is nonnegative, then each real-analytic eigenbranch of $a_t$ converge as $t$ tends to
zero, and if $q_t$ is asymptotic to $a_t$ at first order, then the eigenbranches
of $q_t$ also converge.

In section \ref{IntegrabilitySection}, we use the variational
formula along a real-analytic eigenfunction 
branch $u_t$ to derive an estimate on the projection 
$P_{a_t}^I(u_t)$. This results in the assertion that the function 
\[  t~  \mapsto~ \frac{ \|\dot{a}_t(P_{a_t}^I(u_t))\|}{ \|P_{a_t}^I(u_t)\|^2}  \]
is integrable  (Theorem \ref{Thmlimiteigenbranch}). 
The integrability will be used several times in the sequel to control the projection 
$P_{a_t}^I(u_t)$, and in particular, it will be used to prove that
the eigenspaces essentially become one-dimensional in the limit. 
Note that this result depends on both analytic perturbation theory
and resolvent estimates.

Sections \ref{aSection} through  \ref{ConvEstSepSection} are devoted to 
the study of the one dimensional quadratic forms $a_t^{\mu}$ in (\ref{atmu}). 
Most of the material in these sections is based on asymptotics of 
solutions to second order ordinary differential equations (see, for example, \cite{Olver}). 
In \S \ref{ExponentialEstimates} we provide uniform estimates on the $L^2$-norm
of quasimodes and on the exponential decay of eigenfunctions for large $x$.   
In \S \ref{LGSection} we make a well-known change of variables to 
transform the second order ordinary differential equation associated to $a_t^{\mu}$ 
into the Airy equation with a potential. In \S \ref{AirySection} we use elementary 
estimates of the Airy kernel to estimate both quasimodes and eigenfunctions 
near the turning point $x_E$.

In \S \ref{crucial} we use the preceding estimates to 
prove Proposition \ref{CrucialProposition} which essentially says
that the $L^2$-mass of both eigenfunctions and quasimodes of $a_t^{\mu}$ 
does not concentrate at $x_E$ as $t$ tends to zero.
This proposition is an essential ingredient in proving the projection estimates 
of \S \ref{ProjectionSection}. But first we use it in \S \ref{ConvEstSepSection}
to prove that each real-analytic eigenvalue branch of $a_t^{\mu}$ converges to
a threshold $\mu/\sigma(0)$.

In \S \ref{ConvEstSepSection} we also establish the `super-separation'
of eigenvalue branches for $a_t^{\mu}$.
In the case of degenerating right-triangles, we may use the uniform asymptotics of the 
Bessel function (see \cite{Olver}) to obtain the 'super-separation' near the threshold. 
We prove it directly in Proposition \ref{Superseparation} for general $\sigma$.  

In \S \ref{SectionReduction} we establish some basic properties of 
the quadratic form defined in (\ref{atIntro}).  In \S \ref{ProjectionSection}  
we combine results of \S \ref{QuasimodeSection}, 
\S \ref{IntegrabilitySection}, and \S \ref{crucial}
to derive estimates on $P_{a_t}^I(u_t)$ where $u_t$ is a real-analytic  
eigenfunction branch of $q_t$ with eigenvalue branch $E_t$ converging to
a point $E_0$ belonging to the interior of an interval $I$.  

In \S \ref{SectionLimits} we show that each eigenvalue branch of $q_t$ 
converges to some threshold $\mu/\sigma(0)$  (Theorem \ref{qLimit}). 
Observe that this leads to the following natural question: 
Which thresholds $\mu/\sigma(0)$ are limits of some real-analytic eigenvalue branch 
of $q_t$? Strangely enough, we do not answer it here.

In \S \ref{SectionSimplicity} we prove the generic simplicity of $q_t$.
In \S \ref{AdiabaticSection}, we show how simplices and other domains 
in Euclidean space fit into the general framework presented here. 
Finally, in \S \ref{IdealSection} we prove a generalization of Theorem \ref{alpha}.

\setcounter{tocdepth}{1}
\tableofcontents


\part{Aysmptotic families of quadratic forms}

\section{Quasimode estimates for quadratic forms}

\label{QuasimodeSection}

Let ${\mathcal H}$ be a real Hilbert space with inner product
$\langle \cdot, \cdot \rangle$. Let $a$ be a real-valued, 
densely defined, closed quadratic form on $\Hcal$.
Let ${\rm dom}(a) \subset {\mathcal H}$ denote the domain 
of $a$. 

In the sequel, we will assume that the spectrum $\spec(a)$ of $a$
with respect to $\langle \cdot, \cdot \rangle$ is discrete. Moreover,
we will assume that for each $\lambda \in \spec(a)$, the associated 
eigenspace $V_{\lambda}$ is finite dimensional, 
and we will assume that there exists an orthonormal collection,  
$\{\psi_{\ell}\}_{\ell \in \N}$, of eigenfunctions
such that the span of  $\{\psi_{\ell}\}$ 
is dense in ${\mathcal H}$. 

\begin{lem}(Resolvent estimate) \label{resolvent}
Suppose that the distance, $\delta$, from $E$ to the spectrum of $a$ is positive. 
If  
\[   \left| a( w, v)~ -~  E \cdot \langle w, v \rangle \right|~
     \leq~  \epsilon \cdot \|v \|  \]
then 
\[    \|w\|~  \leq~   \frac{\epsilon}{ \delta}. \]
\end{lem} 

\begin{proof}
The norm of the linear functional $v \mapsto a( w, v)- E \cdot \langle w, v \rangle$
is bounded by $\epsilon$, and hence there exists $f$ with $\|f\| \leq \epsilon$ such that 
$\mapsto a( w, v) - E \cdot \langle w,v\rangle =  \langle f, v \rangle$ for all $v$ in $\dom(a)$. 

Let $A$ be the self-adjoint, closed, densely defined operator on ${\mathcal H}$
such that $\langle A u,v \rangle = a(u,v)$ for all $u, v \in \dom(A)\subset \dom(a)$.
It follows that $(A - E \cdot I) w = f$, and hence 
\[   \|w\|~ \leq~ \| (A~ -~ E \cdot I)^{-1} \| \cdot \|f \|. \] 
The spectral radius of $(A~ -~ E \cdot I)^{-1}$ equals $\delta$,
and thus since $A$ is self-adjoint, 
$\|(A~ -~ E \cdot I)^{-1}\| = \delta$.  The claim follows. 
\end{proof}

Given a closed interval $I \subset [0, \infty)$,
define $P^I_a$ to be the orthogonal projection onto 
$\oplus_{\lambda \in I}  V_{\lambda}$.

\begin{defn}
Let $q$ be a real-valued, closed quadratic form defined on 
$\dom(a)$. We will say that $q$ is {\em $\varepsilon$-close} 
to $a$ if and only if for each $v,w \in \dom(a)$, we have 
\begin{equation}  \label{Asymptotic}
   \left| q(v,w) - a(v,w) \right|~ \leq~ 
  \varepsilon \cdot a(v)^{\frac{1}{2}} \cdot a(w)^{\frac{1}{2}}. 
\end{equation}
\end{defn}

Note that if $0 \leq q-a \leq \eps \cdot a$, then $q$ is $\varepsilon$-close to $a$ by the Cauchy-Schwarz 
inequality.

For each quadratic form $q$ defined on $\dom(a)$, define 
\[  n_q(u)~ =~ \left( \| u\|^2+q(u)\right)^{\und}. \]
Note that if $q$ is $\eps$-close to $a$, then the norms 
$n_q$ and $n_a$ are equivalent on 
$\dom(a)$. Thus, the form domains of $q$ and $a$ with respect to $\|\cdot\|$
coincide. We will denote 
this common form domain by ${\mathcal D}.$

\begin{lem}\label{QuasiEstimate}
Let $q$ and $a$ be quadratic forms such that $q$ is  $\eps$-close to $a$. 
If $u$ is an eigenfunction of $q$ with eigenvalue $E$ contained
in the open interval $I \subset \R$, then
\begin{equation}  \label{uchi}
 a \left(u- P^I_a(u) \right)~ \leq~  \eps^2 \cdot a(u)  \cdot
     \left (1+\frac{E}{\delta}\right )^2
 \end{equation}
where $\delta$ is the distance from $E$ to the complement
$\R \setminus I$. 
\end{lem}
\begin{proof}
Let $v \in \Dcal$. 
Since $q(u,v)= E \cdot \langle u,v \rangle$, 
from (\ref{Asymptotic}) we have
\begin{equation} \label{QEuv}
 \left| E \cdot \langle u,v \rangle~ -~ a(u,v) \right|~
    \leq~  \eps \cdot a(u)^\und \cdot a(v)^\und. 
\end{equation}
There exists $f \in \Hcal$ 
such that for all $v \in \Dcal$ we have 
\begin{equation} \label{fdefn}
\begin{array}{lr}
 E \cdot \langle u,v\rangle~ -~ a(u,v)
     \,=\, \langle f,v \rangle. 
\end{array}
\end{equation}
Write $f = \sum f_{\ell}\cdot  \psi_{\ell}$ and define
$v_{{\rm test}}=  \sum \lambda_{\ell}^{-1}  f_{\ell} \cdot \psi_{\ell}$.
Observe that  
\[ \langle f,v_{{\rm test}} \rangle
=~ \sum_{\ell}~ \frac{|f_l|^2}{\lambda_l}
~=~ a(v_{{\rm test}}). 
\]
By substituting $v=v_{{\rm test}}$ into (\ref{QEuv}), we find that  
\begin{equation} \label{flambda}
  \sum_{\ell}~ \frac{|f_l|^2}{\lambda_l}~ \leq~ \eps^2 \cdot a(u). 
\end{equation}
Let $u= \sum_{\ell} u_{\ell} \cdot \psi_{\ell}$. 
From (\ref{fdefn}) we find 
that $u_{\ell} = (E -\lambda_l)^{-1} \cdot f_{\ell}$ 
for each $\ell \in  \N$. Therefore, 
\[ a(u-P_a^I(u))\,=\, \sum_{\lambda_l \notin I}~
\lambda_l \cdot \frac{|f_l|^2}{|E-\lambda_l|^2}~
\leq~ \eps^2 \cdot a(u) \cdot \sup_{\lambda_l \notin I}
  \frac{\lambda_{\ell}^2}{|E- \lambda_{\ell}|^2}
\] 
where the inequality follows from (\ref{flambda}). We have 
\[  \sup_{\lambda_l \notin I}~  
\frac{\lambda_{\ell}^2}{|E- \lambda_{\ell}|^2}~  \leq~ 
   \sup_{|1-x|> \delta/E}~  
\frac{x^2}{|1- x|^2}~ =~ \left(\frac{E}{\delta} +1 \right)^2. \] 
The desired bound follows.
\end{proof}

The preceding lemma provides control of the norm 
of $P_a^I(u).$ In particular, we have the following.

\begin{coro}\label{CoroNormPu}
Let $q$ and $a$ be quadratic forms such that $q$ is  $\eps$-close to $a$. 
If $u$ is an eigenfunction of $q$ with eigenvalue $E$ contained
in the open interval $I \subset \R$, then
\begin{equation}\label{NormPu}
\| P_a^I(u)\|^2 ~\geq~ 
\left[ 1-\eps^2\cdot\left( 1+\frac{E}{\delta}\right)^2\right]
\cdot \frac{ a(u)}{\sup(I)}
\end{equation}
where $\delta$ is the distance from $E$ to the complement
$\R \setminus I$.
\end{coro}

\begin{proof}
Note that $a(u-P_a^I(u), P_a^I(u))=0$, and hence 
\[ a(P_a^I(u))= a(u)-a\left(u-P_a^I(u)\right ). \]
Thus, it follows from Lemma \ref{QuasiEstimate} that
\[  a(P_a^I(u))~ \geq~ \left (1- \eps^2\left(1+\frac{E}{\delta}\right)^2\right)\cdot a(u).
\]
On the other hand, using the orthonomal set $\{\psi_{\ell}\}$
consisting of eigenfunction of $\psi$ with respect to $\langle \cdot, \cdot \rangle$
one finds that 
\[
a(P_a^I(u))\,\leq\, \sup(I) \cdot \| P_a^I(u)\| ^2. \]
The claim follows.
\end{proof}

We use the preceding to prove the following.

\begin{lem} \label{Closeness}
Let $I$ be an  interval, let $E \in I$, and let $\delta$ denote 
the distance from $E$ to the complement $\R \setminus I$.
Let $u$ be an eigenfunction 
of $q$ with eigenvalue $E$.  If 
$\eps  < (1+ E/\delta )^{-1}$ 
and $q$ is $\eps$-close to $a$,  
then for each $v \in \Dcal$,   
we have  
\begin{equation}  \label{Pref}
 \left| a\left(P^I_a(u), v \right)~ -~
     E \cdot \langle P^I_a(u), v \rangle \right|~ 
 ~ \leq~ \frac{\eps \cdot \sup(I)}{
\left( 1-\eps^2\left(1+\frac{E}{\delta}\right)^2\right)^{\und}}
\cdot \|P_a^I(u)\| \cdot \|v\|.
 \end{equation}
\end{lem}

\begin{proof}
Let $\tilde{u} = P^I_a(u)$, and  $\tilde{v} = P^I_a(v).$ 
Since $P_a^I$ is an orthogonal projection that commutes with $a$,
we have $ a(\tilde{u},v)\,=\,a(\tilde{u},\tilde{v})=a(u,\tilde{v})$ 
and $\langle \tilde{u},v\rangle =  
\langle \tilde{u},\tilde{v} \rangle =\langle u,\tilde{v}\rangle.$ 
Therefore, by replacing $v$ with 
$\tilde{v}$ in (\ref{QEuv}) we obtain
\[   \left| a(\tilde{u}, v)~ -~
     E \cdot \langle \tilde{u}, v \rangle \right|~ 
    \leq~ \eps \cdot a(u)^{\und} \cdot a(\tilde{v})^{\und}. 
\] 
Since $\tilde{v} \in P^I_a({\mathcal H})$, we have 
\[ a(\tilde{v})~ \leq~ \sup(I) \cdot \|\tilde{v} \|^2~ \leq~ \sup (I) \cdot \| v\|^2. \] 
By the hypothesis and Corollary  \ref{CoroNormPu}, we have
\[
a(u)~ \leq~  \left[ 1-\eps^2
\left(1+\frac{E}{\delta}\right)^2\right]^{-1}\cdot \sup (I) \cdot \| P_a^I(u)\| ^2.
\]
By combining these estimates, we obtain the claim.
\end{proof}

Let $\{a_n\}_{n \in {\mathbb N}}$ and $\{q_n\}_{n \in \N}$ 
be sequences of quadratic forms defined on $\Dcal$.
For each $n$, let $E_n$ be an eigenvalue of $q_n$.

\begin{prop}\label{Estonchir}
Suppose that $\lim_{n \rightarrow \infty} E_n$ exists and is finite.   
If the quadratic form $q_n$ is $1/n$-close to $a_n$ for each $n$, 
then there exist $N>0$ and $C>0$ such that 
for each $n>N$ and each eigenfunction $u$ of $q_n$ 
with eigenvalue $E_n$,  we have
\begin{equation}  \label{qmest}
      \left| a_n\left(P_{a_n}^I(u),v\right)~ 
      -~ E_n \cdot \langle P_{a_n}^I(u),v \rangle \right|~
          \leq~  C\cdot \frac{1}{n} \cdot \left\| P_{a_n}^I(u) \right\| \cdot \| v\|.
\end{equation}
\end{prop}

\begin{proof}  
Let $E_0= \lim_{n \rightarrow \infty} E_n$ and let  
$I$ be an open interval that contains $E_0$.
Let $\delta_n$ be the distance from $E_n$ to $\R \setminus I$. 
Since $E_n$ converges to $E_0$ and $I$ is open, there exists $\delta_0>0$
and $N_0$ so that if $n > N_0$, then $\delta_n> \delta_0$.
Choose $N \geq \max\{N_0, 1+ 2E_0/\delta_0\}$ so that if $n > N$, then $E_n < 2 E_0$. 
Then for each $n>N$ we have 
$n^{-1} (1+E_n/\delta_n) \leq 1$ and we can apply Lemma \ref{Closeness}
to obtain the claim. 
\end{proof} 

\begin{remk} 

Let $A_n$ be the self-adjoint operator such that $a_n(v,w)= \langle A_nv, w \rangle$
for all $u, v \in \dom(a_n)$. Let $\chi_n$ denote the projection $P_{a_n}^I(u)$.
Estimate (\ref{qmest}) is then equivalent to
\[
\left\| A_n \chi_n~ -~  E_n \cdot \chi_n \right\|~ \leq~ C\cdot \frac{1}{n} \cdot \| \chi_n\|. \]
In the terminology of semi-classical analysis, the function 
$\chi_n$ is a {\em quasimode} at energy $E_0= \lim_{n \rightarrow \infty} E_n$ (see also Remark 
\ref{QuasimodeRemark}).
\end{remk}


\section{Asymptotic families and eigenvalue convergence} 

\label{AsymptoticSection}
Given a mapping of the form $t \mapsto f_t$, we will use 
$\dot{f}_t$ to denote its first derivative. More
precisely, we define  
\[ \dot{f}_t~ :=~ \left.  \frac{d}{ds}\right|_{s=t} f_s. \] 

Let $a_t$ and $q_t$ be real-analytic families of closed 
quadratic forms densely defined on ${\mathcal D} \subset {\mathcal H}$
for $t>0$.\footnote{For notational simplicity, we will often drop the index $t,$ 
but note that each object related to $a$ or $q$ will, in general, depend on $t$.} 
In this section, we show that the nonnegativity of both $a_t$ and $\dot{a}_t$
implies that each real-analytic eigenvalue branch of $a_t$ converges as $t$ tends to zero. 
We then show that if $q$ is asymptotic to $a$ in the following sense 
then the eigenvalue branches of $q_t$ also converge (Proposition \ref{qconvergence}). 

\begin{defn} \label{AsymptoticDefn}
We will say that {\em $q_t$ is asymptotic to $a_t$ at first order}
iff there exists $C>0$  such that 
for each $t>0$
\begin{equation} \label{tclose}
 \left| q_t(u,v)~ -~ a_t(u,v)\right|~ \leq~ C \cdot 
 t \cdot a_t(u)^{\und} \cdot a_t(v)^{\und},
\end{equation}
and for each $v \in \Dcal$
\begin{equation} \label{Boundqdot}
 \left| \dot{q}_t(v)~ -~ \dot{a}_t(v)\right|~ \leq~ C \cdot a_t(v).
\end{equation}
\end{defn}

\begin{remk} \label{Cequals1}
By reparameterizing the family---replacing $t$ by say $t/C$---one 
may assume, without loss of generality, that $C=1$. We will do so 
in what follows. 
\end{remk}

In what follows, we will assume that the eigenvalues and eigenfunctions 
of $a_t$ and $q_t$ vary real-analytically.  To be precise, 
we will suppose for each $t>0$, there exists an orthonormal collection 
$\{\psi_{\ell}(t)\}_{\ell \in \N}$
of eigenvectors whose span is dense in $\Hcal$
such that $t \mapsto \psi_{\ell}(t)$ is real-analytic for each $\ell \in \N$.  
This assumption is satisfied if the operators that represent 
$a_t$ and $q_t$ with respect to $\langle \cdot, \cdot \rangle$ 
have a compact resolvent for each $t>0$.
See, for example, Remark 4.22 in \S VII.4 of \cite{Kato}.

The following proposition is well-known. 
We include the proof here as a model of the more complicated arguments
that follow.

\begin{prop} \label{aconvergence}
If $a_t \geq 0$ and $\dot{a}_t \geq 0$ for all small $t$, 
then each real-analytic eigenvalue branch of $a_t$ converges 
to a finite limit as $t$ tends to zero. 
\end{prop}

\begin{proof}
Let $\lambda_t$ be a real-analytic eigenvalue branch of $a_t$.
That is, there exists for each $t>0$ a function $u_t$ such that
$t \mapsto u_t$ is real-analytic and 
\begin{equation} \label{Eigenequation}
 a_t(u_t,v)~ =~ \lambda_t \cdot \langle u_t, v \rangle 
\end{equation}
for all $v \in \Dcal$. In particular,  $a_t(u_t)= \lambda_t \cdot \|u_t\|^2$,
and thus since  $a_t \geq 0$, we have $\lambda_t \geq 0$ for each $t>0$.

By differentiating (\ref{Eigenequation}) and then taking the
inner product with $v \in \Dcal$, we obtain
\begin{equation} \label{VariationalFormula}
    \dot{\lambda}_t \cdot \|u_t\|^2 ~ =~ \dot{a}(u_t).  
\end{equation}
Thus, since $\dot{a}_t \geq 0$, the function $t \mapsto \lambda_t$
is increasing in $t$.  Since $\lambda_t$ is bounded below, 
the limit $\lim_{t \rightarrow 0} \lambda_t$ exists.   
\end{proof}

If $q_t$ is asymptotic to $a_t$, then the eigenvalues of $q_t$
also converge.

\begin{prop} \label{qconvergence}
Suppose that for each $t>0$, the quadratic forms $a_t$ and $\dot{a}_t$ are nonnegative.
If $q_t$ is asymptotic to $a_t$ at first order, then
each real-analytic eigenvalue branch of $q_t$ converges 
to a finite limit. 
\end{prop}

\begin{proof}
Let $(E_t,u_t)$ be a real-analytic eigenbranch of $q_t$ with respect 
to $\langle \cdot, \cdot \rangle$. By arguing as in the proof of
Proposition \ref{aconvergence} we find that
\begin{equation} \label{Variation}
  \dot{q}(u_t)~ =~ \dot{E}_t \cdot \| u_t\|^2. 
\end{equation}
Using (\ref{tclose}), we have 
\begin{equation}  \label{HalfEstimate}
q_t(v)~ \geq~ \und~ \cdot a_t(v)
\end{equation}
for all $t$ sufficiently small. Since $a_t \geq 0$,
we have $q_t \geq 0$ and hence $E_t \geq 0$ for small $t$.
From  (\ref{Boundqdot}) and Remark \ref{Cequals1} we have $\dot{q}(u_t)~ \geq~ \dot{a}(u_t)- a(u_t)$
and hence, since $\dot{a} \geq 0$, we have
$\dot{q}(u_t)~ \geq~  -a(u_t)$.
By combining this fact with (\ref{Variation}) and (\ref{HalfEstimate}), 
we find that  
\begin{equation} \label{ODEEst}
 \dot{E}_t~ +~ 2 \cdot E_t~ \geq~ 0 
\end{equation}
for sufficiently small $t$.

To finish the proof, it suffices to show that
the function $f(t)= E_t \cdot \exp(2t)$ has 
a finite limit as $t$ tends to zero. By (\ref{ODEEst})
we have  $f'(t) \geq 0$ for $t < t_0$.
Since the form $q_t$ is non-negative for $t<t_0$, 
the eigenvalue $E_t$ is non-negative, 
and thus $f$ is bounded from below. 
Therefore $\lim_{t \rightarrow 0} f(t)$ exists and is finite. 
\end{proof}


\section{An integrability condition} \label{IntegrabilitySection}

Let $q_t$ be a real-analytic family that is asymptotic to $a_t$
at first order. In this section, we use the estimates of \S \ref{QuasimodeSection}
to derive an integrability condition (Theorem \ref{Thmlimiteigenbranch})
that will be used in \S \ref{SectionSimplicity} 
to prove that the spectrum of $q_t$ is simple 
for most $t$ under certain additional conditions.

Let $E_t$ be an real-analytic eigenvalue branch of $q_t$ that converges 
to $E_0$ as $t$ tends to zero. Let $V_t$ be the associated real-analytic family of 
eigenspaces. Let $I$ be a compact interval whose interior contains $E_0$.  

\begin{remk}\label{DefVt}
The definition of $V_t$ implies that, for each $t>0$, the vector space  
$V_t$ is a subspace of $\ker(A_t-E_t\cdot I)$. If a distinct real-analytic 
eigenvalue branch crosses the branch $E_t$ at $t=t_0$, 
then $V_{t_0}$ is a proper subspace of 
$\ker(A_{t_0}-E_{t_0}\cdot I)$.
\end{remk} 

\begin{thm}\label{Thmlimiteigenbranch}
Let $q_t$ be asymptotic to $a_t$ at first order,
and suppose that for each $t>0$, we have 
\begin{equation} \label{Logarithm2}
 0~ \leq~  \dot{a}_t(v)~ \leq~ t^{-1} \cdot a_t(v).
\end{equation}
If  $t \mapsto u_t \in V_t$ is continuous on 
the complement of a countable set, then the function
\begin{equation} \label{festimate}
t~ \mapsto~ \frac{\dot{a}_t\left(P^I_{a_t}(u_t) \right)}{ 
 \left\| P^I_{a_t}(u_t) \right\|^2}      
\end{equation}
is integrable on each interval of the form $(0,t^*]$.  
\end{thm}

\begin{proof} 
Since the spectrum of $a_t$ is discrete and $E_t$
is real-analytic, the operator family $t \mapsto P^I_{a_t}$ 
is real-analytic on the complement of a countable set. 
By combining this with the hypothesis, we find that the function 
$\dot{a}(P^I_{a_t}(u_t) )/\|P_{a_t}(u_t)\|^2$ is
locally integrable on $(0,t^*]$ for each $t^*>0$.

By Lemma \ref{DotQuasiBound} below, there exists a constant $C>0$ such that
\[
  \dot{E}_t~ \geq~ C \cdot \frac{ \dot{a}_t(\chi_t) }{
  \|\chi_t\|^2} ~  -~ C. 
\]
Integration then gives
\[  E_{t^*}~ -~ E_{t}~ \geq~ C \int_{t}^{t^*} 
  \frac{ \dot{a}_s(\chi_s) }{\|\chi_s\|^2}~ ds~  -~ C (t^*-t).
\]
Since $E_t \geq 0$ and the integrand is
nonnegative, the integral on the right hand side converges
as $t$ tends to zero.  
\end{proof}

\begin{lem} \label{DotQuasiBound}
Suppose that for each $t>0$, we have 
\begin{equation} \label{Logarithm}
 0~ \leq~  \dot{a}_t(v)~ \leq~ t^{-1} \cdot a_t(v).
\end{equation}
If $q_t$ is asymptotic to $a_t$ at first order,
then there exists $t'>0$ and a constant $C>0$ 
such that for each $t\leq t'$ and each eigenvector $u \in V_t$ we have
\begin{equation} \label{OQuasiBound}
  \left| \dot{E}_t \cdot  \|u\|^2~
 -~ \dot{a}_t\left( P^I_{a_t}(u) \right)
 \right|~ \leq~  C \cdot \| u\|^2 
\end{equation}
and 
\begin{equation} \label{QuasiLowerBound}
 \| P_a^I(u)\|~ \geq~ \frac{1}{C} \cdot \|u\|. 
\end{equation}
\end{lem}

\begin{proof}
Since $V_t$ is the real-analytic family of eigenspaces associated
to $E_t$, for each $t>0$ and $u \in V_t$ we have $\dot{q}(u) =\dot{E} \cdot \|u \|^2$ (see Remark \ref{DefVt}). 
Since $E_t$ converges to $E_0$, we find 
using (\ref{tclose}) that
there exists $t_0$ so that for $t<t_0$  
\begin{equation} \label{atbound}
 a_t(u)~ \leq~ 2 q_t(u)~ =~ 2 E_t \cdot \|u\|^2~ \leq~ 2(E_0+1) \cdot \|u\|^2. 
\end{equation}
Thus, from (\ref{Boundqdot}) we find that
\begin{equation}\label{dotE1}
 \left|\dot{E} \cdot \|u\|^2~ -~ \dot{a}_t(u) \right|~ 
           \leq~   2(E_0+1) \cdot \|u\|^2
\end{equation}
for $t< t_0$.

Let $\chi_t=P_{a_t}^I(u)$. 
Since $\dot{a}_t$ is a nonnegative quadratic form, we have 
\[  \dot{a}_t(u)~ \leq~ \dot{a}_t(\chi_t)~ +~ 2 \dot{a}_t(\chi_t)^\und
\cdot \dot{a}_t(u-\chi_t)^\und~ +~ \dot{a}_t(u-\chi_t) \]
and 
\[  \dot{a}_t(\chi_t)~ \leq~ \dot{a}_t(u)~ +~ 2 \dot{a}_t(u)^\und
\cdot \dot{a}_t(\chi_t-u)^\und~ +~ \dot{a}_t(\chi_t-u). \]
The former estimate provides a bound on $\dot{a}_t(u)-\dot{a}_t(\chi_t)$ 
and the latter one gives a bound on its negation. In particular, we find that 
\[ 
\left| \dot{a}_t(u)-\dot{a}_t(\chi_t) \right|~
 \leq~ 2 \cdot \max \left\{ \dot{a}_t(u)^\und,\dot{a}_t(\chi_t)^\und \right\} 
\cdot \dot{a}_t(u-\chi_t)^\und~ +~ \dot{a}_t(u-\chi_t).
\]
Thus, by (\ref{Logarithm}), we have
\begin{equation} \label{AfterLog}
  |\dot{a}_t(u)-\dot{a}_t(\chi_t)|~
\leq~ \frac{2}{t} \cdot \max \left\{ a_t(\chi_t)^\und,a_t(u)^\und \right\}
 a_t(u-\chi_t)^\und~
 +~ \frac{a_t(u-\chi_t)}{t}. 
\end{equation}

Let $\delta_t$ be the distance from $E_t$ to the complement $\R \setminus I$. 
Since $E_0$ belongs to the interior of $I$ and $E_t \rightarrow E_0$, 
there exists $\delta>0$ and $0<t_1\leq t_0$ 
so that if $t< t_1$, then $\delta_t \geq \delta$.
Hence we may apply Lemma \ref{QuasiEstimate} to find that
\[   a_t \left(u~ -~ \chi_t \right)~ \leq~  t^2 \cdot a_t(u)  \cdot
     \left (1+\frac{2 E_0}{\delta}\right )^2
\]
for $t<t_1$. Since $a_t$ is non-negative, from (\ref{atbound}) we have
\[ a_t(\chi_t)~ \leq~ a_t(u)~ \leq~ 2(E_0+1) \cdot \|u\|^2 \]
for $t \leq t_0$. By combining these estimates with (\ref{AfterLog}) we
find that for $t \leq t_1$
\begin{equation} \label{FinalQuasi}
    |\dot{a}_t(u)-\dot{a}_t(\chi_t)|~ \leq~
            2(E_0+1) \cdot \|u\|^2 \cdot \left( 1+ \frac{2 E_0}{\delta} \right) \cdot
            \left( 2 + t \cdot \left(1 + \frac{2 E_0}{\delta}\right) \right).
\end{equation}
Estimate (\ref{OQuasiBound}) then follows from (\ref{dotE1}), (\ref{FinalQuasi})
and the triangle inequality.

If $E_0 > 0$, then there exists $0<t_2 \leq t_1$ such that
if $t<t_2$, 
then 
\[ a_t(u)~ \geq~ \und \cdot q_t(u)~ =~ \und \cdot E_t\cdot \|u\|^2~
 \geq~  \frac{1}{4} \cdot E_0 \cdot \|u\|^2.
\]
Thus, if $E_0>0$, then (\ref{QuasiLowerBound}) follows from 
Corollary \ref{CoroNormPu}. 

On the other hand, if $E_0=0$, then let $t_1$ and $\delta$
be as above. Since $P^I_{a}$ is a spectral projection and the 
eigenspaces are orthogonal, 
we have 
\[ a\left(u- P_a^I(u)\right)~ \geq~ \delta \cdot \left\|u- P_a^I(u) \right\|^2.\]
Thus, by Lemma \ref{QuasiEstimate} and (\ref{atbound})
we have
\[  2 t^2 \cdot (E_0+1) \cdot \left(1 + \frac{E_0+1}{\delta} \right)
 \cdot  \|u\|^2~ \geq~  \delta \cdot \left\|u- P_a^I(u) \right\|^2.
\] 
In particular, if $t^2< (\delta/8) \cdot (E_0+1)^{-1} \cdot (1+ (E_0+1)/\delta)^{-1}$,
then 
\[    \|u\|^2~ \geq~  \frac{1}{4} \cdot\left\|u - P_a^I(u) \right\|^2.  \]
Estimate (\ref{QuasiLowerBound}) then follows from the triangle inequality.
\end{proof}


\part{A family of quadratic forms on the half-line}


\section{Definition and basic properties}  \label{aSection}

In the sequel $\sigma: [0,\infty) \rightarrow \R^+$ will be a smooth positive function such that 
\begin{itemize}
    \item $\lim_{x \rightarrow \infty} \sigma(x)=0$,

    \item $\sigma'(x)<0$ for all $x \geq 0$, 
    \item $|\sigma''|$ has at most polynomial growth on $[0, \infty)$.
\end{itemize}
For $u,v \in C^{\infty}_0((0, \infty))$, define 
\[   \langle u,v \rangle_{\sigma}~ =~ 
\int_0^{\infty} u(x) \cdot v(x) \cdot \sigma(x)~ dx. \]
Let ${\mathcal H}_{\sigma}$ denote the Hilbert space obtained by completing $C_0^{\infty}((0, \infty))$
with respect to the norm  $\|u\|_{\sigma}:=   \sqrt{\langle u,u \rangle_{\sigma}}$.

Let $H^1(0,\infty)$ and $H^1_0(0,\infty)$ denote, respectively,
the classical Sobolev spaces with respect to Lebesgue measure  
on $(0,\infty)$. For each $t>0$ and $u$ in $H^1(0,\infty)$, we define 
\[  a^{\mu}_t(u)~ =~  \int_{0}^{\infty} \left( t^2 \cdot |u'(x)|^2~ 
    +~ \mu \cdot |u(x)|^2 \right)~ dx. \]  

\begin{remk}  \label{MuPos}
If $\mu>0$, then since $\sigma$ is decreasing, we have
\[\| u\|_\sigma^2~  \leq~ \sigma(0) \int_0^\infty |u(x)|^2 dx~
   \leq~ \frac{\sigma(0)}{\mu}a^{\mu}_t(u).
\]
\end{remk}

Let 
\[   \dom_D(a^{\mu}_t) =  H^1_0(0,\infty) \cap \mathcal{H}_{\sigma}  \]
and let
\[ \dom_N(a^{\mu}_t)\,=\,H^1(0,\infty)  \cap \mathcal{H}_{\sigma}.  \]
Both $\dom_D(a^{\mu}_t)$ and $\dom_N(a^{\mu}_t)$ are 
closed form domains for $a$ that are dense in ${\mathcal{H}}_{\sigma}$.

\begin{defn}\label{defDirNeu}
The spectrum of the quadratic form $a_t^{\mu}$ restricted to $\dom_D(a^{\mu}_t)$
(resp. $\dom_N(a^{\mu}_t)$)  with respect to $\langle \cdot, \cdot \rangle_{\sigma}$ 
will be called the {\em Dirichlet} (resp. {\em Neumann}) {\em spectrum of $a_t^{\mu}$}. 
\end{defn}

In the sequel, we will drop the subscript `$D$' from 
 $\dom_D(a_t^\mu)$ and the subscript `$N$' from $\dom_N(a^{\mu}_t)$.  
In particular, unless stated otherwise, all of the results 
below hold for both the Neumann and Dirichlet boundary conditions.
When we refer to the `spectrum' of $a_t^{\mu}$,
we will mean either the Dirichlet or the Neumann spectrum.

\begin{prop}
If $\mu>0$ and $t>0$, then the quadratic form $a^{\mu}_t$ has discrete 
spectrum with respect to $\langle \cdot, \cdot \rangle_{\sigma}$. 
\end{prop}

\begin{proof}
By a standard result in spectral theory---see, 
for example, Theorem XIII \cite{Reed-Simon}---it 
suffices to prove that for each $r>0$ the set 
\[  A_r~ =~ \left\{u \in \dom(a_t^{\mu})~ |~ a^{\mu}_t(u)~ \leq~ r,~ \|u\|_{\sigma}~ \leq~  1 \right\}  \]
is compact with respect to $\|\cdot\|_{\sigma}$.  

Let $(u_n)_{n\in \N} \subset A_r.$  By the definition of $a^\mu_t$,  
for each $n \in N$ and $M>0$, the  $H^1(0,M)$  norm of $u_n$
is at most   $2 r/ \min\{t^2, \mu\}$. Thus, by Rellich's lemma,
$u_n$ has a convergent subsequence in $L^2(0,M)$ for each $M$.
Diagonalization provides a function $u_\infty$ in $L^2_{loc}(0,\infty)$ 
and a subsequence $u_{n_k}$ such that $u_{n_k}$ converges 
to $u_{\infty}$ in $L^2([0,M])$ for each $M$. 

To finish the proof, it suffices to show that $u_{n_k}$ is Cauchy in ${\mathcal H}_{\sigma}$. 
Since $\sigma$ is decreasing, for each $u\in A_r$ we find that 
\begin{eqnarray*}   
\int_{x>M} |u|^2 \cdot \sigma(x)~ dx~
   &\leq& \sigma(M)  \int_{x>M} |u|^2~ dx~ \\
&\leq& \sigma(M) \cdot \frac{r}{\mu}.    
\end{eqnarray*}
Thus, since $\lim_{x \rightarrow \infty} \sigma(x)=0$, 
for each $\eps>0$, there exists $M_\eps$ such that for each $u \in A_r$ we have 
\[ 
    \int_{x>{M_\eps}} |u|^2\cdot \sigma(x)~ dx~ \leq~ \eps.
\]
Since $u_{n_k}$ converges in $L^2([0,M_\eps])$, there exists $k_0$ such that
if $j,k \geq k_0$, then 
\[ 
\int_0^{M_\eps} |u_{n_k}(x)-u_{n_j}(x)|^2 \cdot \sigma(x)~
 dx~ \leq~ \sigma(0) \cdot \|u_{n_k}-u_{n_j}\|_{[0,M_\eps]}^2~ \leq~ \eps.
\]
Combining the two preceding estimates we find that $(u_{n_k})$ is a Cauchy sequence in $\mathcal{H}_\sigma$.
\end{proof}


\section{Estimates of quasimodes and eigenfunctions}\label{ExponentialEstimates}

In the sequel, unless otherwise stated, we assume that $\mu>0$. 
  
Let $r \in {\mathcal H}_{\sigma}$ and let $E\geq 0$.
In this section, we begin our analysis of functions 
$w$ in $\dom(a_t^\mu)$  that satisfy
\begin{equation}  \label{rDefined}  
   a^{\mu}_{t}(w,v)~ -~ E \cdot  \langle w, v\rangle_{\sigma}~ 
         =~ \langle r, v\rangle_{\sigma}~ 
\end{equation}
for all $v \in {\rm Dom}(a_{t}^{\mu})$. 

In applications, the function $r$ in (\ref{rDefined}) will be negligible.  
For example, if $r=0$, then $w$ is an eigenfunction with eigenvalue $E$.
More generally, if 
\begin{equation}  \label{rDefinedQuasi}  
   a^{\mu}_{t_n}(w_n,v)~ -~ E_n \cdot  \langle w_n, v\rangle_{\sigma}~ 
         =~ \langle r_n, v\rangle_{\sigma}~ 
\end{equation}
where $t_n \rightarrow 0$, $w_{n} \in {\rm Dom}(a_t^{\mu})$, $\lim E_{n}\,=\, E_0$ and 
$\|r_{n}\| = O(t_n^{\rho}) \cdot \|w_{t_n}\|$, then the sequence $w_{n}$ is called 
{\em a quasimode of order $\rho$} at energy $E_0.$ (See also Remark \ref{QuasimodeRemark}.)

Our goal is to understand the behavior of both eigenfunctions and quasimodes. 
Of course, in most situations, either the eigenfunction estimate will be 
stronger  than the quasimode estimate and/or the proof will be simpler. 
In the following, we will first provide a general estimate---valid for any 
quasimode---and then, as needed, we will state and prove the stronger 
result for eigenfunctions.
  
By unwinding the definitions, equation (\ref{rDefined}) may be rewritten as  
\begin{equation}  \label{definitionr}
\int_0^{\infty} 
\left( t^2 \cdot w'(x) \cdot v'(x)~ +~ f_E(x) 
\cdot w(x) \cdot v(x) \right)~ dx~ =~  \int_{0}^\infty r (x)\cdot v(x) \cdot \sigma(x)~ dx.
\end{equation}
where 
\[  f_E(x)~ = \mu~ -~ E \cdot \sigma(x).  \]
By integrating (\ref{definitionr}) by parts, we find that $w$ satisfies
(\ref{definitionr}) for all $v \in C^{\infty}_0((0,\infty))$ if and only if 
for each $x \in (0, \infty)$
\begin{equation}  \label{ODE}   
 -t^2 \cdot w''(x)~ +~  f_E(x) \cdot w(x)~ =~ r(x) \cdot \sigma(x). 
\end{equation}
Moreover, the Dirichlet and Neumann cases are characterised by the 
usual boundary conditions. To be precise, $w$ is a Dirichlet (resp. Neumann) 
eigenfunction of $a_t^\mu$ if and only if $w$ is in $\mathcal{H}_\sigma\cap H^1$, 
satisfies equation (\ref{ODE})  and $w(0)=0$ (resp. $w'(0)=0$).  
  
Let $E \geq \mu/\sigma(0)$.\footnote{For example, it follows from
the definitions that each eigenvalue 
of $a_t^\mu$ satisfies this inequality.} 
Since $\sigma$ is strictly decreasing, 
there exists a unique point $x_E \in [0, \infty)$ 
such that $f_E(x_E)$. In particular, if $x > x_E$,
then $f_E(x) > 0$ and if $x< x_E$, then $f_E(x)<0$.

If $w$ is an eigenfunction ($r=0$), then 
one expects $w$ to behave like an exponential 
function when $x>>x_E$ and to oscillate for $x<<x_E$. 
Moreover, as $t$ tends to zero, one expects 
that both types of behavior will become more and more extreme.  

On the other hand, since 
$\lim_{x \rightarrow \infty} \sigma(x)=0$, we do not know, for example,  
that $r$ is bounded as $x$ tends to infinity.  In particular,
for a non-zero $r$, we have no direct argument that shows that 
a solution $w$ to (\ref{ODE}) has exponential decay or is, in fact, bounded.

\subsection{A general $L^2$ estimate}

For each $E \geq \mu/\sigma(0)$ and $s \in [0,\mu)$, 
let $x_E^s \geq x_E$ be the unique solution to 
\begin{equation} \label{xs}
 f_E(x_E^{s})~ =~  s. 
\end{equation}

\begin{remk}\label{RmkSmoothness}
Observe that since the derivative of $\sigma$ does not vanish, the mappings 
$E \mapsto x_E$ and $E \mapsto x_E^s$ are smooth 
from $[\mu/\sigma(0), \infty)$ to $(0,\infty).$ 
Note also that $\lim_{s \rightarrow 0} x_E^s=x_E$ and 
 $\lim_{s \rightarrow \mu} x_E^s=\infty$.
\end{remk}

The following estimate shows that if $w_t$ satisfies (\ref{rDefined})
and 
 \[  \lim_{t \rightarrow 0}~ \frac{\|r\|_\sigma}{\| w_t\|_\sigma}~ =~ 0, \]
then, for any fixed $s,$ the $L^2$ mass of $w_t$ concentrates in the region 
$\{ x~ |~ x \leq x_E^s\}$ as $t$ goes to $0$. Additional work is required to  
prove that $w_t$ actually concentrates in the {\em classically allowed region} 
$\{ x~ |~ x\leq x_E \}$. See Proposition \ref{CrucialProposition}.

\begin{lem}  \label{ExponentialEstimate}
Let $K \subset [\mu/\sigma(0), \infty)$ be compact and let $s \in (0, \mu)$.
There exists a constant $C$ such that for each $E \in K$,
$r \in {\mathcal H}_\sigma$, and 
solution $w$ to (\ref{ODE})  we have
\[     \int_{x_E^s}^{\infty} |w(x)|^2~ dx~ \leq~  
     C \cdot \left(t^2~ +~ \frac{\|r\|_{\sigma}}{\|w\|_{\sigma}} \right) 
    \cdot \int_{0}^{\infty} |w(x)|^2~ dx.
\]
The constant $C$ depends only upon $K$, $\mu$, $\sigma$, and $s$. 
\end{lem}

\begin{proof}
Let $\chi: \R \rightarrow [0,1]$ be a smooth function 
such that $\chi(x)=0$ for all $x \leq 0$
and $\chi(x)=1$ for all $x\geq 1$.  For each $M \in \R$, define 
\[  \rho_M(x)~ =~ \chi\left( \frac{x - x_E}{x_E^s-x_E}\right) \cdot \chi(M+1-x). \]
Substitute $\rho_M \cdot w$ for $v$ in (\ref{definitionr}).
Since $\rho_M$ vanishes for $x\leq x_E$, we obtain
\begin{equation}  \label{afterwrho}
\int_{x_E}^{\infty} 
t^2 \cdot (w')^2 \cdot \rho_M~ +~ t^2 \cdot w'\cdot w \cdot 
\rho_M' +~ f_E(x) 
\cdot w^2 \cdot \rho_M~dx
=~  \int_{x_E}^\infty r \cdot w \cdot \rho_M \cdot \sigma\, dx.
\end{equation}
By integrating by parts, one finds that  
\[  \int_{x_E}^\infty w \cdot w' \cdot \rho_M'~ dx~ 
=~ -\frac{1}{2} \int_{x_E}^\infty w^2 \cdot \rho_M''~ dx, \]
and hence (\ref{afterwrho}) is equivalent to 
\[ 
t^2 \int_{x_E}^{\infty} (w')^2 \cdot \rho_M~ dx~
 +~~ \int_{x_E}^{\infty} f_E \cdot w^2 \cdot \rho_M~ dx~ 
=~  \frac{t^2}{2}  \int_{x_E}^{\infty} \rho_M'' \cdot w^2~ dx 
 + \int_{x_E}^\infty r \cdot w \cdot \rho_M \cdot \sigma~ dx.
\]
The first integral on the left hand side 
is positive. Moreover, since $f_E\cdot w^2\cdot \rho_M$ is non-negative on 
$[x_E,x_E^s]$, and  $0 \leq \rho_M\leq 1$, we have 
\begin{equation} \label{Lastrho}  
 \int_{x_E^s}^{\infty} f_E\cdot w^2 \cdot \rho_M~ dx~ 
\leq~  \frac{t^2}{2}  \int_{x_E}^{\infty} |\rho_M''| \cdot w^2~ dx~ 
 +~ \int_{x_E}^\infty |r \cdot w| \cdot \sigma~ dx.
\end{equation}

The function $|\rho_M''|$ 
is bounded by a constant multiplied by $|x_E^s-x_E|^{-2}$. Therefore, since 
$x_E<x_E^s,$ and$x_E,~x_E^s$ are smooth over the compact $K$ (see remark \ref{RmkSmoothness}), 
there exists $C'>0$ 
such that for each $E \in K$, $x \in [0, \infty)$,
and $M \in \R$, we have $|\rho_M''(x)| \leq C'$.
 
By applying this estimate to (\ref{Lastrho}) and applying 
the Cauchy-Schwarz inequality, we obtain 
\[   \int_{x_E^s}^{\infty} f_E \cdot w^2 \cdot \rho_M~ dx~ 
\leq~  \frac{C' \cdot t^2}{2}  \int_{x_E}^{\infty}  w^2~ dx~ 
      +~   \|r\|_{\sigma} \cdot \|w\|_{\sigma}
\]
Note that for $x \geq x_E^s$, we have $\rho_M(x)= \chi(M+1-x)$.
Thus, since $f_E\cdot w^2$ is integrable, by the Lebesgue dominated 
convergence theorem, 
we may let $M$ tend to $\infty$ and obtain
\[ 
   \int_{x_E^s}^{\infty} f_E \cdot w^2 ~ dx~ 
\leq~  \frac{C' \cdot t^2}{2}  \int_{x_E}^{\infty}  w^2~ dx~ 
      +~   \|r\|_{\sigma} \cdot \|w\|_{\sigma}
\]
Since $\sigma$ is decreasing, the function $f_E(x)$ is increasing
and 
\[  \|w\|_{\sigma}^2~ \leq~ \sigma(0)~ \int_0^{\infty} w^2(x)~ dx. \]
Therefore, we find that 
\[   f_E(x_E^s) \cdot \int_{x_E^+}^{\infty}  w^2 ~ dx~ 
     \leq~ \left( \frac{C' \cdot t^2}{2}~ 
 +~ \sigma(0) \cdot \frac{\|r\|_{\sigma}}{\|w\|_{\sigma}} \right)
    \int_{0}^{\infty}  w^2~ dx.
\]
Since $f_E(x_E^s)= s$, the Lemma follows by choosing 
$C$ to be $s^{-1}\max (\frac{C'}{2},\sigma(0))$. 
\end{proof}

\subsection{An estimate of the $L^2$ mass of an eigenfunction}

If $w$ is an eigenfunction, then the bound given in Lemma
\ref{ExponentialEstimate} can be greatly improved. In particular, 
an eigenfunction is exponentially small in the classically forbidden region, 
and hence one can make $L^2$ estimates with polynomial weights. See 
Lemma \ref{L2Eigenfunction}.

First, we quantify the exponential decay of each eigenfunction.

\begin{lem} \label{ExponentialDecay}
Let $w$ be an eigenfunction of $a^{\mu}_t$ with eigenvalue $\lambda \leq E$.
If $x \geq y \geq x_E^s$, then  
\begin{equation} \label{Exp+}
   w^2(x)~ \leq~ w^2(y) \cdot {\rm exp} \left(-\frac{\sqrt{2s}}{t} \cdot (x-y)\right) 
\end{equation}
\end{lem}

\begin{proof}
Since $r=0$, multiplying equation (\ref{ODE}) by $w$ gives
\[
    t^2 \cdot w'' \cdot w~ =~  f_E \cdot w^2. 
\]
Since $(w^2)'' \geq 2 w'' \cdot w$, we have
\[     t^2 \cdot (w^2)''~ \geq~  2f_E \cdot w^2.  \]
If $x \geq  x_E^s$, then $f_E(x) \geq s$ and thus
\[    (w^2)''~ \geq~  \frac{2s}{t^2} \cdot w^2.  \]
In particular, $w^2$ is convex on $[x_E^s, \infty)$,
and since $w^2$ is non-negative and integrable, 
$w^2(x)$ tends to $0$ as $x$ tends to $\infty$. 

Define $M_y: [x_E^s, \infty) \rightarrow \R$ by
\[
  M_y(x)~ =~ w^2(x)~ 
-~ w^2(y) \cdot \exp\left(-\frac{\sqrt{2s}}{t} \cdot (x-y)\right)~
 \] 
Then $M_y(y)=0$,  $\lim_{x \rightarrow \infty} M_y(x)=0$, and
\begin{equation} \label{MConvexity}
   M_y''(x)~ \geq~ \frac{s}{t^2} \cdot M_y(x).
\end{equation}
Since $s/t^2>0$, the maximum 
principle gives that $M_y(x) \leq 0$ for each $x \geq y$, and 
(\ref{Exp+}) follows. 
\end{proof}

This estimates allows us to prove the following.

\begin{lem} \label{L2Eigenfunction}
For each $\nu>0$,  there exists a function 
$\beta_\nu : (\mu/\sigma(0),\infty)\times (0,\mu)\rightarrow \R $  
such that if $w$ is an eigenfunction of $a_t^{\mu}$ 
with eigenvalue $\lambda \leq E,$ and $t \leq 1$, then
\[
  \int_{3 x_E^s}^{\infty} w^2(x)\cdot (1+x^\nu)~ dx~ 
     \leq~  \beta_\nu(E,s) \cdot t\cdot \int_{x_E^s}^{\infty} w^2(x)~ dx.          
\]
\end{lem}

\begin{proof}
Let $\alpha= \sqrt{2s}/t$. By exchanging the roles of $x$ and $y$ in 
\refeq{Exp+}, we find that for all $x \in [x_E^s, y]$ 
\begin{equation} \label{Exp-}
   w^2(x)~ \geq~ w^2(y) \cdot {\rm exp} \left(\alpha \cdot (y-x)\right).
\end{equation}
Integrating with respect to $x,$ we obtain
\[   \int_{x_E^s}^{y} w^2(x)~ dx~ \geq~ 
  \frac{1}{\alpha} \cdot w^2(y) \cdot \left({\rm exp}(\alpha \cdot(y-x_E^s)-1) \right)
\]
and thus
\begin{equation} \label{wUpper}
  w^2(y)~ \leq~  \frac{\alpha}{{\rm exp}(\alpha \cdot(y-x_E^s))-1} \cdot 
   \int_{x_E^s}^{y} w^2(x)~ dx. 
\end{equation}

If $u \geq 0$, then $u^\nu \leq c_\nu \cdot e^u$, where $c_\nu=\sup\{ x^\nu e^{-u} ~|~u>0 \}.$ Hence, 
we have  
\[ x^{\nu}~  \leq~ c_\nu \cdot \left(\frac{2}{\alpha} \right)^\nu \cdot e^{\alpha \cdot x/2} \] 
By combining this with (\ref{Exp+}), we find that for $x \geq y$ 
\[   w^2(x)\cdot x^{\nu}~ 
   \leq~ c_\nu \cdot \left( \frac{2}{\alpha} \right)^\nu \cdot 
   w^2(y) \cdot {\rm exp} \left(-\alpha \cdot \left( \frac{x}{2}~ - y \right) \right)
\]
By integrating, we find that 
\[   \int_y^{\infty}  w^2(x)\cdot x^{\nu}~ dx~ \leq~
  c_\nu \cdot  \left( \frac{2}{\alpha} \right)^{\nu+1}  
   \cdot  w^2(y) \cdot {\rm exp}( \alpha\cdot y/2).
\]
Putting this together with (\ref{wUpper}) gives
\begin{equation}  \label{wCombine}
  \int_y^{\infty}  w^2(x)\cdot x^{\nu}~ dx~ \leq~
   2 \cdot c_\nu \cdot \left( \frac{2}{\alpha} \right)^\nu
 \cdot \left( 
  \frac{ {\rm exp}(\alpha\cdot y/2)}{{\rm exp}(\alpha \cdot y- \alpha \cdot x_E^s))-1}
  \right)
 \int_{x_E^s}^{y}  w^2(x)~ dx.
\end{equation}

If we let 
\[  c_{\nu}'~ =~ 
 \sup \left\{ \left. x \cdot \frac{\exp(3x/2)}{\exp(2x)-1}~ \right|~ x>0 \right\}
\]
and set $y=3 \cdot x_E^s$, then we have
\[ \frac{ {\rm exp}(\alpha\cdot y/2)}{{\rm exp}(\alpha \cdot y- \alpha \cdot x_E^s))-1}~
  \leq~   \frac{c_{\nu}'} {\alpha\cdot x_E^s}.
\]
By substituting this into (\ref{wCombine}) we obtain 
\begin{equation} \label{preFinal}
  \int_{3 x_E^s}^{\infty} w^2(x)\cdot x^{\nu}~ dx~ 
     \leq~  \frac{2 c_{\nu} \cdot c_{\nu}'}{x_E^s} \cdot t^{\nu+1} \cdot s^{-\frac{\nu+1}{2}}
      \int_{x_E^s}^{\infty} w^2(x)~ dx.          
\end{equation}
The claim then follows by specializing (\ref{preFinal}) to the case $\nu=0$
and adding the resulting estimate to (\ref{preFinal}). 
In particular, we may define
\[  \beta_{\nu}(E,s)~ =~  2 \cdot
 \frac{c_{0} \cdot c_{0}'~ +~ c_{\nu} \cdot c_{\nu}'}{x_E^s \cdot s^{\frac{\nu+1}{2}}}.
\] 
\end{proof}

\subsection{Comparing weighted $L^2$ inner products on eigenfunctions}

Let $p: [0, \infty) \rightarrow \R$ be a positive continuous 
function of (at most) polynomial growth. That is, there 
exist constants $C_p$ and $\nu_p$ such that if $x \geq 0$, then 
\[  0\,<\,p(x)\,\leq\, C_p \cdot \left(1~ +~ x^{\nu_p}\right). \]
We will regard $p$ as a weight for an $L^2$-inner product.

\begin{prop} \label{BiLinEst}
Let $p$ be as above.
There exists a function $\alpha: [\mu/\sigma(0), \infty) \times (0,\mu) \rightarrow \R$
such that if $s \in (0, \mu)$, then 
\begin{equation} \label{alphasmall}
   \lim_{E \rightarrow \mu/\sigma(0)}~ \alpha(E,s)~ =~ 0 
\end{equation}
and a function $\beta:(\mu/\sigma(0), \infty) \times (0,\mu) \rightarrow \R$ such that   
if $w_{\pm}$ is an eigenfunction of $a_t^{\mu}$ with eigenvalue $\lambda_{\pm} \leq E$, 
then  
\begin{equation*}
  \left| \int_0^{\infty} w_+ \cdot w_- \cdot p~ dx~ 
   -~ p(0)  \int_0^{\infty} w_+ \cdot w_-~ dx 
 \right|~ 
 \leq~  \left( \alpha(E,s)~ +~ \beta(E,s) \cdot t \right)   
  \int_0^{\infty} w^2~ dx.
 \end{equation*}
The functions $\alpha$ and $\beta$ depend only on $p$, $E$, $\sigma$, and $\mu$. 
\end{prop}

\begin{proof}
Set
\[ \alpha(E,s)~ =~ \sup \left\{ |p(x)-p(0)|~ |~  0 \leq x \leq 3x_E^s \right\}. \]
Since $p$ is continuous and $\lim_{E \rightarrow \mu/\sigma(0)}  x_E^s=0$
we have (\ref{alphasmall}).
Using the Cauchy-Schwarz inequality we find that
\[ \left| \int_0^{3x_E^s} w_+ \cdot w_- \cdot p~ dx~ 
   -~ p(0)  \int_0^{3x_E^s} w_+ \cdot w_-~ dx \right|~ 
  \leq~ \alpha(E,s) \cdot \|w_+\| \cdot  \|w_-\|. 
\]

We also have 
\[ \left(\int_{3x_E^s}^{\infty} w_+ \cdot w_- \cdot p~ dx\right)^2~
  \leq~ \left(\int_{3x_E^s}^{\infty} |w_+|^2 \cdot p~ dx \right) \cdot
  \left( \int_{3x_E^s}^{\infty} |w_-|^2 \cdot p~ dx \right). 
\] 
By Lemma \ref{L2Eigenfunction} we have 
\[ \int_{3x_E^s}^{\infty} |w_{\pm}|^2 \cdot p~ dx~ \leq~
   C_p \cdot \beta_{\nu_p}(E,s) \cdot  t   \int_{0}^{\infty} |w_{\pm}|^2~ dx~
\]
and also 
\[ p(0) \int_{3x_E^s}^{\infty} |w_{\pm}|^2 ~ dx~ \leq~
   p(0) \cdot \beta_0(E,s) \cdot  t   \int_{0}^{\infty} |w_{\pm}|^2~ dx~
\]
The claim then follows from combining these estimates and using
the triangle inequality.  
\end{proof}


\section{The Langer-Cherry transform} \label{LGSection}

We wish to analyse the behavior of the solutions to (\ref{ODE})
for $x$ near $x_E$ and for $t$ small. To do this, we will use a
transform to put the solution into a normal form. The transform that 
we will use was first considered by Langer \cite{Langer} and Cherry 
\cite{Cherry} and is a variant of the Liouville-Green transformation. 
See Chapter 11 in \cite{Olver}. 

As above, let $f_E= \mu - E \cdot \sigma$ where $\sigma$ is smooth 
with $\sigma'<0$ and $\lim_{x \rightarrow \infty} \sigma(x)=0$.
For $E \geq \mu/\sigma(0)$, there exists unique $x_E \in [0, \infty)$
such that $f_E(x_E)=0$.  In the present context, the Langer-Cherry 
transform is based on the function $\phi_E:[0, \infty) \rightarrow \R$
defined by 
\begin{equation}  \label{phiE}
  \phi_E(x) ~ =~\mbox{sign}(x-x_E) \cdot
\left| \frac{3}{2}~ \int_{x_E}^{x} |f_E(u)|^{\frac{1}{2}}~ du\right|^{\frac{2}{3}}.
\end{equation}

Before defining the Langer-Cherry transform, we 
collect some facts concerning $\phi_E$. 

\begin{lem}   \label{phiproperties}
Let $\Ucal= \left[\frac{\mu}{\sigma(0)},\infty \right)\times [0,\infty)$.
\begin{enumerate}
  \item  \label{phik} 
  The map $(E,x) \mapsto \phi_E(x)$ is smooth on $\Ucal$. 

  \item \label{phipos} 
 $\phi'_E(x)>0$ for each $(E, x) \in \Ucal$.
  \item \label{phiidentity}
   $ (\phi_E')^2 \cdot \phi_E= f_E$.
  \item \label{phif}  
    The map $(E,x)\mapsto f_E(x)/\phi_E(x)$ defined for $x\neq x_E$ extends to a smooth 
     map from $\Ucal$ to $\R^+$.   
  \item \label{phiasymptotics} 
   The limit
 \[ \lim_{x \rightarrow \infty}~ x^{-\frac{2}{3}} \cdot \phi_E(x)~ 
   =~ (3/2)^{\frac{2}{3}} \cdot \mu^\unt \]
  holds uniformly for $E$ in each compact subset of 
   $\left[\frac{\mu}{\sigma(0)},\infty \right)$.
 \end{enumerate}
\end{lem}

\begin{proof} 
We will use an alternative expression for $\phi_E$ from which 
several of these properties will directly follow.
Since $\sigma'(x)<0$ for all $x \in [\mu/\sigma(0), \infty)$, the map
$I: \Ucal \rightarrow \R$  
\[  I(E,u)~ =~
 \int_{0}^1 -E \cdot \sigma' \left(E,~ s\cdot u~ +~ (1-s) \cdot x_E \right)~ ds.
\]
is smooth and positive on $\Ucal$. The  map 
$\pi: \Ucal \rightarrow \R$ defined by
\begin{equation}\label{Defp}
\pi(E,x)~ =~  \int_{0}^{1} s^{\und} \cdot 
      I^{\und} \left(E,~ s\cdot x~ +~ (1-s) \cdot x_E \right)~ ds.
\end{equation} 
is also smooth and positive.

Note that since $f_E(x_E)=0$ and $f_E'(x)= -E \sigma'(x),$ 
the fundamental theorem of calculus gives that
\[
\mu~ -~ E \cdot \sigma(u)~ =~ (u-x_E) \cdot I(E,u).
\]
Direct computation shows that 
\begin{equation}  \label{Alternative}
 \phi_E(x)~ =~ (x-x_E) \cdot \left( \td \cdot \pi(E,x) \right)^{\dt}. 
\end{equation}
In particular, (\ref{phik}) holds.
 
Inspection of (\ref{phiE}) gives that (\ref{phipos}) holds for $x \neq x_E$.
For $x=x_E$, (\ref{phiE}) follows from a calculation using (\ref{Alternative}). 

 (\ref{phiidentity}) holds for $x \neq x_E$
via a direct calculation using (\ref{phiE}). Smoothness
gives the extension of this identity to $\Ucal$.

Note that $(E, \mu) \rightarrow f_E(x)$ and $(x,E) \rightarrow x_E$ are smooth on $\Ucal$.
Since $f_E(x_E)=0$ and $f_E'(x) >0$ for all $x \geq 0$,
the function $(E, x) \rightarrow f_E(x)/(x-x_E)$ is smooth,
and thus (\ref{phif}) follows from (\ref{Alternative}) and the positivity
of $p$. 

(\ref{phiasymptotics}) follows from (\ref{phiE}).
\end{proof}

\begin{defn}
Let $w: [0, \infty) \rightarrow \R$ and let $E \geq \mu/\sigma(0)$.  
Define the {\em Langer-Cherry transform of $w$ at energy $E$} 
to be the function 
\begin{equation}  \label{LCDefn}
 W_E =~ \left( (\phi_E')^{\frac{1}{2}}\cdot w \right )\circ \phi^{-1}_E. 
\end{equation}
\end{defn}

It follows from Lemma \ref{phiproperties} that the Langer-Cherry transform maps 
$C^{k}([0, \infty))$ to $C^{k}([\phi_E(0), \infty))$.
The importance of this transform is due to its effect on solutions 
to the ordinary differential equation (\ref{ODE}). In what follows
we let 
\begin{equation} \label{hDefn}
 \rho_E= (\phi_E')^{-\frac{1}{2}}. 
\end{equation}

\begin{prop} \label{wW}
Let $r: [0, \infty) \rightarrow \R$ and let  $w \in C^2([0,\infty))$.
Let $W_E$ be the Langer-Cherry transform of $w$ at energy $E$. 
Then $w$ satisfies
\[ t^2 \cdot w''~ -~  f_E\cdot w~ =~ -r \cdot \sigma \]
if and only if $W_E$ satisfies 
\begin{equation}\label{WODE}
 t^2 \cdot W_E''~ -~ y \cdot W_E~ 
   =~  -~ t^2 \cdot (\rho_E^3 \cdot \rho_E'') \circ \phi_E^{-1} \cdot W_E~ 
   - \left(\rho_E^3 \cdot r\cdot \sigma \right) \circ \phi_E^{-1}.
\end{equation}
\end{prop}

The proof is a straightforward but lengthy computation.
We include it for the convenience of the reader.\footnote{
  See also, for example, \S 11.3 in \cite{Olver}, where the function $\hat{f}$ is 
               related to $h_E$ by $\hat{f}\,=\,h^4$.}

\begin{proof}
We will suppress the dependence on $E$ from the notation.
Let $h= 1/\rho$. Then from (\ref{LCDefn}) we have 
\begin{equation}  \label{hw}
  W \circ \phi~ = h \cdot w. 
\end{equation}
Differentiation gives
\begin{gather*}
  (W'\circ \phi)\cdot \phi ' \,= \, h'\cdot w+h\cdot w'\\
  (W''\circ \phi)\cdot (\phi')^2+ (W'\circ \phi)\cdot \phi ''
  \,=\, h''\cdot w +2h'\cdot w' + h\cdot w''.
\end{gather*}
Solve the first equation for $W'\circ \phi$ and substitute the result for $W'$ 
in the second equation. One obtains
\begin{equation} \label{AfterSubstitution}
(W''\circ \phi)\cdot (\phi')^2 +\frac{\phi ''}{\phi'}\left( h' \cdot w + h \cdot w'\right )\,
=\, h''\cdot w +2h'\cdot w' + h\cdot w''.
\end{equation}
Since $h= (\phi_E')^{\frac{1}{2}}$, we have
\begin{equation}\label{defh}
\frac{\phi ''}{\phi'} \cdot h \,=\, 2 h',
\end{equation}
and thus, (\ref{AfterSubstitution}) simplifies to
\[
(W''\circ \phi) \cdot (\phi')^2\,+\,\frac{\phi''}{\phi'}\cdot h'\cdot w \,
=\, h''\cdot w\,+\,h\cdot w''.
\]
Using (\ref{defh}) and (\ref{hw}), we obtain 
\begin{equation}\label{wsecond}
(W''\circ \phi)\cdot (\phi')^2~ 
   +~ \left[- \frac{h''}{h}~ +~ 2 \cdot \frac{(h')^2}{h^2}\right] \cdot W\circ \phi~  
   =~ h\cdot w''. 
\end{equation} 
Since $(\phi')^2 = \rho^{-4}$, $h= 1/\rho$,  and 
\[  \rho''~ =~  \frac{-h''}{h^2}~ +~ 2 \cdot \frac{(h')^2}{h^3} \]
we find that
\begin{equation}\label{wthird}
  \rho^{-3} \cdot (W''\circ \phi)~ 
   +~ \rho'' \cdot (W\circ \phi)~ =~
    w''.
\end{equation} 

We now subtitute the left hand side of (\ref{wthird}) into
the first term of (\ref{ODE}) to obtain
\[
  t^2\left( \rho^{-3} \cdot (W''\circ \phi)~ 
   +~ \rho'' \cdot (W\circ \phi) \right)~ -~
 f \cdot w~ =~  - r\cdot \sigma.
\]
By (\ref{LCDefn}), we have $w= \rho \cdot W \circ \phi$,
and by part (\ref{phiidentity}) of Lemma \ref{phiproperties}, 
we have $f= \rho^{-4} \cdot \phi$.
Thus, 
\[
  t^2\left( \rho^{-3} \cdot (W''\circ \phi)~ 
   +~ \rho'' \cdot (W\circ \phi) \right)~ -~
   \rho^{-3} \cdot \phi \cdot (W \circ \phi)~ =~  - r\cdot \sigma.
\]
By multiplying by $\rho^3$, we find that
\[  t^2\cdot (W''\circ \phi )~ 
-~\phi \cdot (W \circ \phi)~ =~  
  -~ t^2 \cdot \rho^3 \cdot \rho'' \cdot (W\circ \phi)~ 
   - \rho^3 \cdot r\cdot \sigma.
\]
The claim follows from making the change of variable $y = \phi(x)$. 
\end{proof}

In the analysis that follows, we will treat the right 
hand side of (\ref{WODE}) as an error term for $t$ and $r$ small. 
The following estimates will help justify this treatment.

\begin{lem} \label{rhoBounded}
Let $K \subset [\mu/\sigma(0), \infty)$ be compact.
There exists $C>0$ such that if $x \geq 0$ and $E \in K$, 
then 
\[  \left| \frac{1}{\rho_E(x)} \right|~ \leq~ C \cdot x^{-\frac{1}{6}}, \]
and 
\[  \left| \rho_E(x) \right|~ \leq~ C \cdot \left(1+ x^{\frac{1}{6}}\right). \]
Moreover, there exists $\nu$ such that  
\[  \left| \rho_E''(x)   \right|~ \leq~ C \cdot \left(1+ x^{\nu} \right). \]
The exponent $\nu$ depends only on $\sigma$. 
The constant $C$ depends only on $\mu$, $\sigma$, and $K$.
\end{lem}

\begin{proof}
By Lemma (\ref{phiidentity}) of Lemma \ref{phiproperties}, 
we have $\rho= (\phi/f)^{\frac{1}{4}}$. Hence since $\lim_{x \rightarrow \infty} f_E(x)=\mu$,
we find from part (\ref{phiasymptotics}) that 
\[ \lim_{x \rightarrow \infty} \rho_E \cdot x^{-\frac{1}{6}}~
   =~  \left(\frac{2}{3 \mu}\right)^{\frac{1}{6}}
\] 
uniformly for $E \in K$. The first two estimates follow. 

To prove the last estimate, one computes using $f=(\phi')^2 \cdot \phi$ that
\[ \rho''~ =~ -\frac{5}{16} \frac{f^{\frac{3}{4}}}{\phi^{\frac{11}{4}}}~ 
 -~ \frac{1}{4} \frac{\phi^{\frac{1}{4}}\cdot f''}{f^{\frac{5}{4}}}~ 
 +~ \frac{5}{16} \frac{\phi^{\frac{3}{4}}\cdot (f')^2}{f^{\frac{9}{4}}}~ 
\]
By  part (\ref{phiasymptotics}) of Lemma \ref{phiproperties}, 
both $\phi$ and $1/\phi$ have polynomial growth that is uniform for $E \in K$. 
By assumption, $\sigma''$ has at most polynomial growth,
and hence, by integration, the function $\sigma'$ also has at most polynomial growth. 
Therefore, $f''$ and $f'$ both have polynomial growth that is uniform over $K$. 
Therefore, since $\lim_{x \rightarrow \infty} f_E(x)=\mu>0$, we find that
$\rho''$ has uniform polynomial growth. 
\end{proof}

\begin{lem}  \label{GEstimate}
Let $I \subset [0, \infty)$ be a compact interval 
and let $K \subset [\mu/\sigma(0), \infty)$ be a compact set. 
There exists a constant $C$ such that for each $E \in K$ such that 
if $w$ is a solution to (\ref{ODE}) and and $W_E$ is
the Langer-Cherry transform of $w$ at energy $E$, then 
we have 
\[
   \int_{\phi_E(I)} \left|t^2 \cdot W_E''(y) -y \cdot W_E(y) \right|^2~ dy~
 \leq~  C \cdot \left(\|r\|_{\sigma}^2~ +~ t^4  \int_{0}^{\infty} |w|^2~ dx \right).   
\]
The constant $C$ depends only on $\mu$, $\sigma$, $I$, and $K$.
\end{lem}
\begin{proof}
For each continuous function $F: I \rightarrow \R$,
let $|F|_{\infty}= \sup\{|F(x)|~ |~ x \in I\}$. 
We perform the change of variables $y=\phi_E(x).$
Since $\phi_E'=\rho_E^{-2}$, we have  $dy =\rho_E^{-2} \cdot dx$,
and thus by (\ref{LCDefn}) and (\ref{hDefn})   
\begin{equation}  \label{ChangeOfVariables}
  |W|^2~ dy~  =~  \rho_E^{-4} \cdot |w|^2~ dx. 
\end{equation} 
Therefore
\begin{equation*}
   \int_{\phi_E(I)}  \left |(\rho_E^{3} \cdot \rho''_E)\circ \phi_E^{-1}\right|^2 \cdot |W|^2~ dy~ \leq~
    |\rho_E|_{\infty}^2 \cdot |\rho''_E|^2_{\infty} 
     \cdot  \int_{I} |w|^2~ dx,
\end{equation*}
and 
\begin{equation*}
   \int_{\phi_E(I)}  \left |(\rho_E^3 \cdot r \cdot \sigma) \circ \phi_E^{-1}\right |^2~ dy~ 
   \leq~
   |\rho_E^{2}\cdot \sigma |_{\infty} \cdot   \|r\|^2_{\sigma}.
\end{equation*}
The claim then follows from squaring and integrating (\ref{WODE}) 
and applying the above estimates. 
\end{proof}

Suppose $w$ is an eigenfunction of $a^\mu_t,$ and denote by $\lambda$ its eigenvalue. 
If we perform the Cherry-Langer transform at energy $E=\lambda$ then $r=0$ and hence 
the conclusion of Lemma \ref{GEstimate} is stronger. Actually, we will need 
the following strengthening of Lemma \ref{GEstimate} which treats the case when
$w$ is an eigenfunction of $a_t^\mu$ but $E$ is close to but not necessarily 
exactly the corresponding eigenvalue. 
 
\begin{lem}\label{LemSpecG}
Let $K \subset [\mu/\sigma(0), \infty)$ be compact.
There exists a constant $C_K$ such that if $t<1$, 
$w$ is an eigenfunction of $a_t^\mu$ with eigenvalue $\lambda \in K$, 
and $W$ is the Cherry-Langer transform of $w$ at energy $E\in K$, 
then 
\begin{equation}\label{SpecEstG}
\int_{\phi_E(0)}^\infty  \left|t^2 \cdot W''~ -~ y \cdot W\right|^2~ dy~ \leq~
  C_K \cdot \left( |\lambda~ -~ E|^2~  +~  t^4 \right) \int_{0}^\infty~ w^2~ dx.
\end{equation}
\end{lem}

\begin{proof}
Note that since $-t^2 \cdot w''\,+\, (\mu- \lambda \cdot \sigma)\cdot w=0$, 
the function $w$ satisfies
\[ 
        -t^2 \cdot w''~ +~ f_E \cdot w~ =~ r.
\]
with $r= (E-\lambda) \cdot \sigma \cdot w$. Therefore we may apply Proposition \ref{wW}.
In particular, it suffices to bound the integrals of the squares of the terms appearing on
the right hand side of (\ref{WODE}). 

By Lemma \ref{rhoBounded} there exists $\nu_1$ and $C_1$ (depending only on $K$)
such that
\[  |\rho_E(x)|^{-4} \cdot |\rho_E^3 \cdot \rho_E''(x)|^2~ \leq~ C_1 \cdot (1~ +~ x^{\nu_1}). \]
Hence by changing variables (recall that $W^2 dy=\rho_E^{-4}w^2 dx$) we find that
\[  \int_{\phi_E(0)}^{\infty} \left|(\rho_E^3 \cdot \rho_E'')\circ \phi_E^{-1}\right|^2 
    \cdot |W(y)|^2~ dy~  
  \leq~ C_1 \int_{0}^{\infty}  |w(x)|^2 \cdot (1~ +~ x^{\nu_1})~ dx. 
\]
Since $w$ is an eigenfunction, we can apply 
Lemma \ref{L2Eigenfunction}. By fixing $s=\mu/2$, we obtain a constant
$C_2$---depending only on $K$---such that 
\[ \int_{3 x_E^s}^{\infty}  |w(x)|^2 \cdot (1~ +~ x^{\nu_1})~ dx~ \leq~
   C_2 \cdot t  \int_{x_E^s}^{\infty}  |w(x)|^2~ dx
\] 
Let $x^*= \sup\{ x_E^s~ |~ E \in K, s=\mu/2\}$. 
Then  
\[ \int_{0}^{3 x_E^s}  |w(x)|^2 \cdot (1~ +~ (3x)^{\nu_1})~ dx~ 
   \leq~ \left(1+ (3x^*)^{\nu_1} \right)  \int_{0}^{\infty}  |w(x)|^2~ dx.  \]
In sum, if $t\geq 1$, then we have a constant $C_3$ such that
\[ \int_{\phi_E(0)}^{\infty} |\rho_E^3 \cdot \rho_E''\circ \phi^{-1}(y)|^2 \cdot |W(y)|^2~ dy~  
  \leq~ C_3   \int_{0}^{\infty}  |w(x)|^2~ dx
\]    

A similar argument shows that there exists $C_4$---depending
only on $K$---such that
\[ \int_{\phi_E(0)}^{\infty}
\left|(\rho_E^3 \cdot \sigma^2)\circ \phi_E^{-1}(y) \cdot W(y) \right|^2~ dy~  
  \leq~ C_4   \int_{0}^{\infty}  |w(x)|^2~ dx. 
\]
By putting these estimates together we obtain the claim.
\end{proof}

We now prove a lemma that allows us 
to compare the (standard) $L^2$ norm of $w$ 
with that of its Cherry-Langer transform $W$ in the limit as 
$E$ tends to $\mu/\sigma(0)$ 
and $t$ tends to $0.$
  
\begin{lem} \label{wToWIntegral}
Let $q:[0, \infty) \rightarrow \R$ be a positive continuous
function of at most polynomial growth.  
Given $\epsilon>0$, there exists $\delta>0$ such that  
if $t < \delta$,  $E< \mu/\sigma(0)+\delta$,
and $w_{\pm}$ is an eigenfunction of $a_{\mu}^t$ with
eigenvalue $\lambda_{\pm} \leq E$, 
then
\begin{equation}\label{NormComp}
\left| \int_{\phi_E(0)}^{\infty} W_+ \cdot W_-~ dy~ 
-~ \frac{1}{\rho_E^{4}(0) \cdot q(0)} \int_0^{\infty} w_+ \cdot w_- \cdot q~ dx \right|~
   \leq~ \epsilon \cdot \|w_+\|_1 \cdot \|w_-\|_1.
\end{equation}
where $W_{\pm}$ is  the Langer-Cherry transform of $w_{\pm}$ at energy $E$.
\end{lem}

\begin{proof}
Changing variables gives
\[
    \int_{\phi_E(0)}^\infty W_+ \cdot W_-~  dy~ 
    =~ \int_0^\infty w_+ \cdot w_- \cdot \rho_E^{-4}(x)~ dx.
\]
By Lemma \ref{rhoBounded}, the function $\rho_E^{-4}$ is bounded, and hence
we can apply Proposition \ref{BiLinEst}. In particular, choose 
$\delta_1>0$ so that if $E < \mu/\sigma(0) + \delta_1$, then 
$\alpha_p(E,\mu/2)< \epsilon/4$ and choose $\delta_2\leq \delta_2$ so that if
$t < \delta_2$, then $\beta(\delta_1, \mu/2) \cdot t < \epsilon/4$.
Thus, if $E< \mu/\sigma(0) + \delta_2$ and $t< \delta_2$, then 
\[  \left| \int_{\phi_E(0)}^\infty W_+ \cdot W_-~  dy~ 
    -~ \rho_E^{-4}(0) \int_0^\infty w_+ \cdot w_-~ dx \right|~ \leq~
   \frac{\epsilon}{2} \cdot \|w_+\| \cdot \|w_-\|.
\]
In a similar fashion we can apply Lemma \ref{BiLinEst} to find 
$\delta \leq \delta_2$ so that if $E< \mu/\sigma(0) + \delta$ and $t< \delta$,
then 
\[  \left| \int_{0}^\infty w_+ \cdot w_- \cdot q~ dy~ 
    -~ q(0) \int_0^\infty w_+ \cdot w_-~ dx \right|~ \leq~
    \frac{\epsilon}{2} \cdot \rho^{4}_E(0) \cdot q(0) \cdot \|w_+\|_1 \cdot \|w_-\|_1.
\]
The claim follows. 
\end{proof}


\section{Airy approximations} \label{AirySection}

In this section we analyse solutions to the inhomogeneous equation
\begin{equation}  \label{WODE2}
  t^2 \cdot W''(y)~ -~ y \cdot W(y)~ =~ g(y). 
\end{equation}
To do this, we will use a solution operator, $\tilde{K}_t$, for the 
associated homogeneous equation
\begin{equation}  \label{Homogeneous}
   t^2 \frac{\partial^2}{\partial y^2}~ W_0 -~ y \cdot W_0~ 
   =~  0.
\end{equation}
Note that $W_0$ is a solution to (\ref{Homogeneous}) if and only if
${\rm A}(u)= W_0(t^{\frac{2}{3}} \cdot u)$ is a solution to 
the {\em Airy equation}
\begin{equation}  \label{AiryEquation1}
  \frac{\partial^2}{\partial u^2}~ A~ -~ u \cdot A~ =~ 0.
\end{equation}

Using, for example, the method of variation of constants, one can construct 
an  integral kernel $K$ for an `inverse' of the operator 
$A(u) \mapsto A''(u) - u \cdot A(u)$ in terms of Airy functions. 
We give the construction of $K$ as well as its basic properties in Appendix 
\ref{AiryAppendix}.  By rescaling (or by direct construction)
we obtain an integral kernel for the operator 
$A(x) \mapsto t^2 \cdot A''(x) - x \cdot A(x)$.
To be precise, define
\[  \tilde{K}_t(y,z)~ =~ t^{-\frac{4}{3}}  \cdot 
      K \left( t^{-\frac{2}{3}} \cdot y,~ t^{-\frac{2}{3}} \cdot z \right), \] 
where $K$ is the integral kernel constructed in Appendix \ref{AiryAppendix}.
\begin{lem}\label{LemSol}
Let $\infty < a  \leq b \leq \infty$.
For each locally integrable $g: [a,b] \rightarrow \R$ 
of at most polynomial growth, the function 
\begin{equation} \label{Kernel3}
 y~ \mapsto~ \int_{a}^{b} \tilde{K}_t(y,z) \cdot g(z)~ dz~ 
\end{equation}
is a solution to (\ref{WODE2}).
\end{lem}

\begin{proof}
This follows from Lemma \ref{AiryKernel} or directly from the  
variation of constants construction.   
\end{proof}

The following estimate is crucial to the  
proof of Proposition \ref{CrucialProposition}.

\begin{lem}  \label{TransitionEstimate}
Let $g:\R \rightarrow \R$ be continuous. 
For each $\infty < a < 0< b$, there 
exist constants $C$ and $t_0>0$ such that if $t<t_0$ and
$W$ satisfies (\ref{WODE2}), then 
\begin{equation}  \label{Transition1}
    \int_{a}^{0} W^2~ \leq~ 
   C \cdot \left(  \int_{a}^{\frac{a}{2}}  W^2~
  +~  t^{-\frac{5}{3}} \int_{a}^{b} g^2~
\right).
\end{equation}
and 
\begin{equation}  \label{Transition2}
    \int_{0}^{\frac{b}{2}} W^2~ \leq~ 
   C \cdot \left(  t^{\frac{1}{3}} \int_{a}^{\frac{a}{2}}  W^2~ +~
      \int_{\frac{b}{2}}^{b}  W^2~
  +~  t^{-\frac{5}{3}} \int_{a}^{b} g^2~
\right).
\end{equation}
The constants $C$ and $t_0$ can be chosen to depend continuously upon $a$ and $b$.
\end{lem}
\begin{proof}
Define $W_0$ on $[a,b]$ by 
\[  W_0(y)~ =~  W(y)~ -~ 
  \int_{a}^{b} \tilde{K}_t(y,z) \cdot g(z)~ dz .\]
Using Lemma \ref{LemSol} and linearity, $W_0$ 
is a solution to (\ref{Homogeneous}). 

Using the Cauchy-Schwarz-Bunyakovsky inequality, we find that
\begin{equation} \label{FirstL2}
  |W(y)~ -~ W_0(y)|^2~ =~
   \left(\int_{a}^{b}\left|\tilde{K}_t(y,z) \right|^2 dz  \right) \left(
   \int_{a}^{b}\left| g(z) \right|^2 dz \right).   \
\end{equation}
A change of variables gives
\begin{equation} \label{ChangeBack}
   \int_{a}^{b}\int_{a}^{b}\left|\tilde{K}_t(y,z) \right|^2 dy~ dz~
    =~  t^{-\qt} \int_{t^{-\frac{2}{3}}a}^{t^{-\frac{2}{3}}b} 
   \int_{t^{-\frac{2}{3}}a}^{t^{-\frac{2}{3}}b} 
  \left|K(u,v) \right|^2 du~ dv. 
\end{equation}
By Lemma \ref{HilbertSchmidtEstimate} in Appendix \ref{AiryAppendix}, the latter integral 
is less than $C_{{\rm Airy}} \cdot \sqrt{\delta} \cdot t^{-\frac{1}{3}}$
where $C_{{\rm Airy}}$ is a universal constant and 
 $\delta= \max \{ |a|, b\}$. Therefore, by
integrating (\ref{FirstL2}) over an interval $I \subset [a,b]$ and
substituting (\ref{ChangeBack}), we find that
\begin{equation}\label{RemainderOverA}
 \|W-W_0\|_{I}^2~ \leq~ 
 C_0 \cdot t^{-\frac{5}{3}} \cdot \|g\|^2_{[a,b]}.
\end{equation}
where $C_0=  C_{{\rm Airy}} \cdot \delta$ and $\|\cdot\|_J$ denotes the 
$L^2$-norm over the interval $J$.
In particular, by the triangle inequality we have 
\[ \|W\|_{I}~ \leq~ \|W_0\|_I~ +~
  C_0^{\frac{1}{2}} 
\cdot t^{-\frac{5}{6}} \cdot \|g\|_{[a,b]},  \]
and hence
\begin{equation} \label{W}
 \|W\|_{I}^2~ \leq~ 2 \cdot  \|W_0\|_I^2~ +
  2 \cdot C_0\cdot t^{-\frac{5}{3}} \cdot \|g\|^2_{[a,b]}.
\end{equation}
Similarly,
\begin{equation} \label{WNought}
\|W_0\|_{I}^2~ \leq~ 2 \cdot  \|W\|_I^2~ +
  2 \cdot C_0 \cdot t^{-\frac{5}{3}} \cdot \|g\|^2_{[a,b]}.
\end{equation}

The function $u \mapsto W_0(t^{\frac{2}{3}} \cdot u)$ 
satisfies the Airy equation (\ref{AiryEquation}). Hence, it follows
from Lemma \ref{TransitionRegion} (in which $s$ is replaced by $t^{-\frac{2}{3}}$)
that there exist constants $M$ and $t_0>0$---depending continuouslu on $a$ and $b$---such 
that if $t \leq t_0$, then 
\begin{equation}  \label{ProtoNeg}
  \int_{a}^0 W_0^2~ dy~ \leq~ M   \int_{a}^{\frac{a}{2}} W_0^2~ dy~ 
\end{equation}
and 
\begin{equation}  \label{ProtoPos}
   \int_{0}^{\frac{b}{2}}~ W_0^2~ dy~
   \leq~ M  \left( t^{\frac{1}{3}}\int_{a}^{\frac{a}{2}} W_0^2~ dy~
  +~ \int_{\frac{b}{2}}^{b} W_0^2~ dy \right).
\end{equation}

By combining (\ref{ProtoPos}) with (\ref{W}) and (\ref{WNought}),
we obtain  (\ref{Transition1}). By combining (\ref{ProtoNeg}) 
with (\ref{W}) and (\ref{WNought}), we obtain  (\ref{Transition2}). 
\end{proof}


\section{A non-concentration estimate}\label{crucial}

Fix $\mu $ and $\sigma$ and let $a_t^{\mu}$ be the family of 
quadratic forms defined as in \S \ref{aSection}. 
The purpose of this section is to prove the following 
non-concentration estimate---see Remark \ref{RemNonC}---that 
is crucial to our proof of generic spectral simplicity. 

\begin{prop}\label{CrucialProposition}
Let $K$ be a compact subset of 
$(\mu \cdot \sigma(0)^{-1},\infty)$, and $C>0$.
There exist constants $t_0>0$ and $\kappa>0$ such that
if $E \in K$, if $t < t_0$, and if for each $v \in {\rm dom}(a_t^{\mu})$,
the function $w$ satisfies
\begin{equation}  \label{rHypothesis}
  \left|a_{t}^{\mu}(w,v)~ -~ E \cdot  \langle w, v\rangle_{\sigma} \right|~ 
\leq~  C \cdot t \cdot \|w\|_{\sigma} \cdot \|v\|_{\sigma},
\end{equation}
then 
\begin{equation}\label{DerEst}
\int_0^{\infty} \left( E \cdot \sigma(x)~ -~ \mu\right) \cdot |w(x)|^2~ 
      dx~ \geq~ 
  \kappa \cdot \| w\|_\sigma^2.
\end{equation}
The constants $t_0$ and $\kappa$ depend only upon
$K$, $C$, $\mu$, and $\sigma$.
\end{prop}

Observe that, in contrast to previous estimates, 
Proposition \ref{CrucialProposition}
is concerned with so-called {\em non-critical} energies, those values of 
$E$ that are strictly greater than the threshold $\mu/\sigma(0)$.

\begin{remk} \label{QuasimodeRemark}
Estimate (\ref{rHypothesis}) is a special case 
of an estimate of the following form:  For all $v \in {\rm dom}(a_{\mu,t})$,
\begin{equation}\label{QMHypot}
\left| a_{t}^\mu(w,v)~ -~ E_t \cdot \langle w,v\rangle_\sigma \right|~ \leq~ 
t^\rho \cdot \|w\|_\sigma\cdot \|v\|_\sigma.
\end{equation}
By the Riesz representation theorem, estimate (\ref{QMHypot})
is equivalent to equation (\ref{rDefined}) 
with $r$ such that $\| r\| \leq t^\rho \cdot \|w\|.$ 
In other words, a sequence $w_n$ satisfying (\ref{QMHypot}) 
is what we have called a quasimode of order $\rho$ at energy $E_0$. 
\end{remk}

\begin{remk}\label{RemNonC}
Suppose that $w_n$ is a sequence of eigenfunctions of $a^{\mu}_{t_n}$ with 
$t_n$ tending to zero as $n$ tends to infinity.  
Then, by Lemma \ref{ExponentialDecay}, each $w_n$ decays exponentially in the region 
$\{x~ |~  E \cdot \sigma(x) -\mu <0\}$ and the rate of decay increases as $n$ increases. 
In particular, we can use  Proposition \ref{L2Eigenfunction} to prove that the measure 
$|w_n(x)|^2\,dx$ concentrates in the `classically allowed region' 
$\{x~ |~  E \cdot \sigma(x)-\mu \geq 0\}.$  Proposition \ref{CrucialProposition}
is a twofold strengthening of this latter statement: We prove that
if $E$ is not critical then $|w_n(x)|^2\,dx$ does not concentrate on  
$\{x~ |~  E \cdot \sigma(x)-\mu = 0\}$, and we prove that this also 
holds true for a quasimode of order $1$. 

Observe that estimate (\ref{DerEst}) for eigenfunctions 
could be obtained using a contradiction argument which is standard in the study of 
semiclassical measures. (See \cite{LucMeasures} for closely related topics). 
However, we believe that this method fails for first order quasimodes.
\end{remk}

\begin{proof}[Proof of Proposition \ref{CrucialProposition}]
Let $E> \mu/\sigma(0)$.  Then $f_E(0)<0$ and since
$f_E= \mu- E \cdot \sigma$ is strictly increasing
with $\lim_{x \rightarrow \infty}f_E(x)=\mu$, there exists a unique $x_E>0$
such that $f_E(x_E)=0$. Since $f_E$ changes sign at $x_E$, 
we have
\begin{equation} \label{Split}
     \int_{0}^{\infty}(-f_E) \cdot  w^2~ dx~  
    =~    \int_{0}^{x_E}|f_E| \cdot  w^2~ dx~ 
    -~   \int_{x_E}^{\infty}|f_E| \cdot  w^2~ dx.  
\end{equation}
Thus, by Lemmas \ref{Neg} and \ref{Pos} below, 
there exist constants $C^+$, $c^->0$ and $t^*>0$
such that if $t< t^*$, then 
\begin{equation}\label{FinalBoth}
  \int_{0}^{\infty} (-f_E) \cdot w^2~ dx~ \geq~
 c^- \int_{0}^{x_E} w^2~ dx~
-~  C^+ \cdot t^{\frac{1}{3}}   \int_0^\infty w^2 dx.
\end{equation}
Thus, if $t< t_0= (c^-/2C^+)^{3}$, then we have (\ref{DerEst})
with $\kappa= c^-/2$.
\end{proof}

\begin{lem} \label{Pos}
There exist constants $C^+$ and $t^+>0$ so that if $t < t^+$, then
\begin{equation}  \label{FinalPos}
\int_{x_E}^{\infty} |f_E| \cdot w^2~ dx~  \leq~  C^+  \cdot 
   t^{\frac{1}{3}} \int_{0}^{x_E} w^2~ dx.
\end{equation}
\end{lem}

\begin{lem} \label{Neg}
There exist constants $c^->0$, and $t^->0$  so that if $t < t^-$, then
\begin{equation} \label{FinalNegative}
   \int_{0}^{x_E}|f_E| \cdot  w^2~ dx~ \geq~
    c^- \int_{0}^{\infty} w^2~ dx.
\end{equation}
\end{lem}

The proofs of Lemma \ref{Neg} and \ref{Pos}
are based on estimates provided in sections \ref{ExponentialEstimates}, 
\ref{LGSection}, and \ref{AirySection}.  In preparation for these proofs
we provide the common context.

First note that the Riesz representation theorem provides 
$r \in {\mathcal H}_{\sigma}$ so that for all $v \in {\rm dom}(a_t)$ 
\[  \left|a_{t}(w,v)~ -~ E \cdot  \langle w, v\rangle_{\sigma} \right|~ 
   =~  \langle r, v \rangle_{\sigma}.
\]
where $\|r\|_{\sigma} \leq C_0\cdot t \cdot \|w\|_{\sigma}$.

Let $W$ denote the Langer-Cherry transform of $w$ at 
energy $E$ (see \S \ref{LGSection}). In particular, 
\[  W~ =~ \left((\phi_E')^{\frac{1}{2}} \cdot w\right) \circ \phi_E^{-1} \]
where $\phi_E$ is defined by (\ref{phiE}).
By Proposition \ref{wW}, the function $W$ satisfies (\ref{WODE2}) 
with $g$ equal to the right-hand side of equation (\ref{WODE}).

As a last preparation for the proofs, we define the endpoints 
of the intervals over which we will apply the estimates from the preceding sections. 
Let $x_E^+$ be defined by $f_E(x_E^+)= \mu/2$. 
In other words, $x_E^+=x_E^{\mu/2}$ where $x_E^s$ is defined in (\ref{xs}). 
Define
\[ y_E^+~ =~  2 \cdot \phi_E(x_E^+) \]
and 
\[ y_E^-~ =~  \phi_E(0). \] 
Since $\sigma$ is decreasing, we have $0 < x_E < x_E^+$,
and hence since $\phi_E$ is strictly increasing, we have $y_E^- < 0 < y_E^+$. 
It follows from Lemma \ref{phiproperties} and Remark \ref{RmkSmoothness} 
that $y^+_E$ and $y_E^-$ depend smoothly on $E$.


\begin{proof}[Proof of Lemma \ref{Pos}]
Since $\sigma$ is decreasing, we have $\sup \left\{|f_E|~ |~ x \geq x_E \right\}=\mu$,
and thus 
\begin{equation}  \label{LowerF}
  \int_{x_E}^{\infty}|f_E| \cdot  w^2~ dx~   \leq~
    \mu \int_{x_E}^{\infty}  w^2~ dx.  
\end{equation}
Since $\phi_E\left([x_E,x_E^+]\right)= \left[0, \frac{y_E^+}{2} \right]$
and $W^2 \cdot dy= (\phi_E')^{2} \cdot w^2 \cdot dx$,
we have
\begin{equation} \label{PosSmall}
  \int_{x_E}^{\infty}  w^2~ dx~ \leq~ C_1~  \int_{0}^{\frac{y_E^+}{2}} W^2~ dy~
   +~ \int_{x_E^+}^{\infty} w^2~ dx,~
\end{equation}
where $C_1\,=\, \max \{(\phi'_E(x))^{-2}~ |~  E\in K, x\in [x_E,x_E^+] \}$.

By Lemma \ref{TransitionEstimate} there exist constants $C_E$ and $t_E>0$
so that if $t< t_E$, then 
\begin{equation}  \label{CE}
 \int_{0}^{\frac{y_E^+}{2}} W^2~ dy~ \leq~   
C_E \left( t^{\frac{1}{3}} \int_{y_E^-}^{\frac{y_E^-}{2}}  W^2~ dy~ +~
      \int_{\frac{y_E^+}{2}}^{y_E^+}  W^2~ dy~ +~ t^{-\frac{5}{3}} \int_{y_E^-}^{y_E^+} 
    g^2~ dx  \right).
\end{equation}
The constants $C_E$ and $t_E$ depend continuously on $E$ and hence 
$C_2= \sup \{ C_E~ |~ E \in K\}$ is finite and $t_2=\inf\{t_E~ |~ E \in K\}$
is positive. Since $W^2 \cdot dy= (\phi_E')^{2} \cdot w^2 \cdot dx$,
we have 
\begin{equation} \label{DoubleW}
  t^{\frac{1}{3}} \int_{y_E^-}^{\frac{y_E^-}{2}}  W^2~ dy~ +~
      \int_{\frac{y_E^+}{2}}^{y_E^+}  W^2~ dy~
\leq~  C_3 \cdot \left( t^{\frac{1}{3}} \int_{0}^{\infty}  w^2~ dx~ +~
      \int_{x_E^+}^{\infty}  w^2~ dx \right)
\end{equation}
where $C_3:= \sup~ \left\{  (\phi_E'(x))^{2}~ 
|~ E \in K,~ x \in  \left[0, \phi^{-1}(y_E^+) \right] \right\}$.

By Lemma \ref{GEstimate}, there exists a constant $C^*$ so that
\begin{equation}  \label{gEstimate} 
 \int_{y_E^-}^{y_E^+} g^2~ dy~
 \leq~  C^* \cdot t^2~  \int_{0}^{\infty} w^2~  dx.   
\end{equation}
By substituting (\ref{DoubleW}) and (\ref{gEstimate}) into (\ref{CE})
we find that if $t< t_2$, then 
\begin{equation}  \label{W+}
 \int_{0}^{\frac{y_E^+}{2}} W^2~ dy~ \leq~
     C_4 \cdot t^{\frac{1}{3}} \int_0^\infty  w^2~ dx~ +~
     C_5 \int_{x^+_E}^{\infty}  w^2~ dx~
\end{equation}
where $C_4=C_2 \cdot (C_3+ C^*)$ and $C_5= C_2\cdot C_3$.
By Lemma \ref{ExponentialEstimate}, there 
exists a constant $C_6$ so that if $t<1$, then 
\begin{equation}   \label{XLargeBound}
     \int_{x_E^+}^{\infty} w^2~ dx~ \leq~  
     C_6 \cdot t^{\frac{1}{3}}~  \int_{0}^{\infty} w^2~ dx.
\end{equation}
By  combining  (\ref{PosSmall}), (\ref{W+}), and (\ref{XLargeBound}),
we find that if $t< t_3:= \min\{1,t_2\}$, then  
\begin{equation}  \label{1/3}
 \int_{x_E}^{\infty} w^2~ dx~ 
\leq~  C_7 \cdot  t^{\frac{1}{3}} \int_0^\infty  w^2~ dx.
\end{equation}
where $C_7= C_1\cdot C_4+ C_1 \cdot C_5 \cdot C_6 + C_6$.
Finally, split the integral on the right hand side of (\ref{1/3})
into the integral over $[0,x_E]$ and the integral over $[x_E,\infty)$.
Then subtract the latter integral from both sides of (\ref{1/3}).
It follows that if $t< \min\{ t_3, (2C_7)^{-3}\}$, then 
\[  \frac{1}{2} \int_{x_E}^{\infty} w^2~ dx~ 
\leq~  C_7 \cdot  t^{\frac{1}{3}} \int_0^{x_E}  w^2~ dx.
\]
The claim then follows by combining this with (\ref{LowerF}).
\end{proof}

We have the following corollary of the proof.

\begin{coro} \label{UpDownEstimate}
There exist constants $C'$ and $t'>0$ such that if $t<t'$
\[  \int_{x_E}^{\infty} w^2~ dx~ 
\leq~  C' \cdot  t^{\frac{1}{3}} \int_0^{x_E}  w^2~ dx.
\]
\end{coro}

\begin{proof}[Proof of Lemma \ref{Neg}]
Since $W^2 \cdot dy= (\phi_E')^{2} \cdot w^2 \cdot dx$ we have
\[ 
   \int_{0}^{x_E}|f_E| \cdot  w^2~ dx~ 
 \geq~  c_1~ \int_{y_E^-}^0 \left|f_E \circ \phi_E^{-1} \right| \cdot W^2~ dy 
\]
where $c_1 = \inf \left\{  \phi'_E(x)^{-2}~ |~ E \in K,~ x \in \left[0, x_E \right] \right\}$.
Since $f_E\circ\phi_E^{-1}$ is negative and increasing 
on $[y_E^-, y_E^-/2]$, we have  
\[  \int_{y_E^-}^{\frac{y_E^-}{2}} \left| f_E \circ \phi_E^{-1} \right| \cdot W^2~ dy~
     \geq~       c_2~ \int_{y_E^-}^{\frac{y_E^-}{2}} W^2~ dy.
\]
where $c_2 =  \inf \left\{ \left. 
 \left|f_E \circ \phi_E^{-1} \left(y_E^-/2 \right) \right|~ \right|~ E \in K
    \right\}$. 
Putting these two estimates together we have 
\begin{equation}  \label{delta}
  \int_{0}^{x_E}|f_E| \cdot  w^2~ dx~   
  \geq~  c_1 \cdot c_2~ \int_{y_E^-}^{\frac{y_E^-}{2}} W^2~ dy.
\end{equation}
It follows from Lemma \ref{phiproperties} that $c_1$ and $c_2$ are both positive.

By Lemma \ref{TransitionEstimate}, there exist
constants $C_E$ and $t_E>0$ so that if $t<t_E$, then 
\begin{equation}  \label{TransitionRecall}
    \int_{y_E^-}^{0} W^2~ \leq~ 
   C_E \cdot \left(  \int_{y_E^-}^{\frac{y_E^-}{2}}  W^2~
  +~  t^{-\frac{5}{3}} \int_{y_E^-}^{y_E^+} g^2~
\right).
\end{equation}
Moreover, $C_E$ and $t_E$ depend continuously on $E$, and hence 
the constants $c_3= \sup \{1/C_E~ |~ E \in K\}$ and $t_1=\inf\{t_E~|~ E \in K\}$
are both postive. By manipulating (\ref{TransitionRecall}) we find that
\begin{equation} \label{TransitionRecall2}  
     \int_{y_E^-}^{\frac{y_E^-}{2}}  W^2~ \geq~
   c_3 \int_{y_E^-}^{0} W^2~ 
  -~  t^{-\frac{5}{3}} \int_{y_E^-}^{y_E^+} g^2
\end{equation}
for each $t<t_1$.

By combining (\ref{delta}), (\ref{TransitionRecall2}), and (\ref{gEstimate})
we find that for $t < t_1$
\begin{equation} \label{Melange}
  \int_{0}^{x_E}|f_E| \cdot  w^2~ dx~   \geq~
 c_4~ \int_{y_E^-}^{0} W^2~ dy~ -~ C \cdot t^{\frac{1}{3}}~ \int_0^\infty w^2~ dx 
\end{equation}
where $c_4= c_1 \cdot c_2 \cdot c_3$ and $C=c_1 \cdot c_2 \cdot C^*$.
Since $W^2 \cdot dy= (\phi_E')^{2} \cdot w^2 \cdot dx$, we have
\begin{equation} \label{BackW}
    \int_{y_E^-}^0  W^2~ dy~ 
     \geq~  c_5 \int_{0}^{x_E}  w^2~ dx~ 
\end{equation}
where $c_5= \inf\{ (\phi_E'(x))^2~ |~ E \in K, x \in [0, x_E]\}$ 
is positive by Lemma \ref{phiproperties}.  By substituting (\ref{BackW})
into (\ref{Melange})  and applying Corollary \ref{UpDownEstimate},
we find that if $t< t_2=\min\{t_1, t'\}$, then
\begin{equation} \label{PreFinalNegative}
   \int_{0}^{x_E}|f_E| \cdot  w^2~ dx~ \geq~
     c_4 \cdot c_5 \int_{0}^{\infty} w^2~ dx~
   -\left(C+ c_4 \cdot c_5 \cdot C'\right) \cdot t^{\frac{1}{3}}~ \int_0^\infty w^2~ dx. 
\end{equation} 
If $t<t^-= \min \left\{t_2, \left(c_4 \cdot c_5/2(C+c_4 \cdot c_5 \cdot C')\right)^3\right\}$, 
then (\ref{FinalNegative}) holds with $c_- =c_4 \cdot c_5/2$.
\end{proof}


\section{Convergence, estimation, and separation of eigenvalues}
  \label{ConvEstSepSection}

Let $a_t^\mu$ be the family of quadratic forms defined as 
in \S \ref{aSection}.  In this section, we will evaluate the 
limit to which each real-analytic eigenvalue branch converges (Proposition \ref{amuconvergence}),
estimate the asymptotic behavior of eigenvalues (Proposition \ref{aSpec}), 
and show that if both $t$ and $E-\mu/\sigma(0)$ are sufficiently small, 
then eigenvalues near energy $E$ must be `super-separated' at order $t$ 
(Theorem \ref{Superseparation}).

\subsection{Convergence}

Let $t \mapsto \lambda_t$ 
be a real-analytic eigenvalue branch
of $a_t^\mu$ with respect to $\langle \cdot, \cdot \rangle_\sigma$.
Since $|w'|^2 \geq 0$ and $\sigma$ is decreasing, we have 
\begin{equation} \label{LambdaLower}
  \lambda_t~ \geq~ \frac{\mu \int_0^{\infty} |w_t|^2~ dx}{
\int_0^{\infty} |w_t|^2\cdot \sigma~ dx}~
\geq~  \frac{\mu}{\sigma(0)}.
\end{equation}
The first derivative of $a_t^\mu$,
\[ \dot{a}_t^\mu(u)~ =~ 2t \int_0^{\infty} \left|w_t'(x)\right|^2~ dx, \]
is nonnegative, and hence by Proposition \ref{aconvergence},
the eigenbranch $\lambda_t$ converges as $t$ tends to zero.

\begin{prop}  \label{amuconvergence}
We have
\[  \lim_{t \rightarrow 0}  \lambda_t~ =~ \frac{\mu}{\sigma(0)}. \]
\end{prop}
\begin{proof}
Let $w_t$ be an eigenfunction branch associated to $E_t$.
The variational formula (\ref{VariationalFormula}) becomes
\begin{equation} \label{AppliedVariational}
 \dot{\lambda} \cdot \|w_t\|_{\sigma}^2~ 
    =~ 2t \int_0^{\infty} \left|w_t'(x)\right|^2~ dx. 
\end{equation}
Using the eigenvalue equation for $a_t^\mu$ with respect 
to $\langle\cdot, \cdot\rangle_{\sigma}$
we find that
\[ t^2 \int_0^{\infty} \left|w_t'(x)\right|^2~ dx~ =~
 \int_0^{\infty} \left( \lambda_t \cdot \sigma(x)~ -~ \mu\right) \cdot |u(x)|^2~ dx.
\]
By combining this with (\ref{AppliedVariational}) and (\ref{LambdaLower})
we find that
\begin{equation}  \label{preCrucial}
 \dot{\lambda} \cdot \|w_t\|_{\sigma}^2~ \geq~
   \frac{2}{t} \cdot  \int_0^{\infty}  
    \left( \lambda_t \cdot \sigma(x)~ -~ \mu\right)~ \cdot |w_t(x)|^2~ dx.
\end{equation}

Suppose to the contrary that 
$\lambda_0:=\lim_{t \rightarrow 0} \lambda_t \neq \mu/\sigma(0)$.
Then by (\ref{LambdaLower}), we have $\lambda_0 > \mu/\sigma(0)$.
Let $K$ be the compact interval $[\lambda_0, \lambda_0+1]$.
Then for all $t$ sufficiently small, $\lambda_t \in K$.
Hence we can apply Proposition \ref{CrucialProposition},
with $E= \lambda_t$, and obtain a constant $\kappa>0$ such that
\[  \int_0^{\infty}  
    \left( \lambda_t \cdot \sigma(x)~ -~ \mu\right)~ \cdot |w_t(x)|^2~ dx~
   \geq~ \kappa \cdot \|w_t(x) \|^2_{\sigma}.
\]
By combining this with (\ref{preCrucial}) we find that
\[   \frac{d}{dt}~ \lambda_t~  \geq~ 
\frac{2}{t} .
\]
The left hand side is integrable on an interval of the form $[0,t_0)$,
but the right hand side is not integrable on such an interval. The claim follows.
\end{proof}

\subsection{Airy eigenvalues}

The remainder of this section concerns quantitative
estimates on the eigenvalues of $a_t^{\mu}$ for $t$ small.
In particular, we will use the Cherry-Langer transform to
compare the eigenvalues of $a_t^{\mu}$ to the eigenvalues
of the operator associated to the Airy equation. 
We first define and study the eigenvalue problem for the model operator.

For each $z \in \R$ and  $u \in C_0^{\infty}[z, \infty)$ define 
\[  \Acal_z(u)(y)~ =~   - u''(y)~ + y \cdot u(y). \]
The operator $\Acal_z$ is symmetric with respect to the   
$L^2([z, \infty), dy)$ inner product, and we
have $\langle \Acal_z(u), u \rangle \geq z \cdot \|u\|^2$. 
Thus, by the method of Friedrichs, we may extend $\Acal_{z}$ 
to a densely defined, self-adjoint operator on 
$L^2([z, \infty), dy)$ with either Dirichlet or Neumann
conditions at $y=z$.

Let $A_{\pm}$ be the solutions to the Airy equation defined in Appendix \ref{AiryAppendix}.

\begin{prop} \label{AiryEigenvalueProp}
The real number $\nu$ is a Dirichlet (resp. Neumann) 
eigenvalue of $\Acal_z$ with respect to the $L^2$-norm 
if and only if $z-\nu$ is a zero of $A_-$ (resp. $A_-'$). 
Moreover, each eigenspace of $\Acal_z$ is $1$-dimensional
and each eigenvalue of $\Acal_z$ is strictly greater than $z$.
\end{prop}

\begin{proof}
If $\psi$ be an eigenfunction of $\Acal_z$ with eigenvalue $\nu$, then   
\[     -\psi''(y)~ + (y- \nu) \cdot \psi~ =~ 0.\]
Hence if we let $A(x)=\psi(x+\nu)$, then $A''(x)=x \cdot A(x)$. 
Since $A$ belongs to $L^2([z, \infty), dy)$, there exists a constant 
$c$ such that $A= c \cdot A_-$, and thus $\psi(y)= c \cdot A_-(y-\nu)$
for each $y \geq z$. 
In particular, if $\psi$ is a Dirichlet (resp. Neumann)
eigenfunction, then $A_-(z-\nu)=0$ (resp. $A_-'(z-\nu)=0$). 
Sturm-Liouville theory ensures that the associated eigenspaces 
are one-dimensional, and the last statement follows from the fact that 
$A_-(x) \neq 0$ for $x\geq 0$.  
\end{proof}

\subsection{Estimation}

\begin{prop}\label{aSpec}
There exists $\delta_0$ and $C$ such that for any 
$t\leq \delta_0,$  if $\lambda \in [\frac{\mu}{\sigma(0)},\frac{\mu}{\sigma(0)}+\delta_0]$ 
is a Dirichlet (resp. Neumann)
eigenvalue of $a_t^\mu$, then there exists a zero, 
$z$, of $A_-$ (resp. $A_-'$) such that 
\begin{equation}
\left| \phi_\lambda(0)~ -~ t^\dt \cdot z \right|~ \leq~ C \cdot t^2.
\end{equation}
\end{prop}
 
\begin{proof}
We set $\epsilon\,=\ \frac{1}{2} \min\{ \rho_E^{-4}(0)~|~E\in [\frac{\mu}{\sigma(0)},\frac{\mu}{\sigma(0)}+1]\},$ and 
we choose $\delta_0$ to be the minimum of $1$ and the $\delta$ provided by Lemma \ref{wToWIntegral} that is 
associated with this $\epsilon$ and $q$ identically $1.$ Let $K$ be the compact 
$[\frac{\mu}{\sigma(0)},\frac{\mu}{\sigma(0)}+\delta_0].$
Let $w$ be an eigenfunction with eigenvalue $\lambda\in K$ and $t\leq \delta_0.$
Let $W$ the Cherry-Langer transform of $w$ at energy $\lambda$. 
According to Lemma \ref{wToWIntegral} and to the choice we made of $\epsilon$, we have 
\[ 
\frac{1}{2\rho_\lambda(0)^4} \int_0^\infty w^2 dx \,\leq \,\int_{\phi_\lambda(0)}^\infty W^2 dy \, \leq \frac{3}{2\rho_\lambda(0)^4} \int_0^\infty w^2 dx.
\]

Combining with Lemma \ref{LemSpecG}, (and using that $\rho_E(0)$ is uniformly bounded away from $0$ over the compact $K$), 
there exists a constant $C'$ such that
\begin{equation} \label{AirySquare2}
  \int_{\phi_\lambda(0)}^\infty |t^2 \cdot W''- y \cdot W|^2~  dy~ 
  \leq~  C \cdot t^4 \int_{\phi_{\lambda}(0)}^\infty |W(y)|^2~ dy.
\end{equation}
Setting $U(x)=W(t^{\dt}\cdot x)$, we have 
\begin{equation} \label{AirySquare3}
  \int_{t^{-\frac{2}{3}} \cdot \phi_\lambda(0)}^\infty |U''- x \cdot U|^2~  dx~ 
  \leq~  C \cdot t^{\frac{8}{3}} \int_{t^{-\dt} \cdot \phi_\lambda(0)} |U(x)|^2~ dx.
\end{equation}

Let $z_t= t^{-\frac{2}{3}} \cdot \phi_{\lambda}(0)$.
Then $U(z)=0$ (resp. $U'(z)=0$) 
if $\lambda$ is a Dirichlet (resp. Neumann) eigenvalue.
In particular, $U$ belongs to the domain of $\Acal_{z_t}$.
Moreover, from (\ref{AirySquare3}) we have that 
\begin{equation} \label{AiryOperatorNorm}
   \left\| \Acal_{z_t} (U) \right\|^2~
   \leq~ C \cdot t^{\frac{8}{3}} \cdot \| U\|^2. 
\end{equation}
Thus, since $\Acal_{z_t}$ is self-adjoint,
\begin{equation} \label{AiryOperatorNorm2}
   \left\langle \Acal_{z_t}^2 (U), U \right\rangle~ 
 \leq~ C \cdot t^{\frac{8}{3}} \cdot \| U\|^2. 
\end{equation}
Thus, by the minimax principle, $\Acal_{z_t}^2$ has an eigenvalue
in the interval $[0, C t^{\frac{8}{3}}]$. Hence $\Acal_{z_t}$ has an eigenvalue
in the interval  $[-C^{\und}t^{\frac{4}{3}}, C^\und t^{\frac{4}{3}}]$, and the claim follows from 
Proposition \ref{AiryEigenvalueProp}.
\end{proof}

\subsection{Separation}

We next show that, as $t$ tends to zero, the eigenvalues of $a_t^\mu$ with respect 
$\langle \cdot, \cdot \rangle_{\sigma}$ are 
separated at order greater than $t$. More precisely, we have the following.

\begin{thm}  \label{Superseparation}
Let $t_1, t_2, t_3,\ldots$
be a sequence of positive real numbers such that 
$\lim_{n \rightarrow \infty} t_n=0$.
For each $n \in {\mathbb Z}^+$, let $\lambda_n^+$ and $\lambda_n^-$ be distinct 
eigenvalues of the quadratic form $a_{t_n}^\mu$.  If 
$\lim_{n \rightarrow \infty} \lambda_n^{\pm} = \mu/\sigma(0)$, then 
\[  \lim_{n \rightarrow \infty}~ \frac{1}{t_n} \cdot 
            \left| \lambda_n^+~ -~ \lambda_n^- \right|~ =~ \infty. 
\] 
\end{thm}

This fact may be understood by using the following semiclassical heuristics: 
The threshold $\frac{\mu}{\sigma(0)}$ is the bottom of the potential, and 
the eigenvalues near it are driven by the shape of this minimum. Since $\sigma'(0)\neq 0,$ 
the asymptotics are given by the eigenvalues of the model problem 
$P_tu\,=\,-t^2 \cdot u''+x \cdot u=0$ 
on $(0,\infty)$. Denote by $e_n(t)$ the $n^{{\rm th}}$ eigenvalue of the model operator. 
Using homogeneity, $e_n(t)$ behaves like $e_n(1) \cdot t^\dt$ (and $e_n(1)$ actually is some 
zero of the Airy function see Proposition \ref{AiryEigenvalueProp}). For fixed $n,$ 
the separation between two eigenvalues is thus of order $t^\dt.$

It would be relatively straightforward to make the preceding reasoning rigorous
in the case of a finite number of real-analytic eigenvalue branches. 
(For instance we could use \cite{Friedlander}). Unfortunately, this is not enough 
for our purposes. In section \S \ref{SectionSimplicity}, we will need 
the result for a sequence of eigenvalues that may belong to an
infinite number of distinct branches.

\begin{remk}
Observe that the same semiclassical heuristics yield that this super-separation 
does not hold near an energy strictly greater than $\frac{\mu}{\sigma(0)}$. 
Indeed, near a non-critical energy, the spectrum is separated at order $t.$
\end{remk}

\begin{proof}[Proof of Theorem \ref{Superseparation}]
Suppose to the contrary that there exists a subsequence---that we will abusively 
call $t_n$---such that $|\lambda_n^--\lambda^+_n|/t_n$ is bounded.
Let $w_n^{\pm}$ denote a sequence of eigenfunctions associated to $\lambda_n^{\pm}$
with $\|w_n^{\pm}\|_{\sigma}=1$. Since $\lambda_n^- \neq \lambda_n^+$, 
we have $\langle w_n^-, w_n^+ \rangle_{\sigma}=0$. 
 
Let $W_n^{\pm}$ denote the Langer-Cherry transform of 
$w_n^{\pm}$ at the energy $E_n= \sup\{\lambda_n^-,\lambda_n^+\}$.
By hypothesis $\lim_{n \rightarrow \infty}=\mu/\sigma(0)$.
By Lemma \ref{LemSpecG} and Lemma \ref{wToWIntegral},
we find that there exist $N_1$ and $C$ such that if $n>N_1$, then  
\begin{equation} \label{OperatorNormWPlusMinus}
      \left\| \left(-t_n^2 \cdot \partial_y^2~ -~  y\right) W_n^{\pm} \right\|^2~ 
                    \leq~ C \cdot t_n^2 \cdot \left\| W_n^{\pm} \right\|^2.
\end{equation}

Since $\langle w_n^-, w_n^+ \rangle_{\sigma}=0$ and $\|w_n^{\pm}\|_{\sigma}=1$, 
it follows from Lemma \ref{wToWIntegral} that there exists $N_2>N_1$
such that if $n>N_2$, then 
\[ \left| \langle W_n^-, W_n^+ \rangle\right |~ \leq~  \frac{1}{2} \cdot \|W_n^-\| \cdot \|W_n^+\|. \]
Observe that this implies that for any linear combination of $W_n^+$ and $W_n^-$ we have 
\[
|\alpha_+|^2 \|W_n^+\|^2\,+\, |\alpha_-|^2\|W_n^- \|^2 \leq 2\| \alpha_+W_n^+\,+\,\alpha_-W_n^-\|^2.
\]
Therefore, it follows from (\ref{OperatorNormWPlusMinus}) that if 
$W$ belongs to the span, ${\mathcal W}_n$, of $\{W^-_n, W^+_n\}$, then 
\[
 re     \left\| \left(-t_n^2 \cdot \partial_y^2~ -~  y\right) W \right\|^2~ 
                    \leq~ 4 \cdot C \cdot t_n^2 \cdot \left\| W \right\|^2.
\]
Let $U(x)= W(t^{\dt}\cdot x)$ and let  ${\mathcal U}_n$ denote
the vector space corresponding to ${\mathcal W}_n$. If
$U \in {\mathcal U}_n$,  then 
\begin{equation} \label{OperatorNormW}
      \left\| \left(\partial_x^2~ -~  x\right) U \right\|^2~ 
                    \leq~ 4 \cdot C \cdot t_n^{\dt} \cdot \left\| U \right\|^2.
\end{equation}

Since $w_n^{\pm}$ satisfies the boundary condition at $0$,
the Langer-Cherry transform $W_n^{\pm}$ at energy $E_n$
satisfies the boundary condition at $\phi_{E_n}(0)$.
It follows that  ${\mathcal U}_n \subset {\rm dom}(\Acal_{z_n})$
where $z_n = t^{-\dt}_n \cdot \phi_{E_n}(0)$. 
By (\ref{OperatorNormW}) we have
\[  \langle \Acal_{z}^2(U),~ U \rangle~  \leq~ 4 \cdot C \cdot t_n^{\dt} \cdot \|U\|^2 \]
for each $U \in {\mathcal U}_n$. Hence, 
by the minimax principle, $\Acal_{z_n}^2$ has at least two independent
eigenvectors with eigenvalues in the interval $[0, 4 C \cdot t_n^{\dt}]$. 
Thus,  $\Acal_{z_n}$ has at least two independent
eigenvectors with eigenvalues in the interval $[-2 \sqrt{C} \cdot t_n^{\frac{1}{3}}, 2 \sqrt{C} \cdot t_n^{\frac{1}{3}}]$. 
By Proposition \ref{AiryEigenvalueProp}, 
the eigenvalues of $\Acal_{z_n}$ are simple, and 
hence $\Acal_{z_n}$ has at least two distinct eigenvalues, $\nu_n^+ < \nu_n^-$
lying in $[-2 \sqrt{C} \cdot t_n^{\frac{1}{3}}, 2 \sqrt{C} \cdot t_n^{\frac{1}{3}}]$. 
By Proposition \ref{AiryEigenvalueProp}, the number $a_n^{\pm}= z_n -\nu_n^{\pm}$ 
is a zero of the funtion $A_-$. Note that
\begin{equation} \label{ZeroDiff}
  |a_n^+-a_n^-|~ \leq~ 4 \sqrt{C} \cdot t_n^{\frac{1}{3}}.
\end{equation}

Since $A_-$ is real-analytic and $A_-(x)\neq 0$ for $x$ nonnegative, 
the zeroes $Z$ of $A_-$ are a countable discrete subset of $(-\infty, 0)$.
In particular, we there is a unique bijection $\ell: Z \rightarrow {\mathbb Z}^+$
such that $a < a'$ implies $\ell(a) > \ell(a')$ and 
$\lim_{k \rightarrow \infty} \ell^{-1}(k)= -\infty.$
From the asymptotics of $A_-$---see Appendix \ref{AiryAppendix}---one
finds that there exists a constant $c>0$ so that
\begin{equation}  \label{ZeroAsymp}
\lim_{k \rightarrow \infty}~  k^{-\frac{2}{3}}
    \cdot \ell^{-1}(k)~ =~ -c.
\end{equation}
\begin{equation}  \label{ZeroAsymp2}
\lim_{k \rightarrow \infty}~  k^{\frac{1}{3}}
    \cdot \left|\ell^{-1}(k)- \ell^{-1}(k+1)\right|~ =~ \frac{2}{3} \cdot c.
\end{equation}

Since $\lim_{n \rightarrow \infty}t_n=0$, estimate (\ref{ZeroDiff})
implies that $\lim_{n \rightarrow \infty} a_n^{\pm} = -\infty$, 
and hence $\lim_{n \rightarrow \infty} \ell(a_n^{\pm})= \infty$.
Therefore, since $a^+_n \neq a_n^-$ for all $n$, we have from
(\ref{ZeroAsymp2}) that there exists $N$ such that if $n> N$ then 
\[ \lim_{k \rightarrow \infty}~ 
    \left| a_n^+~ -~ a_n^-  \right|~ 
     \geq~ \frac{c}{2} \cdot   \ell(a_n^+)^{-\frac{1}{3}}.
\]
By combining this with (\ref{ZeroDiff}) we find that 
\begin{equation} \label{Contra}
      \left( \ell(a_n) \cdot t_n\right)^{\frac{1}{3}}~ \geq~  \frac{c}{4 \sqrt{C}}.
\end{equation}

But since $\lim_{n \rightarrow \infty} E_n =\mu/\sigma(0)$, we have $\lim_{n \rightarrow \infty} \phi_{E_n}(0)=0$.
Therefore, by Proposition \ref{aSpec} we have $\lim_{n \rightarrow \infty} t^{\dt} \cdot a_n^{\pm}=0$.  
By (\ref{ZeroAsymp}) we have 
\[ \lim_{n \rightarrow \infty}~  \frac{a_n}{\ell(a_n)^{\dt}}  =~ -c.\]
Thus, $\lim_{n \rightarrow \infty} t_n^{\dt} \cdot  \ell(a_n)^{\dt}=0$. 
This contradicts (\ref{Contra}).
\end{proof}


\part{Simplicity}


\section{Separation of variables in the abstract}  \label{SectionReduction}

Recall that the first step in our method for proving generic simplicity 
consists of finding a family $a_t$ such that $q_t$ is asymptotic to $a_t$  
and such that $a_t$ decomposes as a direct sum of `1-dimensional' quadratic
forms $a_t^{\mu}$ of the type considered in the previous sections.
In the present section we discuss the decomposition of $a_t$ into 
forms $a_t^{\mu}$. Although the content is very well-known,
we include it here for the purpose of establishing notation
and context.

Let $\langle \cdot, \cdot \rangle_{\sigma}$ be the inner product on 
${\mathcal H}_\sigma$ defined in \S \ref{aSection}.  Let ${\mathcal H}'$ 
be a real Hilbert space with inner product $(\cdot, \cdot)$. Consider
the tensor product $\Hcal := {\mathcal H}_\sigma \bigotimes {\mathcal H}'$ completed
with respect to the inner product $\langle \cdot,\cdot \rangle$ determined by
\begin{equation} \label{InnerProduct}
  \langle u_1 \otimes \vphi_1, u_2 \otimes \vphi_2 \rangle~ :=
      \langle u_1, u_2 \rangle_{\sigma} \cdot (\vphi_1,\vphi_2). 
\end{equation}  
Let $b$ be a positive, closed, densely defined quadratic form on ${\mathcal H}'$.
We will assume that the spectrum of $b$ with respect to $(\cdot, \cdot)$
is discrete and the eigenspaces are finite dimensional. 
For each $t>0$ and 
$u \otimes \vphi \in \Cc^\infty([0,\infty)) \bigotimes {\rm dom}(b)$,
define
\begin{equation}\label{DefaSmooth}
 a_t(u \otimes \vphi)~ =~ t^2 \cdot (\vphi,\vphi)  \int_0^{\infty} |u'(x)|^2~ dx~  
  +~ b(\vphi) \int_0^{\infty} |u(x)|^2~ dx. 
\end{equation}

Let $Y \subset \Cc^{\infty}([0,\infty))$ be a subspace. 
The restriction of $a_t$ to $Y \otimes \Hcal'$ is
a nonnegative real quadratic form. 
By Theorem 1.17 in Chapter VI of \cite{Kato}, this restriction 
has a unique minimal closed extension. In particular, 
let $\dom(a_t)$ be the collection of  $u \in {\mathcal H}_{\sigma} \tensor \Hcal'$
such that there exists a sequence $u_n \in Y \tensor \dom(b)$
such that $\lim_{n\rightarrow \infty} \|u_n-u \|=0$ and 
$u_n$ is Cauchy in the norm 
\[  [u]_{t}~ :=~ a_t(u)~ +~  \|u\|_{\mathcal H}. \]
For each $u \in \dom(a_t)$ define 
\[  a_t(u)~ :=~ \lim_{n \rightarrow \infty} a_t(u_n)   \]
where $u_n$ is a sequence as above. Note that for $t,t'>0$ the norms
$[\cdot]_{t}$ and $[\cdot]_{t'}$ are equivalent, and hence $\dom(a_t)$ 
does not depend on $t$.

\begin{remk}
In applications, either $Y= \Cc([0,\infty))$ or $Y$ consists of smooth
functions whose support is compact and does not include zero.  
In the former case, eigenfunctions of $a_t$ will satisfy a Neumann condition 
at $x=0$ and in the latter case they will satisfy a Dirichlet condition
at $x=0$. 
\end{remk}

\begin{prop}
The family  $t \mapsto a_t$ is a real-analytic family of type (a) 
in the sense of Kato.\footnote{See Chapter VII \S 4.2 in \cite{Kato}.}
\end{prop}

\begin{proof}
For each $t$ the form $a_t$ is closed with respect to $\langle \cdot, \cdot \rangle$,
the domain $\dom(a_t)$ is constant in $t$, and 
for each $u \in \dom(a_t)$, the function $t \mapsto a_t(u)$
is analytic in $t$.
\end{proof}

\begin{eg}
Let ${\mathcal H'}$ be the space of square
integrable functions on a compact Lipschitz 
domain $U \subset \R^n$ with the usual inner product.
Then ${\mathcal H} \bigotimes {\mathcal H}'$ 
is isomorphic to the completion of $C^{\infty}_0((0,\infty) \times U)$
with respect to the inner product
\[ \langle f, g \rangle~ =~ 
 \int_U  \int_{0}^{\infty} f(x,y) \cdot g(x,y) \cdot \sigma(x)~ dx~ dy.
\]   
Let $\tilde{b}$ be the quadratic form defined on $H^1(U)$ by
\begin{equation} \label{bDefn}
   \tilde{b}(\phi)~ =~  \int_U  |\nabla \phi|^2~  dx~ dy. 
\end{equation}
We define $b$ to be the restriction of $\tilde{b}$ to 
any closed subset of $H^1(U)$ on which it defines a 
positive quadratic form. In this case
the quadratic form $a_t$ is equivalent to the form 
\begin{equation} \label{aDefn}
   \overline{a}_t(u)~ =~   \int_{\R^+ \times U} 
 \left( t^2 \cdot |\partial_{x} u|^2~ 
   +~   |\nabla_y u|^2  \right)~  dx~ dy. 
\end{equation}
\end{eg}

%
%

For each $\mu>0$ and $t>0$, we define the quadratic form $a^{\mu}_t$ 
as in \S \ref{aSection}. Observe that this form $a^{\mu}_t$ is equivalent 
to the construction above with $\Hcal'= {\mathbb \R}$ with its standard
inner product and $b(s)=\mu \cdot s^2$. 
Note that the norms $[\cdot]_{t, \mu}$ and $[\cdot]_{t', \mu'}$ that are
used to extend $a_t^{\mu}$ and $a_{t'}^{\mu'}$ are equivalent. Hence $\dom(a_t^{\mu})$ is 
independent of $t$ and $\mu$.  We will denote this common domain by ${\mathcal D}$.

\begin{prop}\label{Domains}
We have $\dom(a_t)= {\mathcal D} \bigotimes \dom(b)$.
\end{prop}

\begin{proof}
For notational convenience set $d(u)= t^2 \int_{0}^{\infty} |u'(x)|^2 dx$.   
In particular, 
\[ a_t(u \otimes v)= d(u)\cdot |v|_{\Hcal'}^2~ +~ |u|_{\sigma}^2 \cdot b(v).  \] 

If $u \in \Dcal$, then there exists a sequence 
$u_n \in Y$ that converges to $u$ in $\Hcal_{\sigma}$ 
such that $d(u_j -u_k)+ |u_j-u_k|^2_{1} + |u_j-u_k|^2_{\sigma}$ 
tends to zero as $j,k$ tend to infinity. Let $v \in \dom(b)$.
We have
\[ a_t(u_j\otimes v- u_k\otimes v)~ =~ d(u_j-u_k) \cdot |v|^2  + |u_j-u_k|^2 \cdot b(v). 
\]
Observe that this quantity tends to zero as  $j,k \rightarrow \infty$.
It follows that $u \otimes v \in \dom(a_t)$, and moreover that
the restriction of $a_t$ to $\Dcal \bigotimes \dom(b)$ is a closed 
form. But the extension to $\dom(a_t)$ is the unique 
minimal closed extension, and so  $\Dcal \bigotimes \dom(b) = \dom(a_t)$. 
\end{proof}

\begin{prop} \label{Separation}
If $\phi$ is a $\mu$-eigenvector for $b$ with respect to $\langle\cdot, \cdot\rangle_{\Hcal'}$,
and $v$ is a $\lambda$-eigenvector of $a_t^{\mu}$ with respect to
$\langle\cdot, \cdot\rangle_{\sigma}$, then 
$v \otimes \phi$ is a $\lambda$-eigenvector of $a_t$ 
with respect to $\langle\cdot, \cdot\rangle_{\Hcal}$. 
Conversely, if $u$ is a  $\lambda$-eigenvector of 
$a_t$ with respect to $\langle\cdot, \cdot\rangle_{\Hcal}$, 
then $u$ is a finite sum $\Sigma v_{\mu} \otimes \phi_{\mu}$
where $v_{\mu}$ is a $\lambda$-eigenvector of $a_{t}^{\mu}$ with respect to
$\langle\cdot, \cdot\rangle_{\sigma}$ 
and $\phi_{\mu}$ is a $\mu$-eigenvector of $b$ with respect 
to $\langle\cdot, \cdot\rangle_{\Hcal'}$.
\end{prop}

\begin{proof} 
For each $u \otimes w \in \Dcal \bigotimes \dom(b)$, we have 
\begin{eqnarray*}
 a_t(v \otimes \phi, u \otimes w)~
   &=& d(v,u)\cdot \langle \phi, w\rangle~ 
                +~ \mu \cdot \langle v, u \rangle  \cdot \langle \phi,w\rangle  \\
   &=&  \lambda \langle v, u \rangle \cdot \langle \phi, w\rangle \\
   &=&  \lambda \langle v \otimes \phi, u \otimes w \rangle.
\end{eqnarray*}
It follows that $v \otimes \phi$ is an eigenvector of $a_t$.

The span of the eigenvectors of $a_t^{\mu}$ (resp. $b$)  is dense in 
$\Hcal$ (resp. $\Hcal'$), and hence the span of the 
tensor products of eigenvectors is dense in $\Hcal \bigotimes \Hcal'$.
Thus, each eigenvector $\psi$ of $a_t$ is a countable sum of tensor products of 
eigenvectors. Let $v \tensor \phi$ be one of the tensor products that appears 
in the sum and suppose that $v$ has eigenvalue $\lambda'$.
Then 
\[   \lambda' \cdot \langle \psi, v \tensor \phi \rangle~ =~ 
  a_t(\psi, v \otimes \phi)~ =~ \lambda \cdot \langle \psi, v \otimes \psi\rangle. \]
Thus, $\lambda'=\lambda$. Since each eigenspace of $a_t^{\mu}$
is 1-dimensional and the spectrum of $a_t^{\mu}$ is bounded 
from below by $\mu/\sigma(0)$, only finitely many terms appear.
\end{proof}

\begin{prop}\label{ConvEigen}
For each analytic eigenvalue branch $\lambda_t$ of $a_t$, 
there exists a unique $\mu \in \spec(b)$ such that 
$\lambda_t$ is an analytic eigenvalue branch of $a^{\mu}_t$. 
In particular, $\lambda_t$ decreases to $\frac{\mu}{\sigma(0)}$  as $t$ tends to $0.$
\end{prop}

\begin{proof}
Let $t_0>0$. For each $\mu \in \spec(b)$, consider the set $A_{\mu}$ of $t \in (0, t_0)$ 
such that $\lambda_t \in \spec(a^{\mu}_t)$. By Proposition \ref{Separation}, the 
union $\bigcup_{\mu} A_{\mu}$ equals $(0,t_0)$. Since $\spec(b)$ is countable, the
Baire Category Theorem implies that there exists $\mu \in \spec(b)$ such that 
$A_{\mu}$ has nonempty interior $A_{\mu}^0$. For each real-analytic eigenvalue branch $\nu_t$
of $a_t^{\mu}$, let $B_{\nu} \subset A_{\mu}^0$  be the set of $t$ such that $\nu_t=\lambda_t$.
Since there are only countably many eigenvalue branches, the
Baire Category Theorem implies that there exists an eigenvalue branch $\nu_t$
of $a_t^{\mu}$ such that $B_{\mu}$ has nonempty interior $B_{\mu}^0$.  
Since $\lambda_t$ and $\nu_t$ are real-analytic functions that 
coincide on a nonempty open set, they agree for all $t$.

The latter statement then follows from Proposition \ref{amuconvergence}.
\end{proof}

\begin{coro}
If each eigenspace of  $b$ is $1$-dimensional, then for each $t$ belonging
to the complement of a countable set,  each eigenspace of $a_t$ with 
respect to $\langle \cdot, \cdot \rangle_{{\mathcal H}}$ is 1-dimensional.
\end{coro}

\begin{proof} 
Let $\lambda_t$ be an analytic eigenbranch of $a_t.$ 
This eigenbranch converges to some $\frac{\mu}{\sigma(0)}.$ 
Since the spectrum of $b$ is simple there is a unique $\vphi_\mu$ such that 
the  corresponding eigenvector can be written $v_t \otimes \vphi_\mu$ 
with $v_t$ an eigenvector of $a_t^\mu.$ Since the spectrum of $a_t^\mu$ is 
simple, the  choice of $v_t$ is unique.  And thus the 
eigenvector branches corresponding to two different eigenvalue branches 
cannot coincide for all $t$. The analyticity of the eigenbranches then yields 
the result.
\end{proof}

We end this section by establishing some notation that will be useful in the
sections that follow.  For each eigenvalue $\mu$ of $b$, let $\Vcal_{\mu}$ 
denote the associated eigenspace and let $P_{\mu}: \Hcal' \rightarrow \Vcal_{\mu}$ 
denote the associated orthogonal projection. Define 
$\Pi_{\mu}:  \Hcal_{\sigma} \bigotimes \Hcal'$ by
\[  \Pi_{\mu}(v \otimes w)~ =~ v \otimes P_{\mu}(w). \]   
If ${\mathcal M}$ is a collection of eigenvalues $\mu$
of $b$, then we define $\Pi_{\Mcal}$ to be the 
orthogonal projection onto the direct sum of $\mu$-eigenspaces.
That is, 
\[ \Pi_{\Mcal}~ =~ \sum_{\mu \in \Mcal} \Pi_{\mu}. 
\] 
Note that the subscript for $\Pi$ may represent either  an eigenvalue or a set of  
eigenvalues.

\begin{ass}  \label{bsimple}
In what follows we assume that each eigenspace of $b$ with respect to 
$\langle \cdot, \cdot \rangle$ is 1-dimensional. 
\end{ass}

One convenient consequence of this assumption is that for each 
$w \in \Hcal$, there exists $\tilde{w}_{\mu} \in  \Hcal_{\sigma}$ and 
a unit norm eigenvector $\phi_\mu$ of $b$ such that
\begin{equation} \label{OneTensor}
 \Pi_{\mu}(w)~ =~ \tilde{w}_{\mu} \tensor \phi_{\mu}. 
\end{equation} 
Indeed, for each $\mu \in \spec(b)$, let $\phi_{\mu} \in \Vcal_{\mu}$. 
Since $\dim(\Vcal_{\mu})=1$, each vector in $\Hcal_{\sigma} \tensor \Vcal_{\mu}$
is of the form $v \otimes \phi_{\mu}$. In particular, there exists 
$\tilde{w}_{\mu}$ so that (\ref{OneTensor}) holds. Note that
\[  w =~ \sum_{\mu \in \spec(b)}~ \Pi_{\mu}(w)~
  =~ \sum_{\mu \in \spec(b)}   \tilde{w}_{\mu} \tensor \phi_{\mu}. \]
%


\section{Projection estimates} \label{ProjectionSection}

In this section,  $q_t$ will denote a family of quadratic 
forms densely defined on $\Hcal$
that is asymptotic at first order\footnote{See Definition \ref{AsymptoticDefn}.}
to the family $a_t$ defined in the preceding section. 
Let $P_{a_t}^I$ is the orthogonal projection  
onto the direct sum of eigenspaces of 
$a_t$ associated to the eigenvalues of $a_t$ that belong to the interval $I$
(see \S \ref{QuasimodeSection}).
We will provide some basic estimates on
\begin{equation}  \label{wt}
  w~ :=~ P_{a_t}^I(u)
\end{equation}
We begin with the following quasimode type estimate.  
In the sequel $\phi_{\mu}$ will denote  
a norm eigenvector of $b$ with eigenvalue $\mu$. By Assumption \ref{bsimple},
$\phi_{\mu}$ is unique up to sign.

\begin{lem}\label{LemLiminf}
Let $J\subset I$ be a proper closed subinterval of a compact interval $I$.   
There exist constants $C>0$ and $t_0>0$ such that 
if $\mu \in \spec(b)$, $t<t_0$, $u$ is an eigenfunction of $q_t$ 
with eigenvalue $E \in I$, $z \in \Dcal$,  then the projection 
$w= P_{a_t}^I(u)$ satisfies
\begin{equation}
\left| a_t \left(\Pi_{\mu}w, z \otimes \phi_{\mu} \right)~ 
   -~E \cdot \left\langle \Pi_{\mu} w,  z \otimes \phi_{\mu}
   \right\rangle \right|~ \leq~  
  C \cdot t \cdot \|z\|_{\sigma} \cdot \|w\|.
\end{equation}
\end{lem}

\begin{proof}
Since $q_t$ and $a_t$ are asymptotic at first order, 
Lemma \ref{Closeness} applies.  In particular, by letting 
$\delta = {\rm dist}(J, \partial I)$,
 $t_0= \frac{1}{2}(1+ E/\delta)^{-1}$, and $C= (4/3) \cdot \sup(I)$,
we have for $t < t_0$ and $v \in \Dcal \bigotimes \dom(b)$
\begin{equation} \label{Preprojector}
\left| a_t \left(w,v \right)~ 
   -~E \cdot \left\langle w,v \right\rangle \right|~ \leq~  
  C \cdot t \cdot \|v\| \cdot \|w\|.
\end{equation}

For each $\mu' \in \spec(b)$, there exists $\tilde{w}_{\mu'} \in \Dcal$ 
so that
\begin{equation} \label{wDecomposition}
  w~ =~ \sum_{\mu' \in \spec(b)}~  \tilde{w}_{\mu'} \otimes \phi_{\mu'} 
\end{equation}
and $v= \tilde{v}_{\mu} \otimes \phi_{\mu}$.
If $\mu' \neq \mu$, then $b(\phi_{\mu}, \phi_{\mu'})=0$ and 
$\langle \phi_{\mu}, \phi_{\mu'}\rangle=0$,
and hence using (\ref{InnerProduct}) and (\ref{DefaSmooth})  we find that
\[  a_t( \tilde{w}_{\mu'} \otimes \phi_{\mu'}, z \otimes \phi_{\mu})~
  -~ E \cdot \left\langle \tilde{w}_{\mu'} \otimes \phi_{\mu'}, 
    z \otimes \phi_{\mu'}\right\rangle~ =~ 0. \]
Thus,
\[  a_t\left(  w, v\right)~ -~ E \cdot \left\langle w,v \right\rangle~ 
    =~ a_t\left( \Pi_{\mu} w, v\right)~ -~ E \cdot \left\langle \Pi_{\mu}w,v \right\rangle.
 \]
The claim then follows from substituting this into (\ref{Preprojector}).
\end{proof}

\begin{lem}  \label{MinusPreProjection}
Let $J\subset I$ be a proper closed subinterval of a compact interval $I$.
Let $\mu \in \spec(b)$ with $\mu < \sigma(0) \cdot \inf(I)$ and let $\epsilon>0$.
There exist constants $\kappa>0$ and $t_0>0$ such that 
if $t<t_0$, $u$ is an eigenfunction of $q_t$ 
with eigenvalue $E \in I$, and 
\[ \left\|\Pi_{\mu} w \right\|~ \geq~ \epsilon~ \cdot \| w \|,
\]
where $w= P_{a_t}^I(u)$, then we have
\begin{equation} \label{adotLowerBound}
  \dot{a}_t\left( \Pi_{\mu}(w) \right)~ \geq~ 
   \frac{\kappa}{t} \cdot \left\|\Pi_{\mu}(w) \right\|_{\sigma}^2.
\end{equation}
\end{lem}

\begin{proof}
We have $\Pi_{\mu}w= \tilde{w}_{\mu} \otimes \phi_{\mu}$ for some
$\tilde{w}_{\mu} \in \Dcal$.
Since, by assumption,
$\|\phi_{\mu}\|=1$, we have $\|\Pi_{\mu}(w)\|= \|\tilde{w}_{\mu}\|$
and hence the assumption becomes
\[  \|\tilde{w}_{\mu} \|_\sigma~ \geq~ \epsilon \cdot \|w\|.
\]
Therefore, Lemma \ref{LemLiminf} gives
\begin{equation}  \label{qmode}
  \left| a^{\mu}_{t} \left(\tilde{w}_{\mu},z\right)~  -~  
  E_t \cdot \langle \tilde{w}_{\mu} ,z\rangle_\sigma \right|~ 
\leq~ C \cdot t \cdot \|z\|_\sigma \cdot \frac{\| \tilde{w}_{\mu} \|_\sigma}{\epsilon}
\end{equation}  
for all sufficiently small $t$.
Since $\mu/\sigma(0) < \inf(I)$, the compact set $I$ is a
subset of $(\mu/\sigma(0),\infty)$.
Hence we may apply Proposition \ref{CrucialProposition} to obtain
$\kappa>0$ and $t_1>0$ so that if $t< t_1$, then 
\begin{equation} \label{NonconcentrationApplied}
\int_0^{\infty} \left( E_t \cdot \sigma(x)~ -~ \mu\right) \cdot |\tilde{w}_{\mu}|^2~ 
      dx~ \geq~ 
  \kappa \cdot \| \tilde{w}_{\mu}\|_\sigma^2.
\end{equation} 

Inspection of (\ref{DefaSmooth}) gives that 
\begin{equation} \label{adottensor}
  \dot{a}_t(\tilde{w}_{\mu} \otimes \phi_{\mu}, \tilde{w}_{\mu'} \otimes \phi_{\mu'})~
   =~   2 t   \cdot  \langle \phi_{\mu}, \phi_{\mu'} \rangle
     \int_{0}^{\infty} 
\left(\partial_x \tilde{w}_{\mu} \cdot \partial_x  \tilde{w}_{\mu'} \right).
\end{equation}
In particular
\[   \dot{a}\left(\tilde{w}_{\mu} \otimes \phi_{\mu} \right)~ 
    =~ 2t \int_{0}^{\infty} \left| \partial_x \tilde{w}_{\mu} \right|^2~ dx 
\]
Thus, by using the definition of $a_t^{\mu}$ and estimates (\ref{qmode}) and 
(\ref{NonconcentrationApplied}) we find that
\begin{eqnarray*}
    \dot{a}\left( \Pi_{\mu}(w) \right) & = &  \frac{2}{t}
     \left( a^{\mu}_t( \tilde{w}_{\mu}) - \mu \int_{0}^{\infty} |\tilde{w}_{\mu}|^2~ dx \right)  \\
  &\geq & \frac{2}{t}  \left( \int_0^{\infty}~ \left(E_t \cdot \sigma~ -~ \mu \right) 
         |\tilde{w}_{\mu}|^2~ dx \right)~ -~
         \frac{2 C}{\epsilon} \cdot \|\tilde{w}_{\mu}\|_{\sigma}^2 \\
 & \geq &   2 \left(\frac{\kappa}{t}~ -~ \frac{C}{\epsilon} \right) 
                \cdot \|\tilde{w}_{\mu}\|_{\sigma}^2. \\
\end{eqnarray*}
By choosing $t_0= \min\{t_1, C/(\epsilon \cdot \kappa)\}$ we obtain the claim.
\end{proof}

\begin{remk}\label{kappamuindep}
In the preceding Lemma the constants $t_0$ and $\kappa$ a priori depends on the chosen $\mu$. 
However, since there is only a finite number of eigenvalues of $b$ that satisfy 
$\mu \leq \sigma(0)\inf I$, we can choose $t_0$ and $\kappa$ depending only on $I$ and not 
on the eigenvalue $\mu.$
\end{remk}

It will be convenient to introduce the following 
notation. Given $\mu \in \spec(b)$, define  
\[  \tilde{\mu}~ =~ \frac{\mu}{\sigma(0)} 
\] 
where $\sigma$ is as in  \S \ref{SectionReduction}.
For each compact interval $I$, define   
\[  \Mcal_I= \{ \mu \in \spec(b) |~ \tilde{\mu} \in I\},  \]
\[  \Mcal_I^-= \{ \mu \in \spec(b) |~ \tilde{\mu} < \inf I \}, \]
\[ \Mcal_I^+= \{ \mu~ \in \spec(b)|~ \tilde{\mu} > \sup I \}. \]
Note that $\spec(b)$ equals the disjoint union of $\Mcal_I^-$,
$\Mcal_I$, and $\Mcal_I^+$, and in particular, each $v \in \Hcal$
can be orthogonally decomposed as
\[  v~ =~ \Pi_{\Mcal_I^-}(v)~ +~ \Pi_{\Mcal_I^-}(v)~ +~ \Pi_{\Mcal_I^-}(v). \]
%

The following lemma is crucial to our proof of generic simplicity.
The proof uses both Theorem \ref{Thmlimiteigenbranch} and---by way of 
Lemma \ref{MinusPreProjection}---Proposition 
\ref{CrucialProposition}.

\begin{lem}  \label{MinusProjection}
Let $J\subset I$ be a proper closed subinterval of a compact interval $I$.   
Let $E_t$ be a real-analytic eigenvalue branch $q_t$, and let $V_t$ 
be the associated family of eigenspaces.
Let $t \mapsto u_t$  be a map from $(0, t_0]$ to $V_t$ that is continuous
on the complement of a countable set.  If $w_t= P_{a_t}^I(u_t)$, then  
\begin{equation} \label{LimInf}
\liminf_{t \rightarrow 0}~
\frac{\|\Pi_{\Mcal^-_I}(w_t)\|}{\|w_t\|}~ 
  =~ 0.
\end{equation}
Here if $w_t=0$, then we interpret the ratio to be equal to $1$.
\end{lem}

\begin{proof}
Suppose that (\ref{LimInf}) is false. 
We have the orthogonal decomposition 
\[ \Pi_{\Mcal_I^-}(w_t)~  =~ \sum_{\mu \in \Mcal_I^-}~  \Pi_{\mu }(w_t), \]
and hence there exists  $\epsilon>0$ and $t_0>0$ such that for 
each $t<t_0$ there exists $\mu_t \in \Mcal^-_I$ such that
\begin{equation}  \label{ProjectionLowerBound}
   \|\Pi_{\mu_t}(w_t)\|~  \geq~ \epsilon \cdot \|w_t\|. 
\end{equation}

Using the orthogonal decomposition of $w$ as in (\ref{wDecomposition})
we find that
\[
\dot{a}_t\left( w_t\right )~ =~ \sum_{\mu \in \spec(b)}
     \dot{a}_t\left( \Pi_{\mu_t}(w_t)  \right).
\]
(See also (\ref{adottensor}).)
In particular, since the quadratic form $\dot{a}_t$ is nonnegative, 
we have that $\dot{a}_t(w_t) \geq \dot{a}_t(\Pi_{\mu_t}(w_t))$. 
Thus, it follows from Lemma \ref{MinusPreProjection} and 
(\ref{ProjectionLowerBound})  that
\[ \dot{a}_t\left( w_t\right )~ \geq~ 
    \frac{\epsilon \cdot \kappa}{t} \cdot \|w_t\|^2 \]
for all $t$ sufficiently small with some $\kappa$ independent of $t$ 
(according to Remark \ref{kappamuindep}).
Thus, it follows from Theorem \ref{Thmlimiteigenbranch} that 
the function  $1/t$ is integrable on an interval whose left endpoint is zero. 
This is absurd.
\end{proof}

\begin{lem} \label{KIsmall}
Let $I$ be a compact interval.  If $w$ belongs to the range of $P_{a_t}^I$, then
\[  
\Pi_{\Mcal_I^+} (w)~ =~ 0.
\]  
In particular, 
\[  
   \|w \|^2~ =~ \left\| \Pi_{\Mcal_I}(w)\right\|^2~ 
   +~  \left\| \Pi_{\Mcal_I^-}(w)\right\|^2. 
\]
\end{lem} 
   
\begin{proof}
By definition, $w$ is a linear combination of eigenfunctions of $a_t$
whose eigenvalues belong to $I$.  Hence by Proposition  \ref{Separation}, 
we have
\[  w~ =~ \sum_{\mu \in \spec(b)}~ \sum_{\lambda \in I \cap \spec(a_{t}^{\mu})}~ 
    v_{\lambda, \mu} \otimes \phi_{\mu}. \]
where $v_{\lambda, \mu}$ is belongs to the $\lambda$-eigenspace of $a_t^{\mu}$
and  $\phi_{\mu}$ belongs to the $\mu$-eigenspace of  $b$. 
Hence
\begin{equation}  \label{M+}
  \Pi_{\Mcal^+_I}(w)~ =~ 
      \sum_{\mu \in \Mcal^+_I}~ 
       \sum_{\lambda \in I \cap \spec(a_{t}^{\mu})}~ v_{\lambda, \mu} \otimes \phi_{\mu}.
\end{equation}
According to Proposition \ref{amuconvergence}, each eigenvalue $\lambda$ of $a_{t}^{\mu}$
satisfies $ \lambda \geq \tilde{\mu}$. 
If $\mu \in \Mcal_I^+$, then $\tilde{\mu} \geq \sup(I)$.
Hence each term in (\ref{M+}) vanishes. 
\end{proof}


\section{The limits of the eigenvalue branches of  $q_t$ } \label{SectionLimits}

Proposition \ref{qconvergence} implies that each real-analytic 
eigenvalue branch $E_t$ of $q_t$ converges as $t$ tends to zero. 
In this section we us e the results of the previous section 
to show that each limit belongs to the set
\[ \widetilde{\spec(b)} = \left\{ \tilde{\mu}~ |~ \mu \in \spec(b) \right\}. \]

\begin{thm} \label{qLimit}
For each real-analytic eigenvalue branch $E_t$ of $q_t$, we have
\[\lim_{t\rightarrow 0} E_t~ \in \widetilde{\spec(b)}. \] 
\end{thm}

\begin{proof}
Suppose to the contrary that the limit, $E_0$, does not 
belong to $\widetilde{\spec(b)}$. 
Since  $\widetilde{\spec(b)}$ is discrete, there exists
a nontrivial compact interval $I$ such that $E_0 \in J$, 
such that 
\begin{equation}  \label{IEmpty}
 J \cap  \widetilde{\spec(b)}~ =~ \emptyset.  
\end{equation}
Since $J$ is nontrivial and $E_t$ is continuous, 
there exists $t_0$ such that if $t< t_0$, then $E_t \in J$.
Let $I$ be a compact interval such that 
$J \subset I \subset \left(\R \setminus \widetilde{\spec(b)} \right)$. 

Let $u_t$ be a real-analytic eigenfunction branch associated to $E_t$
and let $w_t =P_{a_t}^I(u_t)$. We have chosen $I$ so that $\Mcal_I= \emptyset$.
Thus, by Lemma \ref{KIsmall}  
\[ 
 \left\| \Pi_{\Mcal_I^-}(w_t) \right\|^2~ =~  \| w_t\|^2. 
\]
This contradicts Lemma \ref{MinusProjection}.
\end{proof}


\section{Generic simplicity of $q_t$}  \label{SectionSimplicity}

In this section, we prove that the spectrum of $q_t$ is generically simple. 
We will make crucial use of the `super-separation' of the eigenvalues of 
$a_t$ for small $t$ (see Theorem \ref{Superseparation}). 

Before providing the details of the proof, we first 
illustrate how super-separation 
can be useful in proving simplicity.  Consider an eigenbranch 
$(\lambda_t, u_t)$ of $q_t$ such that $\lambda_t \rightarrow \tilde{\mu}$. 
If $\| \Pi_{\mu}u_t\|$ is uniformly bounded away from $0$, then, arguing as in 
the beginning of the proof of Lemma \ref{MinusPreProjection}, we would 
find that $\Pi_{\mu}u_t$ is a 
first order quasimode of $a_t^\mu$ at energy $\tilde{\mu}.$ If the eigenbranch 
is not simple then we would obtain a sequence $t_n$ tending to zero
and two ditinct eigenvalues $\lambda$, $\lambda'$ of $a_{t_n}^\mu$ such that 
$t_n \cdot \left ( \lambda-\lambda'\right )$ 
is bounded. This would contradict super-separation.

\begin{thm} \label{SimplicityTheorem}
Let $E_t$ be a real-analytic eigenvalue branch 
$E_t$ of $q_t$, and let $V_t$ be the associated real-analytic 
family of eigenspaces (see Remark \ref{DefVt}). For each $t \in (0, t_0]$
we have $\dim(V_t)=1$.
\end{thm}

Since each eigenvalue branch of $q_t$ is real-analytic and the spectrum of each $q_t$ 
is discrete with finite dimensional eigenspaces, we have the following corollary.

\begin{coro}
Let $E_t$ be a real-analytic eigenbranch, then $E_t$ 
is a simple eigenvalue of $q_t$ for all $t$ in the complement of a discrete 
subset of $(0, t_0]$.
\end{coro}

\begin{proof}[Proof of Theorem \ref{SimplicityTheorem}]
Suppose that the conclusion does not hold. 
Since $V_t$ is a real-analytic family of vector spaces, 
its dimension is constant and so for each $t \in (0, t_0]$, we have $\dim(V_t)>1$. 

By Theorem \ref{qLimit} there exists $\mu \in \spec(b)$ such that
$E_t$ tends to $\tilde{\mu}=\mu/\sigma(0)$ as $t$ tends to zero. 
Let $I$ be a compact interval so that $I \cap \widetilde{{\rm spec}(b)}=\,\{\tilde{\mu}\}$. 
By Lemma \ref{BigPik} below, there exists $t_3 \leq t_0$ 
and a map $t \mapsto u_t$ from  
$(0,t_3]$ into $V_t$ that is continuous on the complement of a 
discrete set so that if $t \in (0,t_3] \setminus Z'$, then 
\[ \|\Pi_{\mu}(w_t)\|~ <~ \frac{1}{2} \cdot \|w_t\| \]
where $w_t = P_a^I(u_t)$.  Thus, since $\{\mu \}=\Mcal_{I}$, 
Lemma \ref{KIsmall} gives that
\[   \|\Pi_{\Mcal^-_I}(w_t)\|~ \geq~ \frac{1}{2} \cdot \|w_t\|. \]
This contradicts Lemma \ref{MinusProjection}.
\end{proof}

\begin{lem} \label{BigPik}
Let $E_t$ be a real-analytic eigenvalue branch of $q_t$ such that 
for each $t>0$ we have ${\rm dim}(V_t) > 1$.  Let $\mu \in \spec(b)$
be such that $ \lim_{t \rightarrow 0}~ E_t= \tilde{\mu}$, and 
let $I$ be a compact interval such that 
\[  I~ \cap~ \widetilde{ {\rm spec}(b)}~ =~  \tilde{\mu}. \]
There exists $t_0>0$ and a function $t \mapsto u_t$ 
that maps $(0,t_0]$ to $V_{t}$, is continuous on the
complement of a discrete set, and satisfies
\begin{equation}  \label{Pihalf}
 \left\| \Pi_{\mu}(w_t) \right\|~ \leq~ \frac{1}{2} \cdot \|w_t\|
\end{equation}
where $w_t = P_{a_t}^I(u_t)$.
\end{lem}

To prove Lemma \ref{BigPik}, we will use the following well-known fact.

\begin{lem}\label{DistSpec}
Let $\{g_k: (a,b) \rightarrow \R~ |~ k \in \N\}$ be a collection
of real-analytic functions.  If 
for each $k \in N$ and $t \in (a,b)$ we have $g_{k+1}(t)> g_{k}(t)$
then the set 
\[   \left\{~ t \in (a,b)~ |~  g_{k}(t)=0,~ k \in \N  \right\} \]
is a discrete subset of $(a, b)$. 
\end{lem}

\begin{proof}
Suppose that $g_{k}(t)=0$ for some $k \in \N$ and $t \in (a,b)$. 
Since $g_k$ is real-analytic there exists an open set $U \ni t$
such that if $t' \in U \setminus \{t\}$, then $g_k(t)=0$.  
Since $k'>k''$ implies $g_{k'}(t)> g_{k''}(t)$  we have 
\[  t~ \in~  g_{k+1}^{-1}( 0, \infty)~ =~ \bigcup_{k' >k}  g_{k'}^{-1}(0, \infty)~
\] 
and 
\[ t~ \in~  g_{k-1}^{-1}(-\infty, 0)~ =~ \bigcup_{k' <k}  g_{k'}^{-1}(-\infty, 0).
\]
It follows that if 
\[ t'~ \in~ W~ :=~  U~ \cap~ g_{k+1}^{-1}(0, \infty)~ \cap~  g_{k-1}^{-1}(-\infty, 0), \]
$t' \neq t$, and $k' \in \N$, then $g_{k'}(t) \neq 0$.
Since $W$ is open, we have the claim. 
\end{proof}

\begin{proof}[Proof of Lemma \ref{BigPik}]
By Lemma \ref{LemLiminf}, there exist $C$ and $t_1>0$ such that 
if $t\leq t_1$, $z \in \Dcal$, and $u$ is an eigenfunction with eigenvalue $E_t$,
then 
\begin{equation} \label{tildeestimate}
 \left| a^{\mu}_t\left( \tilde{w}_{\mu},z\right)~ -~  E_t \cdot 
\left\langle \tilde{w}_{\mu},z \right\rangle_{\sigma} \right|~ 
\leq~  C \cdot t \cdot \|w\| \cdot \|z\|_{\sigma}
\end{equation} 
where $w= P_{a_t}^I(u)$ and $\tilde{w}\otimes \vphi_\mu\,=\, \Pi_{\mu} w$.

Since $a_t^{\mu}$ is a real-analytic family of type (a) in the 
sense of \cite{Kato}, for each $k \in \N$, there exists 
a real-analytic function $\lambda_k:(0, t_1] \rightarrow \R$ 
so that for each $t \in (0, t_1]$, 
we have $\spec(a_t^{\mu})=\{\lambda_k(t)~ |~ k \in \N\}$.
Since each eigenspace of $a_t^{\mu}$ is 1-dimensional, we may assume
that $k > k'$ implies $\lambda_k(t) > \lambda_{k'}(t)$ for all $t \in (0, t_1]$.

By Theorem \ref{Superseparation}, there exists $t_0 \in(0, t_1]$ such that if 
$t<t_0$, then $k \neq k'$, then 
\begin{equation} \label{Super}
    \left|  \lambda_k(t)~ -~ \lambda_{k'}(t) \right|~ >~ 4 C \cdot t. 
\end{equation}
For each $k \in \N$ and $t \in (0, t_0)$, define 
\[  g_k^{\pm}(t)~ =~  \lambda_{k}(t)~ -~ E_t~ \pm~ 2C \cdot t. \]
Thus, by Lemma \ref{DistSpec}, the set 
\[ Z~ =~ \bigcup_{k \in \N}~ \left( (g_k^+)^{-1}\{0\}~ \bigcup~ (g_{k}^-)^{-1}\{0\} \right) \]
is discrete in $(0, t_0]$.  On each component $J$ of the complement
$(0,t_0] \setminus Z$, we have either 
\begin{itemize}
\item for all $t \in J$,  we have 
  $ {\rm dist}\left( E_t, \spec(a_t^{\mu}) \right) \geq 2 C \cdot t$, or
\item for all $t \in J$, we have 
$  {\rm dist}\left( E_t, \spec(a_t^{\mu}) \right) < 2 C \cdot t$.
\end{itemize}
It suffices to construct in each of these cases a continuous map $t \mapsto u_t$
from $J$ to $V_t$ that satisfies (\ref{Pihalf}). Without loss of generality,
each interval $J$ is precompact in $(0, t_0]$, for otherwise we may, for example,
add the discrete set $\{1/n~ |~ n \in \N\}$ to $Z$.

We consider the first case. Let $u_t$ be a real-analytic eigenfunction
branch of $q_t$ associated to $E_t$. By estimate (\ref{tildeestimate}),
we may apply Lemma \ref{resolvent} with $\epsilon= C \cdot t \cdot \|w_t\|$
and find that 
\begin{equation} \label{half}
 \left\| \tilde{w}_t \right\|_\sigma~ 
  \leq~  \frac{1}{2} \cdot \|w_t\|.
\end{equation}    
Since $\|\Pi_{\mu}w\|= \|\tilde{w}_{\mu}\|_{\sigma}$, the desired (\ref{Pihalf}) 
follows.

We consider the second case. By (\ref{Super}) and since $J \subset (0, t_0)$
there exists a unique $k$ such that if $t \in J$, then
\begin{equation} \label{Matching}
  |E_t- \lambda_{k}(t)|~ <~ 2C \cdot t.
\end{equation}
Let $t \mapsto \tilde{v}_{t}$ be the unique eigenfunction branch of $a_t^{\mu}$ 
associated to the eigenvalue branch $\lambda_{k}$. Since $\dim(V_t)>1$ and 
$V_t$ is an analytic family of vector spaces, there
exist analytic eigenfunction branches $x_t, x_t' \in V_{t}$ so that
for each $t$, the eigenvectors $x_t$ and $x_t'$ are independent.

The function  $t \mapsto \langle x_t, \tilde{v}_t \otimes \phi_{\mu}\rangle$
is real-analytic, and thus it vanishes on at most a finite subset $Z_J \subset J$. 
Away from $Z_J,$ set 
\[   c(t)~ =~ - \frac{\langle x_t',  \tilde{v}_t \otimes \phi_{\mu}\rangle}{
                   \langle x_t,  \tilde{v}_t \otimes \phi_{\mu}\rangle}.
\]
Then $u_t= c(t) \cdot  x_t + x_t'$ depends real-analytically on $t$
and satisfies
\[  \langle  u_t, \tilde{v}_t \otimes \phi_{\mu}\rangle~ =~ 0. \]

For each  $t \in J \setminus Z_J$, let $r_t$ denote the 
restriction of the quadratic form $a_t^{\mu}$ to the orthogonal 
complement of $ \tilde{v}_t \otimes \phi_{\mu}$ in $\Dcal \bigotimes \dom(b)$.
Let $w_t = P_{a_t}^I(u_t)$ and let $\tilde{w}_{\mu,t} \in \Dcal$
such that $\Pi_{\mu} w_{t}= \tilde{w}_{\mu,t} \otimes \phi_{\mu}$. 
From (\ref{tildeestimate}), we have 
\[ \left| r_t\left( \tilde{w}_{\mu,t},z\right)~ -~  E_t \cdot 
\left\langle \tilde{w}_{\mu,t},z \right\rangle_{\sigma} \right|~ 
\leq~  C \cdot t \cdot \|w_t\| \cdot \|z\|_{\sigma}.
\]
It follows from (\ref{Super}) that 
${\rm dist}(E_t,\spec(r_t)) \geq 2 C \cdot t$.
Hence Lemma \ref{resolvent} applies with $\epsilon= 2  C \cdot t \cdot \|w\|$
to give (\ref{Pihalf}).

Therefore, on the complement of $Z \cup \bigcup_J Z_J$, we have constructed 
a real-analytic function $t \mapsto V_t$ so that  (\ref{Pihalf}) holds.
\end{proof}


\part{Applications}


\section{Stretching along an axis} \label{AdiabaticSection}

In this section, we consider a family of quadratic forms $q_t$
obtained by `stretching' certain domains in Euclidean space $\R^{n+1}$
that fiber over an interval. 
To be precise, let $I=[0,c]$ be an interval, let $Y \subset \R^n$
be a compact domain with Lipschitz boundary, and let 
$\rho: [0,c] \rightarrow \R$ be a smooth nonnegative function.
For $t>0$, define $\phi_t: I \times Y \rightarrow \R^{n+1}$ by
\begin{equation} \label{Phit}
   \phi_t(x,y)~ = (x/t, \rho(x) \cdot y). 
\end{equation}
We will consider the Dirichlet Laplacian associated
to the domain $\Omega_t=\phi_t(I\times Y)$. 

\begin{eg}[Triangles and simplices]
Let $Y=[0,a]$ and $\rho(x)=x$. Then $\Omega_t$ is the 
right triangle with vertices $(0,0)$, $(c/t,0)$, $(c/t, c)$. 
More generally, if $\rho(x)=x$ and $Y$ is a $n$-simplex, then 
$\Omega_t$ is a $n+1$-simplex.   
\end{eg}

\begin{thm}\label{ThmSimplDbc} 
If $\rho:[0,a] \rightarrow \R$ is smooth, $\rho(0)=0$, $\rho'>0$,
\[  \lim_{\epsilon \rightarrow 0}~  \int_{\epsilon}^c \frac{dx}{\rho(x)}~ =~ \infty,
\]
and each eigenspace of the Dirichlet Laplacian acting on $L^2(Y)$ is 1-dimensional, 
then for all but countably many $t$, each eigenspace of the Dirichlet
Laplacian acting on $L^2(\Omega_t)$ is 1-dimensional. 
\end{thm}

\begin{proof}
In order to apply Theorem \ref{SimplicityTheorem}, we 
make the following change of variables. Define $\psi: (0,c] \rightarrow [0, \infty)$
by 
\[  \psi(x)~ =~ \int_{x}^c  \frac{dx}{\rho(x)}. \]
By hypothesis, $\psi$ is an orientation reversing homeomorphism. 
Define $\Phi_t: C^{\infty}([0, \infty)\times Y) \rightarrow C^{\infty}(\Omega_t)$
by 
\[ \Phi_t(u)~ =~  
  \left(\rho^{\frac{n-1}{2}} \cdot u 
      \circ \left(\psi \times {\rm Id} \right) \right) \circ \phi_t. \]
where $\phi_t$ is defined by (\ref{Phit}).
We will use $\Phi_t$ to pull-back the $L^2$ inner product and the Dirichlet energy form. 

First note that the Jacobian matrix of $\phi_t$ is 
\begin{equation} \label{Jacobian}
  J \phi~ =~ \left( \begin{array}{cc}
    1/t & 0 \\ \partial_x\rho \cdot y &
   \rho \cdot {\rm Id}  \end{array} \right)
\end{equation}
where ${\rm Id}$ is the $n \times n$ identity matrix, 
and hence the Jacobian determinant $|J \phi_t|$ equals $t^{-1} \cdot \rho^n$.
The Jacobian determinant of $\psi \times {\rm Id}$ is $\rho^{-1}$. 
It follows that
\begin{equation} \label{StretchInnerEucl} 
\int_{\Omega_t} \left(\Phi_t(u) \cdot  \Phi_t(v)\right)~ dV~
   =~ \frac{1}{t} \int_0^{\infty} \int_Y~  u \cdot v~ \sigma(x)~  dx~ dy
\end{equation}
where $\sigma= \rho^2 \circ \psi^{-1}$ and 
where $dy$ denotes Lebesgue measure on $Y \subset \R^n$. 
In order to have an inner product
that does not depend on $t$, we rescale by $t$. Define
\[ \langle u , v \rangle~ =~  \int_0^{\infty} \int_Y~  u \cdot v~ \sigma(x)~  dx~ dy.
\]
Define a family of quadratic forms on $C^{\infty}([0, \infty)\times Y)$ by 
\[  q_t(u)~ =~  
   t \cdot \int_{\Omega_t} \left| \nabla\left( \Phi_t(u) \right)\right|^2~ dx~ dy
\]
Note that the $\Phi_t$ defines an isomorphism from each eigenspace of 
$q_t$ with respect to $\langle \cdot ,\cdot \rangle$ to the eigenspaces of 
the Dirichlet energy form on $\Omega_t$ with respect to the $L^2$-inner
product on $\Omega_t$. In particular, it suffices to show that each 
eigenspace of $q_t$ with respect to $\langle \cdot ,\cdot \rangle$
is 1-dimensional.

Define 
\[ a_t(u)~ =~  \int_0^{\infty} \int_Y \left(
  t^2 \cdot  \left|\partial_x u\right|^2~ +~ \left| \nabla_y u\right|^2 \right)~ dx~ dy. \]
By Theorem \ref{SimplicityTheorem}, it suffices to show that $q_t$ 
is asymptotic to $a_t$ at first order. 

Let $\tau= \rho' \circ \psi^{-1}$.
A straightforward calculation of moderate length
shows that 
\begin{eqnarray*}
   q_t(u,v)-a_t(u,v)
    &=& t \cdot \left( I_1(u,v)~ +~ I_2(u,v)~ +~ I_3(u,v)~ +~ I_4(u,v)~ +~ I_5(u,v)~ \right.  \\
   & & \left.  +~ I_3(v,u)~ +~ I_4(v,u)~ +~ I_5(v,u)\right)  
\end{eqnarray*}
where
\begin{eqnarray*}
 I_1(u,v)~ 
     &=& t \cdot \frac{(n-1)^2}{4}  \int_0^{\infty} \int_Y  \tau^2 \cdot u \cdot v~ dx~ dy, \\
 I_2(u, v)~ &=& 
     t \int_0^{\infty} \int_Y  \tau^2 \cdot \left( y \cdot \nabla_y u \right) \cdot
                     \left( y \cdot \nabla_y v \right)~ dx~ dy, \\
 I_3(u,v)~ &=& 
  t \cdot \frac{n-1}{2} \int_0^{\infty} \int_Y  
               \tau^2 \cdot u \cdot \left(y \cdot \nabla_y v\right)~ dx~ dy,  \\
 I_4(u,v)~ &=&  t  \int_0^{\infty} \int_Y \tau \cdot
   \partial_x u \cdot \left(y \cdot \nabla_y v\right)~  dx~ dy. \\
 I_5(u,v)~ &=& 
     t \cdot \frac{n-1}{2}
      \cdot  \int_0^{\infty} \int_Y  \tau \cdot u \cdot \partial_x v~  dx~ dy, 
\end{eqnarray*}
By the triangle inequality, it 
suffices to show that for each $k=1, \ldots, 5$, 
there exists a constant $C_k$ such that 
$|I_k(u,v)| \leq C_k \cdot a_t(u)^{\frac{1}{2}} \cdot  a_t(v)^{\frac{1}{2}}$
for $t<1$. 

First note that by assumption $|\rho'|$---and hence $|\tau|$---is 
bounded by a constant $C$.  Second, note that if $\lambda_0>0$ is the 
smallest eigenvalue of the Dirichlet Laplacian on $L^2(Y)$, 
then for each $u \in C^{\infty}([0, \infty) \times Y)$ we have 
\begin{equation} \label{Poincare}
    \int_0^{\infty} \int_Y~  u^2~ dx~ dy~
    \leq~ \frac{1}{\lambda_0}~   \int_0^{\infty} \int_Y~  
    \left| \nabla_y u \right|^2~ dx~ dy.
\end{equation}

If $n=1$, then $|I_1(u,v)|$ is trivial. Otherwise,
apply the Cauchy-Schwarz inequality
and estimate (\ref{Poincare}). More precisely
\begin{eqnarray*}
 \frac{4}{C^2(n-1)^2} \cdot 
   |I_1(u,v)| &  \leq & t \cdot \left(\int_0^{\infty} \int_Y~  u^2~ dx~ dy \right)^{\und}
     \cdot  \left(\int_0^{\infty} \int_Y~  v^2~ dx~ dy \right)^{\und} \\
     &\leq& \frac{t}{\lambda_0} 
   \cdot \left(\int_0^{\infty} \int_Y~ \left|\nabla_y u \right|^2~ dx~ dy \right)^{\und}
     \cdot  \left(\int_0^{\infty} \int_Y~  \left|\nabla_y v \right|^2~ dx~ dy \right)^{\und} \\
    &\leq& \frac{t}{\lambda_0} \cdot a_t(u)^{\und} \cdot a_t(v)^{\und}.  
\end{eqnarray*}   

To bound $|I_2(u,v)|$, note that $|y \cdot \nabla_yu|^2\leq |y|^2 \cdot |\nabla_yu|^2$
and that $|y|^2$ is bounded since $Y$ compact. The desired bound of $|I_2(u,v)|$
then follows from an application of the Cauchy-Schwarz inequality.

If $n=1$, then $|I_3(u,v)|$ is trivial. Otherwise, we apply the 
Cauchy-Schwarz inequality and estimate (\ref{Poincare}) as in the bound
of $|I_1(u,v)|$. 

To bound $|I_4(x,y)|$ we apply Cauchy-Schwarz as follows 
\[
 \int  
   \left|t \cdot  \partial_x u\right| \cdot \left|y \cdot \nabla_y v\right|~
 \leq~  
 \left(\int
   \left|t \cdot  \partial_x u\right|^2 \right)^{\und} 
\left(\int \left|y \cdot \nabla_y v\right|^2 \right)^{\und} 
\]
It then follows that
\[  |I_4(u,v)|~ \leq~ C \cdot a_t(u)^{\und}\cdot a_t(u)^{\und}. \]

To bound $|I_5(u,v)|$ apply Cauchy-Schwarz and argue in a fashion
similar to the above. 
\end{proof}

\subsection{Changing the boundary condition}
Theorem \ref{ThmSimplDbc} extends to more general boundary condition that we describe here. 
Inspecting the proof, the only thing we have used from the Laplace operator on $Y$ is that 
it satisfies the Poincar\'e inequality \ref{Poincare}. This 
fact is true for any mixed Dirichlet-Neumann boundary condition except Neumann on all faces. 

As a consequence we may take on the faces of $\Omega_t$ of the form $I\times \partial Y$ 
any kind of boundary condition except full Neumann. 

On the face $\{1\} \times Y$ we may take Dirichlet or Neumann as we want since we have allowed 
Dirichlet or Neumann at $0$ for the one-dimensional model operators $a_t^\mu$. 


\section{Domains in the hyperbolic plane with a cusp}

\label{IdealSection}

Recall that the hyperbolic metric on the upper half-plane $\R \times \R^+$
is defined by $(dx^2~ + dy^2)/y^2$. The associated Riemannian measure
is given by  $d \mu= y^{-2} dx~ dy$ and the gradient is given by 
$\nabla f= y^2 ( \partial_x f \cdot \partial_x +  \partial_y f \cdot \partial_y)$.

Let $h: (-\eta, \eta) \rightarrow \R$ be a positive real-analytic function such that 
$h'(0)=0$. For each $t < \eta$, define $\Omega_t$ by 
\[  \Omega_t~ =~ \left\{(x,y) \in \R \times \R^+~ |~
        -t \leq x \leq t~ \mbox{ and } y \geq h(x) \right\}.
\]
The domain $\Omega_t$ is unbounded but has finite hyperbolic area. 
It is known that the hyperbolic Dirichlet Laplacian acting 
on $L^2(\Omega_t, d\mu)$ is compactly resolved and hence has discrete spectrum
(see e.g. \cite{LaxPhillips}).\footnote{The Neumann Laplacian is not compactly resolved,
and in fact, has essential spectrum.}

\begin{eg}
Let $h:(-1, 1) \rightarrow \R$ be defined by $h(x)=\sqrt{1-x^2}$.
For each $t< 1$, the domain $\Omega_t$ is a hyperbolic triangle with one 
ideal vertex. In particular, $\Omega_{1/2}$ 
is a fundamental domain for the modular group $SL(2, {\mathbb Z})$
acting on $\R \times \R^+ \subset {\mathbb C}$ as linear fractional transformations. 
\end{eg}

\begin{thm} \label{HyperbolicSimple}
For all but countably many $t$, each eigenspace of the Dirichlet Laplacian 
acting on $L^2(\Omega_t, d\mu)$ is 1-dimensional.
\end{thm}

The remainder of this section is devoted to the proof of
Theorem \ref{HyperbolicSimple}.

The spectrum of the hyperbolic Laplacian on $\Omega_t$ 
coincides with the spectrum of the Dirichlet energy form  
\begin{equation} \label{HyperbolicQuadratic}
 \Ecal(u)~ =~  \int_{\Omega_t}
  \left( |\partial_x u|^2~ +~ |\partial_y u|^2 \right)~ dx~ dy,
\end{equation}
with respect to the inner product
\begin{equation}  \label{HyperbolicProduct}
  \langle u, v \rangle_{\mu}~ =~  \int_{\Omega_t} u \cdot v~ \frac{dx~ dy}{y^2}.
\end{equation}

In order to study the variational behavior of the eigenvalues, we
first adjust the domains by constructing a family of 
diffeomorphisms $\phi_t$ from the fixed set 
${\mathcal U}=[-1,1]\times [h(0), \infty[$ onto $\Omega_t$. 
In particular, define 
\[ \phi_t  \left(  \begin{array}{c}  a  \\ b \end{array} \right)~
  =~  \left(  \begin{array}{c}  t \cdot a \\ b~ +~ h(t \cdot a)~ -~ h(0)\end{array} \right). \]
For each  $u\in \Cc^\infty({\mathcal U})$, we define 
\[ \tilde u~ =~ \psi \cdot u\circ \phi_t^{-1} \]
where
\[  \psi(x,y)~ =~  \frac{y}{y-h(x)+h(0)}. \]
Since $\phi_t$ is a smooth diffeomorphism from ${\mathcal U}$ onto $\Omega_t$ 
and $\psi$ is smooth on $\Omega_t$, the mapping $u\mapsto \tilde{u}$ is 
a bijection from $\Cc^\infty({\mathcal U})$ onto $\Cc^\infty(\Omega_t).$ 

Since the Jacobian of $\phi_t$ is 
\begin{equation} \label{JacobianHyperbolic}
  J (\phi_t)
\left( \begin{array}{c} a \\ b \end{array} \right)~
 =~ \left( \begin{array}{cc}
    t & 0 \\ t \cdot h'(t \cdot a)  &
  1  \end{array} \right)
\end{equation}
and $\psi \circ \phi_t= (y\circ \phi_t)/b$, we find that, for 
any smooth $u$ and $v$ compactly supported in ${\mathcal U}$, 
\begin{equation} \label{StretchInnerHyp} 
~~
t^{-1} \int_{\Omega_t} {\tilde u}\cdot \tilde{v}\,\frac{dxdy}{y^2}\,=\,
\int_{{\mathcal U}} u \cdot v~ \frac{da~ db}{b^2}. 
\end{equation}
In particular, the mapping $u \mapsto \tilde{u}$ extends to an isometry of 
$\Hcal:=L^2({\mathcal U},~ da\cdot db/b^2)$ onto $L^2(\Omega_t, t^{-1}d\mu).$ 

We now pull-back the Dirichlet energy form from $\Omega_t$ 
to ${\mathcal U}$. 
In particular, we define $q_t: \Cc^{\infty}({\mathcal U}) \rightarrow \R$ by
\[  q_t(u)~ =~ t \cdot \Ecal (\tilde{u}).
\]
The form extends to a closed densely defined form on ${\mathcal H}$.
By construction, $\lambda$ belongs to the spectrum of $q_t$ if and only if 
$t^{-2} \cdot \lambda$ belongs to the Laplace spectrum of the 
hyperbolic triangle $\Omega_t$. Because $h$ is real-analytic, $t \mapsto \phi_t$
is a real-analytic family of bi-Lipschitz homeomorphisms. It follows that 
$q_t$ is a real-analytic family of quadratic forms of type (a) in the sense
of Kato \cite{Kato}.

On $\Cc^{\infty}({\mathcal U}),$ we also define 
\[  a_t(u)~ =~
\int_{{\mathcal U}} \left( t^2 \cdot |\partial_b u|^2~ +~  
    |\partial_a u|^2 \right)~ da~ db. 
\]

Theorem \ref{HyperbolicSimple} follows from 
Theorem \ref{SimplicityTheorem} and the following proposition.

\begin{prop}
$q_t$ is asymptotic to $a_t$ at first order. 
\end{prop}

\begin{proof}
Let $\bar{u}\,=\, (\psi\circ \phi_t) \cdot u$. One computes that 
\begin{eqnarray*}
\left( \partial_y \tilde{u} \right)\circ \phi_t &=& \partial_b \bar{u} 
\\
\left( \partial_x \tilde{u} \right)\circ \phi_t
    &=& \frac{1}{t} \cdot \partial_a \bar{u}~ 
       -~ h'(ta)  \cdot \partial_b \bar{u}.
\end{eqnarray*}
Thus, by making a change of variables in the integral that defines $\Ecal$, 
we find that
\begin{equation}  \label{qExpress}
q_t(u)~ =~ \int_{{\mathcal U}} \left| \partial_a \bar{u}\right |^2 
-2t \cdot h'(ta) \cdot \partial_a \bar{u} \cdot \partial_b \bar{u}~ +~
 t^2  \cdot (1+h'(ta)^2)\left| \partial_b \bar{u}\right |^2~ da~ db
\end{equation}
where $\bar{u}= \bar{\psi} \cdot u$. To aid in computation we define
a weighted gradient 
\[  \bar{\nabla}w~ =~ \left[ \partial_a w,~ t \cdot \partial_b w \right].
\]
and we define
\[     A_t~ =~ \left( \begin{array}{cc}  1 & - h'(t\cdot a) \\ 
                          - h'(t \cdot a) & 1+h'(t \cdot a)^2\end{array} \right) \]
Thus, (\ref{qExpress}) becomes
\[  q_t(u,v)~ =~ \int_{\Ucal}
  \bar{\nabla}\bar{u} \cdot A_t \cdot \bar{\nabla}\bar{v}~ da ~ db
\]
and 
\[  a_t(u,v)~ =~  \int_{\Ucal}  \bar{\nabla}u \cdot \bar{\nabla}v~ da ~ db
\]
Letting $\bar{\psi}= \psi \circ \phi$, we have
\[  \bar{\nabla}\bar{w}~ =~ \bar{\psi} \cdot \bar{\nabla} w~ 
     +~ w \cdot \bar{\nabla} \psi. 
\]  
and hence $q_t(u,v)-a_t(u,v)$ is the sum of four terms 
\begin{eqnarray}
\label{leading} & & 
\int_{\Ucal} \bar{\nabla}u \cdot(\bar{\psi}^2 \cdot A_t-I) \cdot \bar{\nabla}v~ da~ db~, \\
& & \label{mixed}
\int_{\Ucal} 
    \bar{\psi} \cdot v \cdot (\bar{\nabla} \bar{\psi} \cdot A_t \cdot \bar{\nabla} u)~
 da~ db~, \\
& & \label{mixed2}
\int_{\Ucal} 
   \bar{\psi} \cdot u \cdot (\bar{\nabla} \bar{\psi} \cdot A_t \cdot \bar{\nabla} v)~ 
 da~ db~, \\
& & \label{uv}
\int_{\Ucal} 
  (\nabla \bar{\psi} \cdot A_t \cdot \bar{\nabla} \bar{\psi}) \cdot u \cdot v~
  da~ db
\end{eqnarray}
where $I$ denotes the $2 \times 2$ identity matrix.
To finish the proof, it suffices to show that each of these
terms is bounded by $O(t) \cdot  a_t(u)^{\frac{1}{2}} \cdot a_t(v)^{\frac{1}{2}}$
where $O(t)$ represents a function that is bounded by a constant times
$t$ for $t$ small.

In order to estimate these terms, we use elementary estimates 
of $h(t \cdot a)$, $h'(t \cdot a)$,  $\bar{\psi}$, and   $\bar{\nabla} \bar{\psi}$.
In particular, since $h'(0)=0$ we have that $|h(t \cdot a) -h(0)| = O(t)$
and $|h'(t \cdot a)| =O(t)$ uniformly for $a \in [-1,1]$. 
Thus, since 
\[  \bar{\psi}(a,b)~ =~ 1~ - \frac{h(t \cdot a) - h(0)}{b}  \] 
we find that $|\bar{\psi}^2(a,b) -1|= O(t)$ and $|\nabla \bar{\psi}|=O(t)$ 
uniformly for $(a, b) \in {\mathcal U}$. 

To bound (\ref{leading}), note that
\[ {\rm tr}(\bar{\psi}^2 \cdot A -I)~ =~ 2(\bar{\psi}^2-1)~ +~ \bar{\psi}^2 \cdot h'(t \cdot a)^2 \]
and 
\[ {\rm det}(\bar{\psi}^2 \cdot A -I)~ =~  (\bar{\psi}^2 -1)^2~  -~ h'(t \cdot a)^2.  \] 
Hence ${\rm tr}(\bar{\psi}^2 \cdot A -I)=O(t)$
and ${\rm det}(\bar{\psi}^2 \cdot A -I)=O(t^2)$.
It follows that the eigenvalues of $\bar{\psi}^2 \cdot A -I$
are $O(t)$. Therefore, 
\[  \int_{\Ucal} \bar{\nabla}u
 \cdot(\bar{\psi}^2 \cdot A_t-I) \cdot \bar{\nabla}v~ da~ db~
=~ O(t) \cdot \int_{\Ucal} \bar{\nabla}u \cdot \bar{\nabla}v~ da~ db.
\]

To estimate (\ref{mixed}) we first note that the 
eigenvalues of $A_t$ are $O(1)$. 
Then we apply Cauchy-Schwarz 
\[  | \bar{\nabla} \bar{\psi} \cdot \bar{\nabla} u |~ 
  \leq | \bar{\nabla} \bar{\psi}| \cdot | \bar{\nabla} u | 
\]
and then the elementary estimate on $|\bar{\nabla} \bar{\psi}|$ to find that
\[\int_{\Ucal} 
    |\bar{\psi}| \cdot |v| \cdot |\bar{\nabla} \bar{\psi} \cdot \bar{\nabla} u|~
 da~ db~ \leq~
O(t) \int_{\Ucal} v \cdot |\bar{\nabla} u|~
 da~ db.
\]
Cauchy-Schwarz applied to the latter integral gives
\[ \int_{\Ucal} |v| \cdot |\bar{\nabla} u|~
 da~ db~ \leq~
  \left(\int_{\Ucal} |v|^2~ da~ db\right)^{\frac{1}{2}} 
\cdot  \left(\int_{\Ucal} |\bar{\nabla} u|^2~ da~ db\right)^{\frac{1}{2}}. 
\] 
From a Poincar\'e inquality---Lemma \ref{Poinc} below---we find that 
\[  \int_{\Ucal} |v|^2~ da~ db~
\leq~ \pi^2 \int_{\Ucal} |\bar{\nabla} v|^2~ da~ db. 
\]
In sum we find that the expression in (\ref{mixed}) is bounded
by $O(t) \cdot a_t(u)^{\frac{1}{2}} \cdot a_t(v)^{\frac{1}{2}}$. 
Switching the r\^oles of $u$ and $v$, we obtain the same
bound for the expression in (\ref{mixed2}).

To estimate (\ref{uv}) we use the fact that 
the norm of the eigenvalues of $A_t$ are $O(1)$
and the fact that $|\bar{\nabla} \bar{\psi}|^2=O(t^2)$
to find that 
\[  \int_{\Ucal} 
  |\nabla \bar{\psi} \cdot A_t \cdot \bar{\nabla} \bar{\psi}| \cdot |u| \cdot |v|~
  da~ db~ =~
  O(t) \cdot  \int_{\Ucal} |u| \cdot |v|~
  da~ db~
\] 
By applying Cauchy-Schwarz and the following Poincar\'e inequality we obtain the 
claim. 
\end{proof}

\begin{lem}\label{Poinc}
Any $u\in \Cc^{\infty}({\mathcal U})$ satisfies : 
$$
\int_{{\mathcal U}} \left | u\right |^2 \, da~ db~
 \leq~ \pi^2 \int_{U} \left| \partial_a u\right |^2 \,da~ db.
$$ 
\end{lem}
\begin{proof}
We decompose $u\,=\,\sum_k u_k(b)\sin(k\pi a).$ Then we have 
\begin{equation*}
\begin{array}{lcl}
\displaystyle \int_{U} \left| \partial_a u\right |^2 &=&\sum_k k^2\pi^2 \int_{h(0)}^\infty u_k(b)^2~ db \\
								& \geq & \pi^2  \sum_k \int_{h(0)}^\infty u_k(b)^2 db \,=\,\pi^2 \displaystyle \int_{{\mathcal U}} \left | u\right |^2 \, da~ db.
\end{array}
\end{equation*}
\end{proof}


\part{Appendices and references}


\appendix

\section{Solutions to the Airy equation} \label{AiryAppendix}

Here we consider solutions to Airy's differential equation
\begin{equation} \label{AiryEquation}
  A''(u)~ =~ u\cdot A(u) 
\end{equation}
for $u \in \R$. It is well-known that
there exist unique solutions $A_+$ and $A_-$ that satisfy\footnote{
The functions $\pi^{-\frac{1}{2}} \cdot A_{\pm}$ 
are the classical Airy functions ${\rm Ai}$ and ${\rm Bi}$. 
See, for example, \cite{Olver}  Chapter 11.} 
\begin{equation} \label{PositiveAsymptotics}
 A_{\pm}(u)~ =~ \frac{ u^{-\frac{1}{4}}}{2^{\frac{1\pm 1}{2}}} \cdot
    \exp \left( \pm \frac{2}{3} \cdot u^{\frac{3}{2}} \right) 
      \left(  1~ +~ O\left(u^{-\frac{3}{2}} \right)  \right)~ 
\end{equation} 
and 
\begin{equation} \label{NegativeAsymptotics} 
   A_{\pm}(-u)~ =~  u^{-\frac{1}{4}} \left(
   \cos \left( \frac{2}{3} \cdot u^{\frac{3}{2}}~ \mp~ \frac{\pi}{4}  \right) 
  +~ O\left( u^{-\frac{3}{2}} \right)  \right)
\end{equation}
where $u^{\frac{3}{2}} \cdot O(u^{-\frac{3}{2}})$ is bounded on $[1, \infty)$. 

Let $W$ denote the Wronskian of $\{A_+,A_-\}$. 
Define $K: \R \times \R \rightarrow \R$ by
\[  K(u,v)~ =~  W^{-1} \cdot 
  \left\{  \begin{array}{cc} A_+(u) \cdot A_-(v)    & \mbox{ if }  v  \geq u \geq 0 \mbox{ or }
                                                         v \geq 0 \geq u  \\
            A_-(u) \cdot A_+(v)    &  \mbox{ if }  u \geq  v \geq 0  \\
    A_+(u) \cdot A_-(v) - A_-(u) \cdot A_+(v)   & \mbox{ if }  u  \leq v \leq 0  \\
%
%
                         0   &  \mbox{ otherwise } 
   \end{array}  \right.
\]

\begin{lem} \label{AiryKernel}
Let $\infty <\alpha \leq 0 \leq  \beta \leq \infty$. 
For each locally integrable function $g:[\alpha, \beta] \rightarrow \R$
of at most polynomial growth, we have  
\begin{equation} \label{Kernel}
 (\partial^2_u - u) \int_{\alpha}^{\beta} K(u,v) \cdot g(v)~ dv~ =~ g(u), 
\end{equation}
\end{lem}
\begin{proof}
Note that the Wronskian $W$ is constant and hence 
by, for example, variation of parameters we find that the function 
\[ P(u)~ =~  W^{-1} \cdot  A_{+}(u)  \int_u^{\beta}A_{-}(v) \cdot g(v)~ dv~ 
    +~  W^{-1} \cdot   A_{-}(u)  \int_0^{u} A_+(v) \cdot g(v)~ dv~
\]
is a solution to $P''(u) - u \cdot P(u)= g(u)$.
Hence $K$ satisfies (\ref{Kernel}). 
\end{proof}

\begin{lem} \label{AiryKernelEstimates}
There exists a constant $C_{{\rm Airy}}$ so that
\begin{equation} 
|K(u,v)|~ \leq~ C_{{\rm Airy}} \cdot \left\{ \begin{array}{cc}
    {\rm exp}\left(-\left|v-u \right|\right) &  \mbox{ if }  u,v \geq 0  \\
   \left|u \cdot v \right|^{-\frac{1}{4}} &  \mbox{ if }  u \leq v \leq 0  \\
    |u|^{-\frac{1}{4}} \cdot {\rm exp}\left(-v\right) &  \mbox{ if }  u \leq 0 \leq v   \\
    \end{array}  \right.
\end{equation}
and
\begin{equation} 
 \left|\partial_uK(u,v) \right|~ \leq~ C_{{\rm Airy}} \cdot \left\{ \begin{array}{cc}
    {\rm exp}\left(-\left|v-u \right|\right) &  \mbox{ if }  u,v \geq 0  \\
   \left|u \cdot v \right|^{\frac{1}{4}} &  \mbox{ if }  u \leq v \leq 0  \\
    |u|^{\frac{1}{4}} \cdot {\rm exp}\left(-v\right) &  \mbox{ if }  u \leq 0 \leq v   \\
    \end{array}  \right.
\end{equation}
\end{lem}

\begin{proof}
Straightforward using definition of $K$ and 
the Asymptotic of the Airy functions \cite{Olver}.
\end{proof}

\begin{lem} \label{HilbertSchmidtEstimate}
There exists a constant $C$ so that
\begin{equation} \label{HilbertSchmidt}
   \int_{-\alpha}^{\alpha} \int_{-\alpha}^{\alpha} |K(u,v)|^2~ du~ dv~
   \leq~  C \cdot \sqrt{\alpha},
\end{equation}
\end{lem}

\begin{proof}
This follows directly from Lemma \ref{AiryKernelEstimates}. \end{proof}

\begin{lem} \label{TransitionRegion}
Let $b^- < a^- < 0 < b^+ < a^+$. There exist constants $C$ and $s_0$ such that 
if $s > s_0$ and $A$ is a solution to (\ref{AiryEquation}), then 
\begin{equation} \label{NegEst} 
        \int_{s \cdot a^-}^{0} A^2~ du~ 
    \leq~  C \int_{s \cdot  a^-}^{s \cdot b^-} A^2~ du, 
\end{equation}
and 
\begin{equation} \label{PosEst}
        \int_{0}^{s \cdot b^+} A^2~ du~ 
    \leq~  C \left( s^{-\frac{1}{2}} \int_{s \cdot  b^-}^{s \cdot a^-} A^2~ du~  
       +~   \int_{s \cdot b^+}^{s \cdot  2 b_+} A^2~ du \right).
\end{equation}
The constants $C$ and $s_0$ may be chosen to depend continuously
on $a^-, b^-, a^+$, and $b^+$.
\end{lem}

\begin{proof}
Let $0 < \alpha < \beta$. By using (\ref{NegativeAsymptotics}) 
and the identity $\cos^2(\xi)= 2^{-1} \cdot (1+\cos(2\xi))$, we have
\[   \int_{-\beta}^{-\alpha} A_{\pm}^2~ du~ =~
     \frac{1}{2} \int_{\alpha}^{\beta} u^{-\frac{1}{2}}~ du~
   +~   \frac{1}{2} \int_{\alpha}^{\beta} u^{-\frac{1}{2}} \cdot \cos(2\xi)~ du~
   + \int_{\alpha}^{\beta} O \left( (1+u)^{-2}\right)~ du
\]
where $\xi= (2/3) \cdot x^{\frac{3}{2}} \mp \pi/4$.
Integration by parts gives
\[   \int_{\alpha}^{\beta} u^{-\frac{1}{2}} \cdot \cos(2\xi)~ du~
=~  \left. \frac{1}{2}  \cdot u^{-\frac{1}{2}} \cdot \sin(2\xi) \right|_{\alpha}^{\beta}~ 
+~ \frac{1}{4}  \int_{\alpha}^{\beta} u^{-\frac{3}{2}} \cdot \sin(2\xi)~ du, 
\]
and hence we have 
\begin{equation}  \label{--neg}
   \int_{-\beta}^{-\alpha} A_{\pm}^2~ du~ 
    =~  \beta^{\frac{1}{2}}- \alpha^{\frac{1}{2}}
        +~ O \left(\beta^{-\frac{1}{2}}+ \alpha^{-\frac{1}{2}} \right).
\end{equation}
Since $A_{\pm}$ is bounded on $[-1,0]$ we also have 
\begin{equation}  \label{--neg0}
   \int_{-\beta}^{0} A_{\pm}^2~ du~ 
    =~  \beta^{\frac{1}{2}}~
     +~  O(1). 
\end{equation}

Using (\ref{NegativeAsymptotics}) and an argument similar to the one above, we find that
for $0<\alpha< \beta$, we have 
\begin{equation}  \label{+-neg}
   \int_{-\beta}^{-\alpha} A_{+} \cdot A_-~ du~ 
    =~  
         O \left(\beta^{-\frac{1}{2}}+ \alpha^{-\frac{1}{2}} \right).
\end{equation}
Since $A_{\pm}$ is bounded on $[-1,0]$, it follows that
\begin{equation} \label{+-neg0}
   \int_{-\beta}^{0}  A_{+} \cdot A_{-}~ du~ 
    =~   O(1).  
\end{equation}

We now specialize to the case $\alpha=-s \cdot a^-$ and $\beta=-s \cdot b^-$.
By (\ref{--neg}) and (\ref{--neg0}), 
there exists $s_1$---depending continuously on 
$b^-< a^- <0$---such that for $s >s_1$
\begin{equation} \label{++endneg} 
 \int_{s \cdot b^-}^{s \cdot a^-} A_{\pm}^2~ du~ 
  \geq~ m
    \int_{s \cdot a^-}^{0} A_{\pm}^2~ du. 
\end{equation}
where 
\[ m~ =~ \frac{1}{2} \cdot  \left(1- \left( \frac{a^-}{b^-} \right)^{\frac{1}{2}} \right).\]
By (\ref{--neg0}) and (\ref{+-neg}), there exists a constant
$s_2$---depending continuously on $b^-, a^-<0$---such that if $s> s_2$, then 
\begin{equation} \label{+-endneg2} 
  \left|  \int_{s \cdot b^-}^{s \cdot a^-}  A_{+} \cdot A_{-}~ du \right|~
  \leq~
  \frac{m}{2} \cdot  \int_{s \cdot a^-}^{0}  A_{\pm}^2~ du. 
\end{equation}

If $A$ is a general solution to (\ref{AiryEquation}), then there exist $c_+, c_- \in \R$
such that 
\[  A~ =~  c_+ \cdot A_+~  +~ c_- \cdot A_-.\]
Using (\ref{+-endneg2}) we find that
\[  2 |c_+\cdot c_-| \cdot 
\left|  \int_{s \cdot b^-}^{s \cdot a^-}  A_{+} \cdot A_{-}~ du \right|~
\leq~ \frac{m}{2} \cdot \left( c_+^2  \int_{s \cdot a^-}^{0}  A_{+}^2~ du~ +~
 c_-^2  \int_{s \cdot a^-}^{0}  A_{-}^2~ du \right).
\]
By combining this with (\ref{++endneg}) we find that 
if $s > \max\{s_1,s_2\}$, then 
\begin{equation}  \label{PreNegEst}
    \int_{s \cdot b^-}^{s \cdot a^-} A^2~ du~ 
  \geq~  \frac{m}{4} 
    \int_{s \cdot a^-}^{0} A^2~ du. 
\end{equation}
This finishes the proof of the first estimate.

To prove the second estimate, first 
define $f(u)= \exp((2/3) \cdot u^{\frac{3}{2}})$ and let $0 < \alpha < \beta$.
By using (\ref{PositiveAsymptotics}) and integrating by parts we find that
\[   \int_{\alpha}^{\beta}  A_{+}^2~ du~ 
    =~  \left. \frac{1}{4} \cdot u^{-1} \cdot f(u) 
  \cdot \left( 1 +  O( u^{-\frac{5}{2}}) \right)
    \right|_{\alpha}^{\beta} 
\]
and, thus since $A_{\pm}$ is bounded on $[0,1]$,  
\[   \int_{0}^{\beta}  A_{+}^2~ du~ 
    =~  \frac{1}{4} \cdot \beta^{-1} \cdot f(\beta) 
 \cdot \left(1~ +~ O(\beta^{-\frac{5}{2}}) \right).
\]
It follows that there exists $s_3$ so that for $s> s_3$ 
\begin{equation} \label{++pos}
 \int_{s \cdot b^+}^{s \cdot 2b^+}  A_{+}^2~ du~ 
    \geq~ \frac{1}{2} \cdot \int_{0}^{s \cdot b^+}  A_{+}^2~ du.
\end{equation}
Equation (\ref{PositiveAsymptotics}) also implies that
\[  \int_{\alpha}^{\beta}  A_+ \cdot A_{-}~ du~ 
    =~  \beta^{\frac{1}{2}}~ -~ \alpha^{\frac{1}{2}}~ 
    +~  O \left(  \beta^{-\und} + \alpha^{-\und} \right).
\]
In particular, there exists $s_4>0$ so that if $s> s_4$,
then
\begin{equation} \label{+-pos}
  \int_{s \cdot b^+}^{s \cdot 2b^+}  A_+ \cdot A_-~ du~ 
  \geq~ 0
\end{equation}
By (\ref{PositiveAsymptotics}), the function  $A_{-}^2$
is integrable on $[0, \infty)$.  Let $I$ be the value of this 
integral. Using (\ref{--neg}) we find that 
there exists $s_5$ such that if $s> s_5$, then 
\begin{equation} \label{--pos}
  \int_{0}^{s \cdot b^+}  A_{-}^2~ du~ \leq~ 
  M \cdot  s^{-\frac{1}{2}}
    \int_{s \cdot b_-}^{s \cdot a^-}  A_{-}^2~ du
\end{equation}
where $M= 2 I/ \left((b^-)^{\und}- (a^-)^{\und}\right)$.
From (\ref{--neg}) and (\ref{+-neg}) we find that there 
exists $s_6$ such that if $s> s_6$, then 
\begin{equation} \label{halfneg}
  \left|\int_{s \cdot b^-}^{s \cdot a^-}  A_{+} \cdot A_-~ du\right|~ \leq~ 
 \frac{1}{2} \int_{s \cdot b^-}^{s \cdot a^-}  A_{\pm}^2~ du.
\end{equation}

Let  $A=c_+A_+ + c_-A_-$ be a general solution to the Airy equation.
From (\ref{++pos}) and (\ref{+-pos}) it follows that if $s> \max\{s_3,s_4\}$, then 
\begin{equation} \label{finalpos}
 c_+^2 \int_{0}^{s \cdot b^+}  A_+^2~ du~ \leq~ 2 \int_{s \cdot b^+}^{2 s \cdot b^+}  A^2~ du. 
\end{equation}
From (\ref{halfneg}) we have that if $s > s_6$, then 
\[  2 |c_+\cdot c_-| \cdot 
\left|  \int_{s \cdot b^-}^{s \cdot a^-}  A_{+} \cdot A_{-}~ du \right|~
\leq~ \frac{1}{2} \cdot \left( c_+^2  \int_{s \cdot b^-}^{s \cdot a^-}  A_{+}^2~ du~ +~
 c_-^2  \int_{s \cdot b^-}^{s \cdot a^-}  A_{-}^2~ du \right).
\] 
It follows that for $s> s_6$
\begin{equation*}
   c_-^2 \int_{s \cdot b^-}^{s \cdot a^-} A_-^2~ du~ \leq~ 
     2 \int_{s \cdot b^-}^{s \cdot a^-} A^2~ du.
\end{equation*}
Putting this together with (\ref{--pos}) gives
\begin{equation} \label{--pos2}
  c_-^2 \int_{0}^{s \cdot b^+}  A_{-}^2~ du~ \leq~ 2M \cdot s^{-\und}
              \int_{s \cdot b^-}^{s \cdot a^-}  A^2~ du. 
\end{equation}

By combining (\ref{finalpos}) and (\ref{--pos2}) we find that
\[   \frac{1}{2}  \int_{0}^{s \cdot b^+}  A^2~ du~
   \leq~ 2M \cdot s^{-\und}
              \int_{s \cdot b^-}^{s \cdot a^-}  A^2~ du~
   +~ 2 \int_{s \cdot b^+}^{2 s \cdot b^+}   A^2~ du.
\]
This completes the proof of the second estimate.
\end{proof}


\end{document}